\documentclass[12pt, reqno, a4paper,oneside]{amsart}
\usepackage{amsthm,amsmath,amssymb, amscd}
\usepackage[a4paper]{geometry}
\usepackage[colorlinks=true,allcolors=blue!80!black]{hyperref}
\usepackage{tikz, tikz-cd}
\usepackage[all]{xy}
\usepackage{subfigure}
\usepackage{enumerate}
\usepackage{bbm}

\theoremstyle{plain}
\newtheorem{thm}{Theorem}[section]
\newtheorem{lem}[thm]{Lemma}
\newtheorem{prop}[thm]{Proposition}
\newtheorem{cor}[thm]{Corollary}

\newtheorem{lemma}[thm]{Lemma}

\newtheorem{corollary}[thm]{Corollary}
\newtheorem{claim}[thm]{Claim}

\newtheorem*{thm*}{Theorem}
\newtheorem*{cor*}{Corollary}
\newtheorem{abcthm}{Theorem}

\theoremstyle{definition}
\newtheorem{defn}[thm]{Definition}
\newtheorem*{defn*}{Definition}
\newtheorem*{question*}{Question}

\newtheorem*{exam*}{Example}

\newtheorem{example}[thm]{Example}

\newtheorem{rem}[thm]{Remark}
\newtheorem*{rem*}{Remark}
\newtheorem{remark}[thm]{Remark}

\newcommand{\tl}{\mathrm{TL}}

\renewcommand{\t}{\mathbbm{1}}
\newcommand{\C}{\mathbb C}
\newcommand{\F}{\mathcal F}
\newcommand{\Z}{\mathbb Z}
\newcommand{\frakS}{\mathfrak{S}}
\newcommand{\calC}{\mathcal{C}}
\newcommand{\calI}{\mathcal I}
\newcommand{\calJ}{\mathcal J}
\newcommand{\ootimes}[1]{\otimes_{\tl_{#1}}}
\newcommand{\ul}{\underline}
\newcommand{\IE}{{}^{I}E}
\newcommand{\IIE}{{}^{II}E}
\renewcommand{\H}{\mathcal{H}}

\newcommand\Span{span}
\DeclareMathOperator{\Tor}{Tor}
\DeclareMathOperator{\Ext}{Ext}

\DeclareMathOperator{\Hom}{Hom}

\newcommand{\uq}{{U_q(\mathfrak{sl}_2)}}

\newcommand{\red}[1]{#1}

\newcommand{\bc}[2]{{\genfrac{(}{)}{0pt}{}{#1}{#2}}}
\newcommand{\qbc}[2]{{\genfrac{[}{]}{0pt}{}{#1}{#2}}_q}
\newcommand{\qi}[1]{[{#1}]_q}
\newcommand{\gbc}[3]{{\genfrac{[}{]}{0pt}{}{#1}{#2}}^G_{#3}}
\newcommand{\gi}[2]{[{#1}]_{#2}^G}
\newcommand{\JW}{\mathrm{JW}}

\renewcommand{\k}{\mathbbm{k}}

\begin{document}	
	
	\title{The homology of the Temperley-Lieb algebras}
		\author{Rachael Boyd}
    \address{DPMMS, University of Cambridge}
    \email{rachaelboyd@dpmms.cam.ac.uk}
    \urladdr{https://www.dpmms.cam.ac.uk/~rjb226/} 

    \author{Richard Hepworth}
    \address{Institute of Mathematics, University of Aberdeen}
    \email{r.hepworth@abdn.ac.uk}
    \urladdr{http://homepages.abdn.ac.uk/r.hepworth/pages/} 
    

    \begin{abstract}
        This paper studies the homology and cohomology of the Temperley-Lieb algebra~$\tl_n(a)$, interpreted as appropriate Tor and Ext groups.
        Our main result applies under the common assumption that~$a=v+v^{-1}$ for some unit~$v$ in the ground ring, and states that the homology and cohomology vanish up to and including degree~$(n-2)$.  
        To achieve this we simultaneously prove homological stability and compute the stable homology.
        We show that our vanishing range is sharp when~$n$ is even.
       
        Our methods are inspired by the tools and techniques of homological stability for families of groups. We construct and exploit a chain complex of `planar injective words' that is analogous to the complex of injective words used to prove stability for the symmetric groups.  However, in this algebraic setting we encounter a novel difficulty:~$\tl_n(a)$ is not flat over~$\tl_m(a)$ for~$m<n$, so that Shapiro's lemma is unavailable. We resolve this difficulty by constructing what we call `inductive resolutions' of the relevant modules.
        
        Vanishing results for the homology and cohomology of Temperley-Lieb algebras can also be obtained from the existence of the Jones-Wenzl projector.
        Our own vanishing results are in general far stronger than these, but in a restricted case we are able to obtain additional vanishing results via the existence of the Jones-Wenzl projector. 
        
        We believe that these results, together with the second author's work on Iwahori-Hecke algebras, are the first time the techniques of homological stability have been applied to algebras that are not group algebras.
    \end{abstract}

	\maketitle
	
\setcounter{tocdepth}{1}
\tableofcontents	

\section{Introduction}
In this work we study the homology and cohomology of the \emph{Temperley-Lieb algebras}. In particular, we simultaneously prove that the algebras satisfy homological stability, and that their stable homology vanishes.

A sequence of groups and inclusions~$G_0\to G_1\to G_2\to\cdots$ is said to satisfy \emph{homological stability} if for each degree~$d$ the induced sequence of homology groups
\[
    H_d(G_0)\to H_d(G_1)\to H_d(G_2)\to \cdots
\]
eventually consists of isomorphisms.
Homological stability can also be formulated for sequences of spaces.
There are many important examples of groups and spaces for which homological stability is known to hold, such as symmetric groups~\cite{Nakaoka}, general linear groups~\cite{Charney,Maazen,VanDerKallen}, mapping class groups of surfaces~\cite{Harer,RandalWilliamsMCG} and~$3$-manifolds~\cite{HatcherWahl}, automorphism groups of free groups~\cite{HatcherVogtmannRational, HatcherVogtmannStability}, diffeomorphism groups of high-dimensional manifolds~\cite{GalatiusRandalWilliams}, configuration spaces~\cite{Church, RandalWilliamsConfig}, Coxeter groups~\cite{HepworthCoxeter}, Artin monoids~\cite{BoydArtin}, and many more.
In almost all cases, homological stability is one of the strongest things we know about the homology of these families.
It is often coupled with computations of the stable homology~$\lim_{n\to\infty}H_\ast(G_n)$, which is equal to the homology of the~$G_n$ in the \emph{stable range} of degrees, i.e.~those degrees for which stability holds.

The homology and cohomology of a group~$G$ can be expressed in the language of homological algebra as
\[
    H_\ast(G) = \Tor_\ast^{RG}(\t,\t),
    \qquad
    H^\ast(G) = \Ext^\ast_{RG}(\t,\t),
\]
where~$R$ is the coefficient ring for homology and cohomology,~$RG$ is the group algebra of~$G$ and~$\t$ is its trivial module.
Thus the homology and cohomology of a group depend only on the {group algebra}~$RG$ and its trivial module~$\t$.  
It is therefore natural to consider the homology and cohomology of an arbitrary algebra equipped with a `trivial' module. 
Moreover, one may ask whether homological stability occurs in this wider context.

In \cite{HepworthIH} the second author proved homological stability for \emph{Iwahori-Hecke} algebras of type~$A$. 
These are deformations of the group rings of the symmetric groups that are important in representation theory, knot theory, and combinatorics.
There is a fairly standard suite of techniques used to prove homological stability, albeit with immense local variation, and the proof strategy of~\cite{HepworthIH} followed all the steps familiar from the setting of groups. 
As is typical, the hardest step was to prove that the homology of a certain (chain) complex vanishes in a large range of degrees. 

In the present paper we will prove homological stability for the \emph{Temperley-Lieb algebras}, and we will prove that the stable homology vanishes. However amongst the familiar steps in our proof lies a novel obstacle and --- to counter it --- a novel construction.
At a certain point the usual techniques fail because Shapiro's lemma cannot be applied, as we will explain below.
This is a new difficulty that never occurs in the setting of groups, but we are able to resolve it for the algebras at hand, and in fact our solution facilitates the unusually strong results that we are able to obtain.
It is not surprising that the Iwahori-Hecke case is more straightforward than the Temperley-Lieb case: Iwahori-Hecke algebras are deformations of group rings, whereas the Temperley-Lieb algebras are significantly different.

To the best of our knowledge, the present paper and~\cite{HepworthIH} are the first time the techniques of homological stability have been applied to algebras that are not group algebras, and together they serve as proof-of-concept for the export of homological stability techniques to the setting of algebras.
The moral of \cite{HepworthIH} is that the `usual' techniques of homological stability suffice, so long as the algebras involved satisfy a certain flatness condition. 
The moral of the present paper is that failure of the flatness condition can in some cases be overcome, using new ingredients and techniques, and can even lead to stronger results than in the flat scenario. 
Since the completion of this paper, we have extended our techniques to study the homology of the Brauer algebras in joint work with Patzt~\cite{BoydHepworthPatzt}.

\subsection{Temperley-Lieb algebras}
Let~$n\geqslant 0$, let~$R$ be a commutative ring, and let~$a\in R$.
The \emph{Temperley-Lieb algebra}~$\tl_n(a)$ is the~$R$-algebra with basis \red{(by which we will always mean~$R$-module basis)} given by the planar diagrams on~$n$ strands, taken up to isotopy, and with multiplication given by pasting diagrams and replacing closed loops with factors of~$a$. 
The last sentence was intentionally brief, but we hope that its meaning becomes clearer with the following illustration of two elements~$x,y\in\tl_5(a)$ 
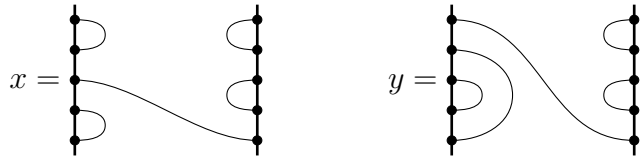
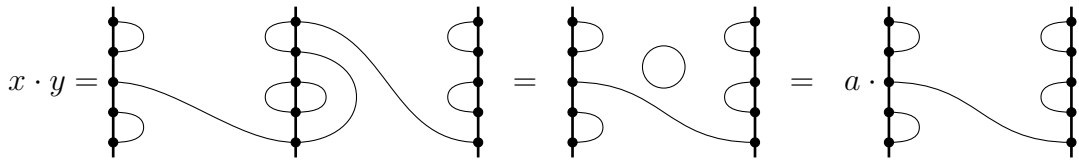
\[
    x=
    \begin{tikzpicture}[scale=0.4, baseline=(base)]
        \coordinate (base) at (0,2.75);
        \draw[line width = 1](0,0.5)--(0,5.5);
        \draw[line width = 1](6,0.5)--(6,5.5);
        \foreach \x in {1,2, 3,4,5}{
            \draw[fill=black] (0,\x) circle [radius=0.15] (6,\x) circle [radius=0.15];
        } 
        \draw (0,1) to[out=0,in=-90] (1,1.5) to[out=90,in=0] (0,2);
        \draw (0,4) to[out=0,in=-90] (1,4.5) to[out=90,in=0] (0,5);
        \draw (6,4) to[out=180,in=-90] (5,4.5) to[out=90,in=180] (6,5);
        \draw (6,2) to[out=180,in=-90] (5,2.5) to[out=90,in=180] (6,3);
        \draw (0,3) .. controls (2,3) and (4,1) .. (6,1);
    \end{tikzpicture}
    \qquad\qquad
    y=
    \begin{tikzpicture}[scale=0.4, baseline=(base)]
        \coordinate (base) at (0,2.75);
        \draw[line width = 1](0,0.5)--(0,5.5);
        \draw[line width = 1](6,0.5)--(6,5.5);
        \foreach \x in {1,2, 3,4,5}{
            \draw[fill=black] (0,\x) circle [radius=0.15] (6,\x) circle [radius=0.15];
        } 
        \draw (0,2) to[out=0,in=-90] (1,2.5) to[out=90,in=0] (0,3);
        \draw (6,4) to[out=180,in=-90] (5,4.5) to[out=90,in=180] (6,5);
        \draw (6,2) to[out=180,in=-90] (5,2.5) to[out=90,in=180] (6,3);
        \draw (0,1) to[out=0,in=-90] (2,2.5) to[out=90,in=0] (0,4);
        \draw (0,5) .. controls (3,5) and (3,1) .. (6,1);
    \end{tikzpicture}
\]
and their product~$x\cdot y$.
\[
    x\cdot y=
    \begin{tikzpicture}[scale=0.4, baseline=(base)]
        \coordinate (base) at (0,2.75);
        \draw[line width = 1](0,0.5)--(0,5.5);
        \draw[line width = 1](6,0.5)--(6,5.5);
        \draw[line width = 1](12,0.5)--(12,5.5);
        \foreach \x in {1,2, 3,4,5}{
            \draw[fill=black] (0,\x) circle [radius=0.15] (6,\x) circle [radius=0.15] (12,\x) circle [radius=0.15];
        } 
        \draw (0,1) to[out=0,in=-90] (1,1.5) to[out=90,in=0] (0,2);
        \draw (0,4) to[out=0,in=-90] (1,4.5) to[out=90,in=0] (0,5);
        \draw (6,4) to[out=180,in=-90] (5,4.5) to[out=90,in=180] (6,5);
        \draw (6,2) to[out=180,in=-90] (5,2.5) to[out=90,in=180] (6,3);
        \draw (0,3) .. controls (2,3) and (4,1) .. (6,1);
        \draw (6,2) to[out=0,in=-90] (7,2.5) to[out=90,in=0] (6,3);
        \draw (12,4) to[out=180,in=-90] (11,4.5) to[out=90,in=180] (12,5);
        \draw (12,2) to[out=180,in=-90] (11,2.5) to[out=90,in=180] (12,3);
        \draw (6,1) to[out=0,in=-90] (8,2.5) to[out=90,in=0] (6,4);
        \draw (6,5) .. controls (9,5) and (9,1) .. (12,1);
    \end{tikzpicture}
    \ \ 
    =
    \ \ 
    \begin{tikzpicture}[scale=0.4, baseline=(base)]
        \coordinate (base) at (0,2.75);
        \draw[line width = 1](0,0.5)--(0,5.5);
        \draw[line width = 1](6,0.5)--(6,5.5);
        \foreach \x in {1,2, 3,4,5}{
            \draw[fill=black] (0,\x) circle [radius=0.15] (6,\x) circle [radius=0.15];
        } 
        \draw (0,1) to[out=0,in=-90] (1,1.5) to[out=90,in=0] (0,2);
        \draw (0,4) to[out=0,in=-90] (1,4.5) to[out=90,in=0] (0,5);
        \draw (0,3) .. controls (3,3) and (3,1) .. (6,1);
        \draw (6,4) to[out=180,in=-90] (5,4.5) to[out=90,in=180] (6,5);
        \draw (6,2) to[out=180,in=-90] (5,2.5) to[out=90,in=180] (6,3);
        \draw (3,3.5) circle (0.7);
    \end{tikzpicture}
    \ \ 
    =
    \ \ 
    a\cdot
    \begin{tikzpicture}[scale=0.4, baseline=(base)]
        \coordinate (base) at (0,2.75);
        \draw[line width = 1](0,0.5)--(0,5.5);
        \draw[line width = 1](6,0.5)--(6,5.5);
        \foreach \x in {1,2, 3,4,5}{
            \draw[fill=black] (0,\x) circle [radius=0.15] (6,\x) circle [radius=0.15];
        } 
        \draw (0,1) to[out=0,in=-90] (1,1.5) to[out=90,in=0] (0,2);
        \draw (0,4) to[out=0,in=-90] (1,4.5) to[out=90,in=0] (0,5);
        \draw (0,3) .. controls (3,3) and (3,1) .. (6,1);
        \draw (6,4) to[out=180,in=-90] (5,4.5) to[out=90,in=180] (6,5);
        \draw (6,2) to[out=180,in=-90] (5,2.5) to[out=90,in=180] (6,3);
    \end{tikzpicture}
\]
The Temperley-Lieb algebras arose in theoretical physics in the 1970s \cite{TemperleyLieb}.  
They were later rediscovered by Jones in his work on von Neumann algebras \cite{JonesIndex}, and used in the first definition of the Jones polynomial \cite{JonesBull}.
Kauffman gave the above diagrammatic interpretation of the algebras in \cite{KauffmanState} and \cite{KauffmanInvariant}.

The Temperley-Lieb algebra~$\tl_n(a)$ is perhaps best studied in the case where $a=v+v^{-1}$, for~$v\in R$ a unit.
In this case, it is a quotient of the Iwahori-Hecke algebra of type~$A_{n-1}$ with parameter~$q=v^2$ (so it is closely related to the symmetric group) and it receives a homomorphism from the group algebra of the braid group on~$n$ strands.
It can also be described as the endomorphism algebra of~$V_q^{\otimes n}$, where~$V_q$ is a certain~$2$-dimensional representation of the quantum group~$U_q(\mathfrak{sl}_2)$.
We recommend \cite{RidoutStAubin} and \cite{KasselTuraev} for further reading on~$\tl_n(a)$, and~\cite{Westbury} and~\cite{GrahamLehrer}
for details on their representation theory.

\subsection{Homology of Temperley-Lieb algebras}

The Temperley-Lieb algebra $\tl_n(a)$ has a \emph{trivial module}~$\t$ consisting of a copy of~$R$ on which all diagrams other than the identity diagram act as multiplication by~$0$.
It therefore has homology and cohomology groups~$\Tor_\ast^{\tl_n(a)}(\t,\t)$ and~$\Ext^\ast_{\tl_n(a)}(\t,\t)$.

Our first result is a vanishing theorem in the case that the parameter~$a\in R$ is invertible.

\begin{abcthm}\label{theorem-invertible}
    Let~$R$ be a commutative ring, and let~$a$ be a unit in~$R$.  Then $\Tor^{\tl_n(a)}_d(\t,\t)$ and~$\Ext_{\tl_n(a)}^d(\t,\t)$ both vanish for~$d>0$.
\end{abcthm}

The next result holds regardless of whether or not~$a$ is invertible, and uses the common assumption that~$a=v+v^{-1}$,~$v\in R^\times$. \red{However we see shortly that this assumption can be removed.}

\begin{abcthm}\label{theorem-vanishing-range}
    Let~$R$ be a commutative ring, let~$v\in R$ be a unit, let~$a=v+v^{-1}$, and let~$n\geqslant 0$.  Then 
    \[
        \Tor^{\tl_n(a)}_d(\t,\t)=0
        \quad\text{and}\quad 
        \Ext_{\tl_n(a)}^d(\t,\t)=0
    \]
    for~$1\leqslant d\leqslant (n-2)$ if~$n$ is even, and for ~$1\leqslant d\leqslant (n-1)$ if~$n$ is odd.
\end{abcthm}

Thus the map $\Tor^{\tl_{n-1}(a)}_d(\t,\t)\to\Tor^{\tl_n(a)}_d(\t,\t)$ is an isomorphism for $d\leqslant n-3$, so that we have homological stability,
and $\lim_{n\to\infty}\Tor^{\tl_n(a)}_\ast(\t,\t)=0$ in positive degrees, so the stable homology is trivial.
The latter is reminiscent of Quillen's result on the vanishing stable homology of general linear groups of finite fields in defining characteristic~\cite{QuillenGL}, and of Szymik-Wahl's result on the acyclicity of the Thompson groups~\cite{SzymikWahl}.
Theorems~\ref{theorem-invertible} and~\ref{theorem-vanishing-range} might lead us to expect that the homology and cohomology of the~$\tl_n(a)$ are largely trivial, but in fact the results are as strong as possible, at least for~$n$ even: 

\begin{abcthm}\label{theorem-sharpness}
   In the setting of Theorem~\ref{theorem-vanishing-range} above, suppose further that~$n$ is even and that~$a=v+v^{-1}$ is not a unit.  
   Then~$\Tor^{\tl_n(a)}_{n-1}(\t,\t)\neq 0$.
\end{abcthm}

Thus Theorem~\ref{theorem-invertible} does not extend to the case of~$a$ not invertible, and the stable range in Theorem~\ref{theorem-vanishing-range} is sharp.
In fact we can say more:~When $n$ is even, $\Tor^{\tl_n(a)}_{n-1}(\t,\t)\cong R/bR$ where~$b$ is a multiple of~$a$ \red{(unfortunately our methods do not allow us to say anything more concrete about~$b$)}.

\begin{rem*}
    One can compute~$\Tor_1^{\tl_n(a)}(\t,\t)$ directly using the method of \cite[Exercise 3.1.3]{Weibel}: it is~$R/aR$ for~$n=2$, and vanishes otherwise. 
    We also compute the homology and cohomology of~$\tl_2(a)$ by an explicit resolution: 
    $\Tor^{\tl_2(a)}_\ast(\t,\t)$ is $R/aR$ in odd degrees, and the kernel $R_a$ of $r\mapsto ar$ in positive even degrees, so that if $a$ is not invertible then $\Tor^{\tl_2(a)}_\ast(\t,\t)$ is non-trivial in infinitely many degrees.
\end{rem*}

\red{In~\cite{RandalWilliamsNote}, Randal-Williams shows that in fact you can remove our assumption that~$a=v+v^{-1}$ for a unit~$v\in R$, by applying Theorem~\ref{theorem-sharpness} for an associated ring~$S$. This yields the following strengthening of Theorem~\ref{theorem-vanishing-range}.

\begin{cor*}\cite[Theorem B']{RandalWilliamsNote}\label{corollary - RW strengthening}
	 Let~$R$ be a commutative ring, $a$ be any element in $R$, and~$n\geqslant 0$.  Then 
	\[
	\Tor^{\tl_n(a)}_d(\t,\t)=0
	\quad\text{and}\quad 
	\Ext_{\tl_n(a)}^d(\t,\t)=0
	\]
	for~$1\leqslant d\leqslant (n-2)$ if~$n$ is even, and for ~$1\leqslant d\leqslant (n-1)$ if~$n$ is odd.
	\begin{proof}
		The full proof can be found in \cite{RandalWilliamsNote}, and uses the Base Change Spectral Sequence~\cite[Section~5.6]{Weibel}. This is applied to the faithfully flat ring homomorphism~$R\to S$ where $S=R[v]/(v^2-a\cdot v+1)$, which by construction has a unit~$v$ and element~$a$ such that~$a=v+v^{-1}$. The results in Theorem~\ref{theorem-vanishing-range} for the ring~$S$ can now be transferred to analogous results for the ring~$R$.
	\end{proof}
\end{cor*}
}
\subsection{Jones-Wenzl projectors}
The \emph{Jones-Wenzl projector} or \emph{Jones-Wenzl idempotent} $\JW_n$, if it exists, is the element of $\tl_n(a)$ 
uniquely characterised by the following two properties:
\begin{itemize}
    \item $\JW_n\in 1+I_n$
    \item $\JW_n \cdot I_n = 0 = I_n \cdot\JW_n$ 
\end{itemize}
where $I_n$ is the two-sided ideal in $\tl_n(a)$ spanned by all diagrams other than the identity diagram.
The Jones-Wenzl projector was first introduced by Jones~\cite{JonesIndex}, was further studied by Wenzl~\cite{WenzlProjections}, and has since become important in representation theory, knot theory and the study of $3$-manifolds.

The Jones-Wenzl projector exists if and only if the trivial module $\t$ is projective. 
Moreover, when the ground ring $R$ is a field, there is a 
simple and explicit criterion for the existence of $\JW_n$, given in terms of the parameter $a$.
Thus, when this criterion holds, the vanishing of $\Tor_\ast^{\tl_n(a)}(\t,\t)$ and $\Ext^\ast_{\tl_n(a)}(\t,\t)$ in positive degrees follows immediately.

Our own Theorems~\ref{theorem-invertible} and~\ref{theorem-vanishing-range} are in general far stronger than the vanishing results obtained from the existence of $\JW_n$, as they do not require~$R$ to be a field, and the constraints are weaker. 
Indeed, in the case of $n$ even, Theorems~\ref{theorem-invertible} and~\ref{theorem-sharpness} are the final word on vanishing, since they imply that the homology and cohomology of $\tl_n(a)$ vanish in all positive degrees if and only if $a$ is invertible.
However, in the case of $n$ odd and $R$ a field, there are some situations where our theorems do not incorporate all vanishing results given by the existence of $\JW_n$.
These cases are encapsulated in the following.

\begin{abcthm}\label{thm-JW}
    Let~$n=2k+1$, and let~$R$ be a field whose characteristic does not divide $\binom{k}{t}$ for any~$1\leq t\leq k$. 
    Let~$v$ be a unit in~$R$ and assume that~$a=v+v^{-1}=0$. 
    Then $\Tor_\ast^{\tl_n(0)}(\t,\t)$ and $\Ext_{\tl_n(0)}^\ast(\t,\t)$ vanish in positive degrees.
\end{abcthm}

\red{As with Theorem~\ref{theorem-vanishing-range}, the assumption that~$a=v+v^{-1}$ for~$v$ a unit can be removed in this result.}

Combining Theorem~\ref{thm-JW} with Theorem~\ref{theorem-invertible} yields rather comprehensive vanishing results when~$R$ is a field with appropriate characteristic.
For example, it now follows that when $R$ is any field, the homology and cohomology of $\tl_3(v+v^{-1})$ vanish regardless of the choice of $v$.  
Similarly, the homology and cohomology of $\tl_5(v+v^{-1})$ will vanish over any field and for any value of $v$, except possibly in characteristic $2$ when $v+v^{-1}=0$.
Since the first appearance of our paper, Sroka \cite{Sroka} has used related techniques to show that when~$n$ is odd the~$\Tor$ groups vanish in all positive degrees, for any choice of~$R$.

The next few sections of this introduction will discuss the proofs of our main results in some detail.

\subsection{Planar injective words}

Several proofs of homological stability for the symmetric group~\cite{Maazen,Kerz,RandalWilliamsConfig} make use of the \emph{complex of injective words}. This is a highly-connected complex with an action of the symmetric group~$\frakS_n$.
Our main tool for proving Theorems~\ref{theorem-vanishing-range} and~\ref{theorem-sharpness} is the \emph{complex of planar injective words}~$W(n)$, a Temperley-Lieb analogue of the complex of injective words that we introduce and study here for the first time.
It is a chain complex of~$\tl_n(a)$-modules, and in degree~$i$ it is given by the tensor product module~$\tl_n(a)\otimes_{\tl_{n-i-1}(a)}\t$.  
This is analogous to the complex of injective words, whose~$i$-simplices form a single~$\frakS_n$-orbit with typical stabiliser~$\frakS_{n-i-1}$, which is an alternative way of saying that the~$i$-th chain group is isomorphic to~$R\frakS_n\otimes_{R\frakS_{n-i-1}}\t$.
We show the following high-acyclicity result. In order to construct appropriate differentials for~$W(n)$ we exploit a homomorphism from the group algebra of the braid group on~$n$ strands, which is not necessarily apparent from the definition of~$\tl_n(a)$. This is where the restriction of~$a$ to~$a=v+v^{-1}$ is necessary.
    
\begin{abcthm}\label{theorem-high-acyclicity}
    $H_d(W(n))$ vanishes in degrees~$d\leqslant (n-2)$.
\end{abcthm}

The complex~$W(n)$ has rich combinatorial properties, analogous to those of the complex of injective words, that we explore in the companion paper~\cite{BoydHepworthComb}.
In particular, Theorem~\ref{theorem-high-acyclicity} tells us that the homology of~$W(n)$ is concentrated in the top degree~$H_{n-1}(W(n))$, and in~\cite{BoydHepworthComb} we show that \red{when~$R$ is Noetherian} the rank of this top homology group is the~$n$-th \emph{Fine number}~$F_n$~\cite{DeutschShapiro}, an analogue of the number of derangements on~$n$ letters. Furthermore we show that the differentials of~$W(n)$ encode the \emph{Jacobsthal numbers}~\cite{JacobsthalOEIS}. Finally in the semisimple case we show that~$H_{n-1}(W(n))$ has descriptions firstly categorifying an alternating sum for the Fine numbers, and secondly in terms of standard Young tableaux.
We call the~$\tl_n(a)$-module~$H_{n-1}(W(n))$ the \emph{Fineberg module}, and we denote it~$\F_n(a)$. 
We know little about $\F_n(a)$ in general, though in the cases $n=2,3,4$
we give examples describing it in terms of the cell modules of $\tl_n(a)$.

The proof of Theorem~\ref{theorem-high-acyclicity} is perhaps the most difficult technical result in this paper.
It is obtained by filtering~$W(n)$ and showing that the filtration quotients are (suspensions of truncations of) copies of~$W(n-1)$, and then proceeding by induction.

\subsection{Spectral sequences and Shapiro's lemma}

Let us now outline how we use the complex of planar injective words~$W(n)$ to prove Theorems~\ref{theorem-vanishing-range} and~\ref{theorem-sharpness}.
Following standard approaches to homological stability for groups, we consider a spectral sequence obtained from the complex~$W(n)$.
The~$E^1$-page of our spectral sequence consists of the groups~$\Tor_j^{\tl_n(a)}(\t,\tl_n(a)\otimes_{\tl_{n-i-1}(a)}\t)$. Furthermore, thanks to Theorem~\ref{theorem-high-acyclicity}, the spectral sequence converges to~$\Tor^{\tl_n(a)}_{\ast-n+1}(\t,\F_n(a))$, where~$\F_n(a) = H_{n-1}(W(n))$ is the Fineberg module.
Our experience from homological stability tells us to apply \emph{Shapiro's lemma}, or in this context a \emph{change-of-rings isomorphism}, to identify 
\[
    \Tor^{\tl_n(a)}_\ast(\t,\tl_{n}(a)\otimes_{\tl_{n-i-1}(a)}\t)
    \qquad\text{with}\qquad
    \Tor^{\tl_{n-i-1}(a)}_\ast(\t,\t).
\]
This identification applied to the columns of our spectral sequence would allow us to implement an inductive hypothesis.
However, such a change-of-rings isomorphism would only be valid if~$\tl_n(a)$ were flat as a~$\tl_{n-i-1}(a)$-module, and this is not the case.
This failure of Shapiro's lemma is a potentially serious obstacle to proceeding further.
However, we are able to identify the columns of our spectral sequence by independent means, as follows:

\begin{abcthm}\label{theorem-shapiro}
    Let~$R$ be a commutative ring and let~$a\in R$.
    Let~$0\leqslant m<n$.  
    Then~$\Tor^{\tl_n(a)}_d(\t,\tl_n(a)\otimes_{\tl_{m}(a)}\t)$ and~$\Ext_{\tl_n(a)}^d(\tl_n(a)\otimes_{\tl_{m}(a)}\t,\t)$ both vanish for~$d>0$.
\end{abcthm}

In conjunction with a computation of the~$d=0$ case, this gives us the vanishing results of Theorem~\ref{theorem-vanishing-range}.
Moreover, in the case of~$n$ even we are able to analyse the rest of the spectral sequence (there is a single differential and a single extension problem) in sufficient detail to prove the sharpness result of Theorem~\ref{theorem-sharpness}. 
This involves a careful study of the Fineberg module~$\F_n(a)$.
In general, our method identifies~$\Tor_\ast^{\tl_n(a)}(\t,\t)$ with~$\Tor_{\ast-n}^{\tl_n(a)}(\t,\F_n(a))$, except in degrees~$\ast=n-1,n$ when~$n$ is even.

\subsection{Inductive resolutions}
It remains for us to discuss the proofs of  Theorems~\ref{theorem-invertible} and~\ref{theorem-shapiro}.
These results are proved by a novel method that exploits the structure of the Temperley-Lieb algebras, and in particular they lie outwith the standard tool-kit of homological stability.
Moreover, it is Theorem~\ref{theorem-shapiro} which allows us to overcome the failure of Shapiro's lemma.

The two theorems are very similar: Theorem~\ref{theorem-invertible} is an instance of the more general statement that~$\Tor_\ast^{\tl_n(a)}(\t,\tl_n(a)\otimes_{\tl_m(a)}\t)$ vanishes in positive degrees for~$m\leqslant n$ and~$a$ invertible, while Theorem~\ref{theorem-shapiro} states that the same groups vanish for~$m<n$ and~$a$ arbitrary.
These are both proved by strong induction on~$m$. 
The initial cases~$m=0,1$ are immediate because then~$\tl_m(a)=R$ so~$\tl_n(a)\otimes_{\tl_m(a)}\t$ is free. 
The induction step is proved by constructing and exploiting a resolution of~$\tl_n(a)\otimes_{\tl_m(a)}\t$ whose terms have the form~$\tl_n(a)\otimes_{\tl_{m-1}(a)}\t$ and~$\tl_n(a)\otimes_{\tl_{m-2}(a)}\t$, and then applying the inductive hypothesis.
We call these resolutions \emph{inductive resolutions} since they resolve the next module in terms of those already considered.

Our technique of inductive resolutions is generalised in a joint paper with Patzt, \cite{BoydHepworthPatzt}, where we show that the homology of the Brauer algebras is isomorphic to the homology of the symmetric groups in a stable range when the parameter~$\delta$ is not invertible, and in every degree when~$\delta$ is invertible.
This provides concrete evidence that the new techniques developed in this paper can be adapted to other algebras to obtain results of similar strength.


\subsection{Discussion: Homological stability for algebras}

As stated earlier, we regard the present paper, together with the results of \cite{HepworthIH} on Iwahori-Hecke algebras, as proof-of-concept for the export of the techniques of homological stability to the setting of algebras. 
And, since the first appearance of this paper, these techniques have been extended to the setting of Brauer algebras in our joint work with Patzt~\cite{BoydHepworthPatzt}.
We hope that the present paper, together with~\cite{HepworthIH,BoydHepworthPatzt}, will be a springboard for further research in this direction.

One of the main motivations for studying the homology of groups, is that homology is a useful `measurement' of the group. 
Put another way, homology is a powerful invariant, where the power comes from the fact that it is both informative, and (relatively) computable. 
The~$\Tor$ and $\Ext$ groups of algebras are likewise strong invariants, and it is our hope that homology and cohomology of algebras can be utilised as a tool to answer questions in the fields where the algebras arise. 
For example, modern representation theory is rich in conjectures, and home to surprising isomorphisms between apparently very different algebras~\cite{BrundanKleshchev,BowmanCoxHazi}. 
Understanding the similarities and differences between naturally-arising algebras is precisely the kind of question that could be investigated via $\Tor$ and $\Ext$-groups.

We will now discuss some questions arising from our work.
Readers with experience in homological stability will be able to think of many new questions in this direction, so we will simply list some that are most prominent in our minds.

The Temperley-Lieb algebra can be regarded as an algebra of~$1$-dimensional cobordisms embedded in~$2$ dimensions, and the Brauer algebra can similarly be viewed as an algebra of~$1$-dimensional cobordisms embedded in infinite dimensions.

\begin{question*}
    Are there analogues of the Temperley-Lieb algebra consisting of~$d$-dimensional cobordisms embedded in~$n$ dimensions?
    Does homological stability hold for these algebras?  
    And can the stability be understood in an essentially geometric way?
\end{question*}

And more generally:

\begin{question*}
    For which natural families of algebras does homological stability hold?
\end{question*}

Candidate algebras, closely related to the existing cases, are: Iwahori-Hecke and Temperley-Lieb algebras of types~$B$ and~$D$; the periodic and dilute Temperley-Lieb algebras; and the blob, partition and Birman-Murakami-Wenzl algebras.
We invite the reader to think of possibilities from further afield.

There have recently been advances in building general frameworks for homological stability proofs. In~\cite{RandalWilliamsWahl} Randal-Williams and Wahl introduce a categorical framework that encapsulates, improves and extends several of the standard techniques used in homological stability proofs for groups. In~\cite{GKRW-Ek} Galatius, Kupers and Randal-Williams introduce a framework that applies to~$E_k$-algebras in simplicial modules.  It exploits the notion of cellular $E_k$-algebras, and incorporates methods for proving~\emph{higher stability} results.
This invites us to pose the following questions.

\begin{question*}
    Does the general homological stability machinery of Randal-Williams and Wahl~\cite{RandalWilliamsWahl} generalise to an~$R$-linear version, giving a general framework to prove that a family of~$R$-algebras~$A_0\to A_1\to A_2\to\cdots$ satisfies homological stability?
\end{question*}
    
In this question, the most interesting issue is what form the resulting complexes will take.
One might expect that for a family of algebras the relevant complexes will be constructed from tensor products, as with our complex $W(n)$. 
However it may happen, as in this paper, that flatness issues arise, in which case it seems unlikely that complexes built from the honest tensor products will be sufficient.

\begin{question*}
    Can the homological stability machinery of Galatius, Kupers and Randal-Williams~\cite{GKRW-Ek} be applied in the setting of algebras? 
\end{question*}

It seems extremely likely that homology of Temperley-Lieb algebras will indeed fit into the framework of~\cite{GKRW-Ek}, by using appropriate simplicial models for the $\Tor^{\tl_n(a)}_\ast(\t,\t)$, or more precisely for the chain complexes underlying these $\Tor$ groups.
Again, the difficulty will lie in identifying and computing the associated \emph{splitting complexes}, especially when flatness issues arise.

\subsection{Outline}

In Section~\ref{section-temperley-lieb} we recall the definition of the Temperley-Lieb algebra, the Jones basis, the relationship with Iwahori-Hecke algebras, and we establish results on the induced modules~$\tl_n(a)\otimes_{\tl_m(a)}\t$ that will be important in the rest of the paper.
Section~\ref{section-inductive} establishes our inductive resolutions and proves Theorems~\ref{theorem-invertible} and~\ref{theorem-shapiro}.
Section~\ref{section-Wn} introduces the complex of planar injective words~$W(n)$ and the Fineberg module~$\F_n(a)$.
Sections~\ref{section-stability} and~\ref{section-sharpness} then use~$W(n)$, in particular its high-acyclicity (Theorem~\ref{theorem-high-acyclicity}), to prove Theorems~\ref{theorem-vanishing-range} and~\ref{theorem-sharpness}.
Section~\ref{section-tl-two} investigates our results in the case of~$\tl_2(a)$, computing the homology directly and also in terms of the Fineberg module~$\F_2(a)$.
Section~\ref{section-high-acyclicity} proves Theorem~\ref{theorem-high-acyclicity}.
Section~\ref{section-JW} investigates the vanishing results given by the Jones-Wenzl projectors and proves Theorem~\ref{thm-JW}.

\subsection{Acknowledgements}
The authors would like to thank the Max Planck Institute for Mathematics in Bonn for its support and hospitality. We are grateful to Robert Spencer and Paul Wedrich for interesting and helpful discussions about the representation theory of Temperley-Lieb algebras, \red{and to Oscar Randal-Williams for his extension of Theorem~\ref{theorem-vanishing-range}. We would also like to thank the anonymous referees for very insightful comments, which have much improved this paper.}

\section{Temperley-Lieb Algebras}
\label{section-temperley-lieb}

In this section we will cover the basic facts about the Temperley-Lieb algebra that we will need for the rest of the paper. There is some overlap between the material recalled here and in~\cite{BoydHepworthComb}.
In particular, we cover the definitions by generators and relations and by diagrams; we discuss the Jones basis for~$\tl_n(a)$; we look at the induced modules~$\tl_n(a)\otimes_{\tl_m(a)}\t$ that will be an essential ingredient in all that follows; and we discuss the homomorphism from the Iwahori-Hecke algebra of type~$A_{n-1}$ into~$\tl_n(a)$.
Historical references on Temperley-Lieb algebras were given in the introduction.
General references for readers new to the~$\tl_n(a)$ are Section~5.7 of Kassel and Turaev's book~\cite{KasselTuraev} on the braid groups, and especially Sections~1 and~2 of Ridout and Saint-Aubin's survey on the representation theory of the~$\tl_n(a)$~\cite{RidoutStAubin}.

\begin{defn}[The Temperley-Lieb algebra~$\tl_n(a)$]
\label{definition-temperley-lieb}
    Let~$R$ be a commutative ring and let~$a\in R$.  Let~$n$ be a non-negative integer.
    The \emph{Temperley-Lieb algebra}~$\tl_n(a)$
    is defined to be the~$R$-algebra with generators~$U_1,\ldots,U_{n-1}$ and the following relations:
    \begin{enumerate}
        \item
       ~$U_iU_j=U_jU_i$ for~$j\neq i\pm 1$
        \item
       ~$U_iU_jU_i = U_i$ for ~$j=i\pm 1$
        \item
       ~$U_i^2 = aU_i$ for all~$i$.
    \end{enumerate}
    Thus elements of the Temperley-Lieb algebra are formal sums of monomials in the~$U_i$, with coefficients in the ground ring~$R$, modulo the relations above. 
    We often write~$\tl_n(a)$ as~$\tl_n$. We note here that~$\tl_0=\tl_1=R$.
\end{defn}

There is an alternative definition of~$\tl_n$ in terms of diagrams.
In this description, an element of~$\tl_n$ is an~$R$-linear combination of \emph{planar diagrams} (or~$1$-dimensional cobordisms).   
Each planar diagram consists of two vertical lines in the plane, decorated with~$n$  dots labelled~$1,\ldots,n$ from bottom to top, together with a collection of~$n$ arcs joining the dots in pairs.  
The arcs must lie between the vertical lines, they must be disjoint, and the diagrams are taken up to isotopy.  
For example, here are two planar diagrams in the case~$n=5$:
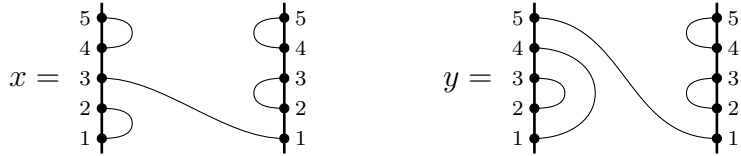
\[
    x=
    \begin{tikzpicture}[scale=0.4, baseline=(base)]
        \coordinate (base) at (0,2.75);
        \draw[line width = 1](0,0.5)--(0,5.5);
        \draw[line width = 1](6,0.5)--(6,5.5);
        \foreach \x in {1,2, 3,4,5}{
            \draw[fill=black] (0,\x) circle [radius=0.15] (6,\x) circle [radius=0.15];
            \draw (0,\x) node[left] {$\scriptstyle{\x}$};
            \draw (6,\x) node[right] {$\scriptstyle{\x}$};
        } 
        \draw (0,1) to[out=0,in=-90] (1,1.5) to[out=90,in=0] (0,2);
        \draw (0,4) to[out=0,in=-90] (1,4.5) to[out=90,in=0] (0,5);
        \draw (6,4) to[out=180,in=-90] (5,4.5) to[out=90,in=180] (6,5);
        \draw (6,2) to[out=180,in=-90] (5,2.5) to[out=90,in=180] (6,3);
        \draw (0,3) .. controls (2,3) and (4,1) .. (6,1);
    \end{tikzpicture}
    \qquad\qquad
    y=
    \begin{tikzpicture}[scale=0.4, baseline=(base)]
        \coordinate (base) at (0,2.75);
        \draw[line width = 1](0,0.5)--(0,5.5);
        \draw[line width = 1](6,0.5)--(6,5.5);
        \foreach \x in {1,2, 3,4,5}{
            \draw[fill=black] (0,\x) circle [radius=0.15] (6,\x) circle [radius=0.15];
            \draw (0,\x) node[left] {$\scriptstyle{\x}$};
            \draw (6,\x) node[right] {$\scriptstyle{\x}$};
        } 
        \draw (0,2) to[out=0,in=-90] (1,2.5) to[out=90,in=0] (0,3);
        \draw (6,4) to[out=180,in=-90] (5,4.5) to[out=90,in=180] (6,5);
        \draw (6,2) to[out=180,in=-90] (5,2.5) to[out=90,in=180] (6,3);
        \draw (0,1) to[out=0,in=-90] (2,2.5) to[out=90,in=0] (0,4);
        \draw (0,5) .. controls (3,5) and (3,1) .. (6,1);
    \end{tikzpicture}
\]
We will often omit the labels on the dots.
Multiplication of diagrams is given by placing them side-by-side and joining the ends.
Any closed loops created by this process are then erased and replaced with a factor of~$a$.
For example, the product~$xy$ of the elements~$x$ and~$y$ above is:
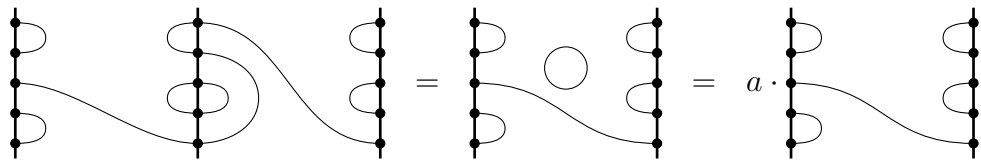
\[
    \begin{tikzpicture}[scale=0.4, baseline=(base)]
        \coordinate (base) at (0,2.75);
        \draw[line width = 1](0,0.5)--(0,5.5);
        \draw[line width = 1](6,0.5)--(6,5.5);
        \draw[line width = 1](12,0.5)--(12,5.5);
        \foreach \x in {1,2, 3,4,5}{
            \draw[fill=black] (0,\x) circle [radius=0.15] (6,\x) circle [radius=0.15] (12,\x) circle [radius=0.15];
        } 
        \draw (0,1) to[out=0,in=-90] (1,1.5) to[out=90,in=0] (0,2);
        \draw (0,4) to[out=0,in=-90] (1,4.5) to[out=90,in=0] (0,5);
        \draw (6,4) to[out=180,in=-90] (5,4.5) to[out=90,in=180] (6,5);
        \draw (6,2) to[out=180,in=-90] (5,2.5) to[out=90,in=180] (6,3);
        \draw (0,3) .. controls (2,3) and (4,1) .. (6,1);
        \draw (6,2) to[out=0,in=-90] (7,2.5) to[out=90,in=0] (6,3);
        \draw (12,4) to[out=180,in=-90] (11,4.5) to[out=90,in=180] (12,5);
        \draw (12,2) to[out=180,in=-90] (11,2.5) to[out=90,in=180] (12,3);
        \draw (6,1) to[out=0,in=-90] (8,2.5) to[out=90,in=0] (6,4);
        \draw (6,5) .. controls (9,5) and (9,1) .. (12,1);
    \end{tikzpicture}
    \ \ 
    =
    \ \ 
    \begin{tikzpicture}[scale=0.4, baseline=(base)]
        \coordinate (base) at (0,2.75);
        \draw[line width = 1](0,0.5)--(0,5.5);
        \draw[line width = 1](6,0.5)--(6,5.5);
        \foreach \x in {1,2, 3,4,5}{
            \draw[fill=black] (0,\x) circle [radius=0.15] (6,\x) circle [radius=0.15];
        } 
        \draw (0,1) to[out=0,in=-90] (1,1.5) to[out=90,in=0] (0,2);
        \draw (0,4) to[out=0,in=-90] (1,4.5) to[out=90,in=0] (0,5);
        \draw (0,3) .. controls (3,3) and (3,1) .. (6,1);
        \draw (6,4) to[out=180,in=-90] (5,4.5) to[out=90,in=180] (6,5);
        \draw (6,2) to[out=180,in=-90] (5,2.5) to[out=90,in=180] (6,3);
        \draw (3,3.5) circle (0.7);
    \end{tikzpicture}
    \ \ 
    =
    \ \ 
    a\cdot
    \begin{tikzpicture}[scale=0.4, baseline=(base)]
        \coordinate (base) at (0,2.75);
        \draw[line width = 1](0,0.5)--(0,5.5);
        \draw[line width = 1](6,0.5)--(6,5.5);
        \foreach \x in {1,2, 3,4,5}{
            \draw[fill=black] (0,\x) circle [radius=0.15] (6,\x) circle [radius=0.15];
        } 
        \draw (0,1) to[out=0,in=-90] (1,1.5) to[out=90,in=0] (0,2);
        \draw (0,4) to[out=0,in=-90] (1,4.5) to[out=90,in=0] (0,5);
        \draw (0,3) .. controls (3,3) and (3,1) .. (6,1);
        \draw (6,4) to[out=180,in=-90] (5,4.5) to[out=90,in=180] (6,5);
        \draw (6,2) to[out=180,in=-90] (5,2.5) to[out=90,in=180] (6,3);
    \end{tikzpicture}
\]
(We have subscribed to the heresy of~\cite{RidoutStAubin} by drawing planar diagrams that go from left to right rather than top to bottom.) 

One can pass from the generators-and-relations definition of~$\tl_n$ in Definition~\ref{definition-temperley-lieb} to the diagrammatic description of the previous paragraph as follows.
For~$1\leqslant i \leqslant n-1$, to each~$U_i$ we associate the planar diagram shown below.
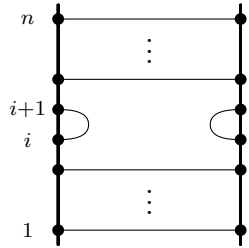
\[
    \begin{tikzpicture}[scale=0.4]
    \foreach \x in {1, 3,4,5,6,8}
    \foreach \y in {0,6}
    \draw[fill=black, line width=1] (\y,\x) circle [radius=0.15](\y,.5)--(\y,8.5);
    \foreach \x in {1, 3,6,8} 
    \draw (0,\x) --(6,\x);
    \foreach \x in {3}
   \draw (\x,2.2) node {$\scriptstyle{\vdots}$} (\x,7.2) node {$\scriptstyle{\vdots}$};
    \draw (-1,1)node {$\scriptstyle{1}$} 
    (-1,4) node {$\scriptstyle{i}$}  (-1,5) node {$\scriptstyle{i+1}$}   (-1,8) node {$\scriptstyle{n}$};
    \draw (7,1)node {$\scriptstyle{\phantom{i+2}}$};
    \draw (0,4) to[out=0,in=-90] (1,4.5) to[out=90,in=0] (0,5);
    \draw (6,4) to[out=180,in=-90] (5,4.5) to[out=90,in=180] (6,5);
    \end{tikzpicture}
\]
We refer to an arc joining adjacent dots as a \emph{cup}. The relations for the Temperley-Lieb algebras are satisfied, and two of them are illustrated in Figure~\ref{fig: tl relations}. 
The fact that this determines an isomorphism between the algebra defined by generators and relations, and the one defined by diagrams, is proved in \cite[Theorem~2.4]{RidoutStAubin},  \cite[Theorem~5.34]{KasselTuraev}, and~\cite[Section 6]{KauffmanDiagrammatic}.

\begin{figure}
    \centering
    \subfigure[The relation~$U_i^2=aU_i$.]{
    \begin{tikzpicture}[scale=0.4]
    \foreach \x in {1, 3,4,5,6,8}
    \foreach \y in {0,4,8,10,16}
    \draw[fill=black, line width=1] (\y,\x) circle [radius=0.15] (\y,.5)--(\y,8.5);
    \foreach \x in {1, 3,6,8}
    \draw[black] (0,\x) --(8,\x) (10,\x)--(16,\x);
    \foreach \x in {2,6,13}
    \draw (\x,2.2) node {$\scriptstyle{\vdots}$} (\x,7.2) node {$\scriptstyle{\vdots}$};
    \draw (-1,1)node {$\scriptstyle{1}$} 
    (-1,4) node {$\scriptstyle{i}$}  (-1,5) node {$\scriptstyle{i+1}$}  (-1,6) node {$\scriptstyle{i+2}$}  (-1,8) node {$\scriptstyle{n}$};
    \foreach \x in {0,4,10}
    \draw (\x,4) to[out=0,in=-90] (\x+1.5,4.5) to[out=90,in=0] (\x,5);
    \foreach \x in {4,8,16}
    \draw (\x,4) to[out=180,in=-90] (\x-1.5,4.5) to[out=90,in=180] (\x,5);
    \draw (13,4) to[out=0, in=-90] (14,4.5) to[out=90,in=0] (13,5) to[out=180,in=90] (12,4.5) to[out=-90, in=180] (13,4);
    \draw (9, 4.5) node[scale=1.5] {$=$};
    \end{tikzpicture}}
    \subfigure[The relation~$U_iU_{i+1}U_i=U_i$.]{
    \begin{tikzpicture}[scale=0.4]
    \foreach \x in {1,3,4,5,6,8}
    \foreach \y in {0,3,6,9,12,16}
    \draw[fill=black, line width=1] (\y,\x) circle [radius=0.15](\y,.5)--(\y,8.5);
    \foreach \x in {1,3,8}
    \draw[black] (0,\x) --(9,\x) (12,\x)--(16,\x);
    \foreach \x in {1.5,4.5,7.5,14}
    \draw (\x,2.2) node {$\scriptstyle{\vdots}$} (\x,7.2) node {$\scriptstyle{\vdots}$};
    \draw (-1,1)node {$\scriptstyle{1}$} 
    (-1,4) node {$\scriptstyle{i}$}  (-1,5) node {$\scriptstyle{i+1}$}  (-1,6) node {$\scriptstyle{i+2}$}  (-1,8) node {$\scriptstyle{n}$};
    \foreach \x in {0,6,12}
    \draw (\x,4) to[out=0,in=-90] (\x+1,4.5) to[out=90,in=0] (\x,5);
    \foreach \x in {3,9,16}
    \draw (\x,4) to[out=180,in=-90] (\x-1,4.5) to[out=90,in=180] (\x,5);
    \draw (6,5) to[out=180,in=-90] (5,5.5) to[out=90,in=180] (6,6) -- (9,6) (12,6)--(16,6) (3,4)--(6,4) (0,6)--(3,6) to[out=0,in=90] (4,5.5) to[out=-90, in=0] (3,5);
    \draw (10.5, 4.5) node[scale=1.5] {$=$};
    \end{tikzpicture}
    }
    \caption{Diagrammatic relations in~$\tl_n$.}
    \label{fig: tl relations}
\end{figure}

In the rest of the paper we will refer to the diagrammatic point of view on the Temperley-Lieb algebra, but we will not rely on it for any proofs.

\subsection{The Jones basis}

From the diagrammatic point of view the Temperley-Lieb algebra~$\tl_n$ has an evident~$R$-basis given by the (isotopy classes of) planar diagrams.  This is called the \emph{diagram basis}.   We now recall the analogue of the diagram basis given in terms of the~$U_i$, which is called the \emph{Jones basis} for~$\tl_n$, and we prove some additional facts about it that we will require later.  See \cite[Section 5.7]{KasselTuraev}, \cite[Section~2]{RidoutStAubin} or~\cite[Section 6]{KauffmanDiagrammatic}, but note that conventions vary, and see Remark~\ref{rem-iso to KT gens} below in particular.

\begin{defn}[Jones normal form]
    The \emph{Jones normal form} for elements of $\tl_n(a)$ is defined as follows. Let 
    	\[
    	    n>a_k>a_{k-1}>\cdots>a_1>0 \hspace{1cm} n>b_k>b_{k-1}>\cdots >b_1>0
    	\]
	be integers such that~$b_i\geqslant a_i$ for all~$i$. Let~$\ul{a}=(a_k, \ldots a_1)$ and~$\ul{b}=(b_k, \ldots b_1)$. Then set
    	\[
    	x_{\ul{a}, \ul{b}}=(U_{a_k}\ldots U_{b_k})\cdot(U_{a_{k-1}}\ldots U_{b_{k-1}})\cdots( U_{a_1}\ldots U_{b_1})
    	\]
   	where the subscripts of the generators increase in each tuple~$U_{a_i}\ldots U_{b_i}$.
   	A word written in the form~$x_{\ul{a},\ul{b}}$ is said to be written in \emph{Jones normal form} for~$\tl_n(a)$.
\end{defn}
   	
\begin{example}
    In~$\tl_5$ the words 
    \begin{align*}
        U_1U_2U_3U_4 &= (U_1U_2U_3U_4) =x_{(1),(4)}
        \\
        U_4U_3U_2U_1 &= (U_4)\cdot(U_3)\cdot(U_2)\cdot(U_1)= x_{(4,3,2,1),(4,3,2,1)}
        \\
        U_3U_4U_1U_2 &= (U_3U_4)\cdot(U_1U_2)=x_{(3,1),(4,2)}
        \\
        U_2U_3 U_1U_2 &= (U_2U_3)\cdot(U_1U_2)=x_{(2,1),(3,2)}
    \end{align*}
    are in Jones normal form.
    The word~$U_2U_1U_4U_2U_3$ is not, 
    but it can be rewritten using the defining relations to give
   \[U_2U_1U_4U_2U_3 = U_4U_2U_1U_2U_3 = U_4U_2U_3 = (U_4)(U_2U_3)=x_{(4,2),(4,3)}.\]
\end{example}
   	
Denote the subset of~$\tl_n$ consisting of all~$x_{\ul{a},\ul{b}}$ with~$\ul{a}=(a_1,\ldots,a_k)$ and~$\ul{b}=(b_1,\ldots,b_k)$ by~$\tl_{n,k}$. Then the set
    	\[
    	\tl_{n,0} \sqcup \tl_{n,1} \sqcup \cdots \sqcup \tl_{n,n-1}
    	\]
is a basis \red{(recall that by basis we always mean~$R$-module basis)} of~$\tl_n$, called the \emph{Jones basis}.
For a proof of this fact see~\cite[Corollary~5.32]{KasselTuraev}, \cite[pp.967-969]{RidoutStAubin} or~\cite[Section 6]{KauffmanDiagrammatic}, though we again warn the reader that conventions vary.

There is an algorithm for taking a diagram and writing it as an element of the Jones basis, see ~\cite[Section 6]{KauffmanDiagrammatic}.
We summarise the algorithm here. Let the \emph{$i$-th row} of the diagram be the horizontal strip whose left and right ends lie between the dots~$i$ and~$(i+1)$ on each vertical line.
Take a planar diagram, and ensure that it is drawn in minimal form: all arcs connecting the same side of the diagram to itself are drawn as semicircles, and all arcs from left to right are drawn without any cups, i.e.~transverse to all vertical lines, \red{and such that each arc of the diagram intersects each row transversely and at most once}.

Proceed along each row of the diagram, connecting the consecutive arcs encountered with a dotted horizontal line labelled by the row in question.  This is done in an alternating fashion: the first arc encountered is connected to the second by a dotted line, then the third is connected to the fourth, and so on.
If we start with the elements~$x$ and~$y$ used earlier in this section, then this gives us the following:
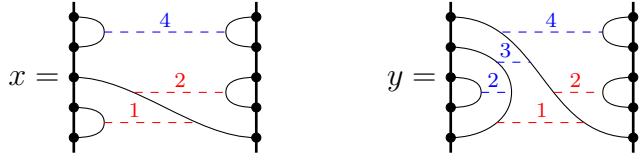
\[
    x=
    \begin{tikzpicture}[scale=0.4, baseline=(base)]
        \coordinate (base) at (0,2.75);
        \draw[line width = 1](0,0.5)--(0,5.5);
        \draw[line width = 1](6,0.5)--(6,5.5);
        \foreach \x in {1,2, 3,4,5}{
            \draw[fill=black] (0,\x) circle [radius=0.15] (6,\x) circle [radius=0.15];
        } 
        \draw (0,1) to[out=0,in=-90] (1,1.5) to[out=90,in=0] (0,2);
        \draw (0,4) to[out=0,in=-90] (1,4.5) to[out=90,in=0] (0,5);
        \draw (6,4) to[out=180,in=-90] (5,4.5) to[out=90,in=180] (6,5);
        \draw (6,2) to[out=180,in=-90] (5,2.5) to[out=90,in=180] (6,3);
        \draw (0,3) .. controls (2,3) and (4,1) .. (6,1);
	\draw[blue,dashed](1,4.5)--(5,4.5);
	\node[blue] at (3,4.9){$\scriptstyle 4$};
	\draw[red,dashed](2,2.5)--(5,2.5);
	\node[red] at (3.5,2.9){$\scriptstyle 2$};
	\draw[red,dashed](1,1.5)--(4,1.5);
	\node[red] at (2,1.9){$\scriptstyle 1$};
    \end{tikzpicture}
    \qquad\qquad
    y=
    \begin{tikzpicture}[scale=0.4, baseline=(base)]
        \coordinate (base) at (0,2.75);
        \draw[line width = 1](0,0.5)--(0,5.5);
        \draw[line width = 1](6,0.5)--(6,5.5);
        \foreach \x in {1,2, 3,4,5}{
            \draw[fill=black] (0,\x) circle [radius=0.15] (6,\x) circle [radius=0.15];
        } 
        \draw (0,2) to[out=0,in=-90] (1,2.5) to[out=90,in=0] (0,3);
        \draw (6,4) to[out=180,in=-90] (5,4.5) to[out=90,in=180] (6,5);
        \draw (6,2) to[out=180,in=-90] (5,2.5) to[out=90,in=180] (6,3);
        \draw (0,1) to[out=0,in=-90] (2,2.5) to[out=90,in=0] (0,4);
        \draw (0,5) .. controls (3,5) and (3,1) .. (6,1);
	\draw[blue,dashed](1.7,4.5)--(5,4.5);
	\node[blue] at (3.3,4.9){$\scriptstyle 4$};
	\draw[blue,dashed](1.5,3.5)--(2.6,3.5);
	\node[blue] at (1.8,3.9){$\scriptstyle 3$};
	\draw[blue,dashed](1,2.5)--(2,2.5);
	\node[blue] at (1.4,2.9){$\scriptstyle 2$};
	\draw[red,dashed](3.4,2.5)--(5,2.5);
	\node[red] at (4.1,2.9){$\scriptstyle 2$};
	\draw[red,dashed](1.5,1.5)--(4.5,1.5);
	\node[red] at (3,1.9){$\scriptstyle 1$};
    \end{tikzpicture}
\]
A \emph{sequence} in such a decorated diagram is taken by travelling right along the dotted arcs and up along the solid arcs from one dotted arc to the next, starting as far to the left as possible.
The above diagrams each have two sequences, indicated in blue and red.
The sequences in a diagram are linearly ordered by scanning from top to bottom and recording a sequence when one of its dotted lines is first encountered.
So in the above diagrams the blue sequences precede the red ones.
One now obtains a Jones normal form for the element by working through the sequences in turn, writing out the labels from left to right, and then taking the corresponding monomial in the~$U_i$:
\[
    x=(U_4)(U_1U_2)=x_{(4,1),(4,2)},\qquad\qquad y=(U_2U_3U_4)(U_1U_2)=x_{(2,1),(4,2)}.
\]

We now present a proof that the Jones basis spans, adding slightly more detail than we found in the references.  The extra detail will be used in the next section.

\begin{defn}\label{defn-index and terminus}
Given a word~$w=U_{i_1}\ldots U_{i_n}$ in the~$U_i$, define the \emph{terminus} to be the subscript of the final letter of the word appearing,~$i_n$, and denote it~$t(w)$. Set~$t(1)=\infty$ as a convention. Define the~\emph{index} of~$w$ to be the minimum subscript~$i_j$ appearing, and denote it~$i(w)$.
\end{defn}

\begin{remark}\label{rem-iso to KT gens}
    Note here that the notions of Jones normal form and index in $\tl_n(a)$ coincide with those of~\cite{KasselTuraev}, under the bijection which sends the generator~$e_i$ of~\cite{KasselTuraev} to the generator~$U_{n-i}$ used in this paper, for~$1\leqslant i \leqslant n-1$.
\end{remark}

The following two lemmas are an enhancement of Lemmas~5.25 and 5.26 of~\cite{KasselTuraev}.

\begin{lemma}\label{lem-index once}
    Any word~$w\in\tl_n(a)$ is equal in~$\tl_n(a)$ to a scalar multiple of a word~$w'$ in which
    \begin{enumerate}[(a)]
        \item~$i(w)=i(w')$ and~$U_{i(w)}$ appears exactly once in~$w'$;
        \item\red{$t(w')=t(w)$.}
    \end{enumerate}
\end{lemma}

\red{
        Point (a) occurs as Lemma~5.25 of~\cite{KasselTuraev}, 
	and the following is a simple extension 
	of the proof that appears there.  
	We have opted to give our proof in full because, as well as
	the minor extension of the proof, our notation differs from that 
	of~\cite{KasselTuraev} as in Remark~\ref{rem-iso to KT gens}.

	\begin{proof}
		We proceed by reverse induction on the index $i(w)$ of $w$,
		which lies in the range $1\leqslant i(w)\leqslant n-1$.
		If $i(w)=n-1$, then $w=U_{n-1}^i$ for some $i\geq 1$,
		so that  $w=a^{i-1}U_{n-1}$ is a scalar multiple of the word
		$U_{n-1}$. 
		Since the words $U_{n-1}^i$ and $U_{n-1}$
		have the same index and terminus, 
		the result holds in this case.
		
		Suppose that the claim holds for all words of index $>p$
		and let $w$ be a nonempty word of
		index $p$.
		Suppose that $U_p$ appears in $w$ at least twice.
		Then we may write $w=w_1 U_pw'U_pw_2$ where $i(w')=\ell>p$.
		
		If $\ell>p+1$, then all letters of $w'$ commute with $U_p$,
		so that 
		\[
			w=w_1 U_pw'U_pw_2=w_1w'U_p^2w_2=aw_1w'U_pw_2.
		\]
		Thus we have reduced the number of occurrences of $U_p$
		in $w$ while preserving the (nonempty) final portion
		$U_pw_2$ of the word, so that the terminus remains unchanged.

		If $\ell=p+1$, then by the induction hypothesis we may
		assume  that $U_{p+1}$ appears only once in $w'$,
		so that $w'=w_3 U_{p+1}w_4$ where $w_3,w_4$ are words of 
		index $\geq p+2$.
		Therefore $w_3,w_4$ commute with $U_p$, and consequently
		\begin{align*}
			w&=w_1 U_pw'U_pw_2
			=w_1 U_pw_3 U_{p+1}w_4  U_pw_2
			\\
			&=w_1 w_3 U_pU_{p+1}U_pw_4 w_2
			=w_1 w_3 U_pw_4 w_2
			=w_1 w_3 w_4   U_pw_2.
		\end{align*}
		So again, we have reduced the number of occurrences of 
		$U_p$ in the word while preserving the final (nonempty)
		portion $U_pw_2$, and in particular preserving the terminus.

		Repeating the process of reducing the number of occurrences
		of $U_p$ while preserving the terminus, we find that $w$
		is a scalar multiple of a word $w'$ of the required
		form.
	\end{proof}
}

\begin{lemma}\label{lem-terminus of JNF}
     Any word~$w\in\tl_n(a)$ is equivalent in~$\tl_n(a)$ to a scalar multiple of a word~$w'$ such that
     \begin{enumerate}[(a)]
         \item~$w'$ is written in Jones normal form;
         \item~$t(w')\leqslant t(w)$;
         \item if~$t(w')<t(w)$ then~$t(w')\leqslant t(w)-2$.
     \end{enumerate}
     \red{
     \begin{proof}
     As in the previous Lemma, point (a) occurs as~\cite[Lemma 5.26]{KasselTuraev}. We refer the reader to that proof, with the following modifications:
        \begin{itemize}
            \item Invoke the bijection of generators of Remark~\ref{rem-iso to KT gens}. This amounts to replacing each occurrence of~$e_i$ with~$U_{n-i}$, so for example the subscripts~$1$ and~$n-1$ are interchanged, and inequalities are `reversed'.
            \item Whenever the inductive hypothesis is used in~\cite[Lemma~5.25]{KasselTuraev}, instead use the statement of the present lemma as a stronger inductive hypothesis.
         \item At the point where \cite[Lemma 5.25]{KasselTuraev} is used in~\cite[Lemma~5.26]{KasselTuraev}, use instead Lemma~\ref{lem-index once}.
        \end{itemize}
	With these modifications in place, one can simply observe
	how the terminus changes in the proof 
	of~\cite[Lemma~5.26]{KasselTuraev}, to obtain the present 
	strengthening of that result.
     \end{proof}
     }
\end{lemma}

\subsection{Induced modules of Temperley-Lieb Algebras}\label{section-relative}

\begin{defn}[The trivial module~$\t$]
    The \emph{trivial} module~$\t$ of the Temperley-Lieb algebra~$\tl_n(a)$ is the module consisting of~$R$ with the action of~$\tl_n(a)$ in which all of the generators~$U_1,\ldots,U_{n-1}$ act as~$0$.
    We can regard~$\t$ as either a left or right module over~$\tl_n(a)$, and we will usually do that without indicating so in the notation.
\end{defn}	

\begin{defn}[Sub-algebra convention] 
    For~$m\leqslant n$, we will regard~$\tl_m(a)$ as the sub-algebra of~$\tl_n(a)$ generated by the elements~$U_1,\ldots,U_{m-1}$.  
    We will often regard~$\tl_n(a)$ as a left~$\tl_n(a)$-module and a right~$\tl_m(a)$-module, so that we obtain the left~$\tl_n(a)$-module~$\tl_n(a)\otimes_{\tl_m(a)}\t$.
\end{defn}

\red{\begin{rem}
	Elements of~$\tl_n(a)\otimes_{\tl_m(a)}\t$ can always be written as elementary tensors of the form~$y\otimes 1$, since in this module~$x\otimes r = rx\otimes 1$ for all~$r\in R$.
\end{rem}
}
The modules~$\tl_n\otimes_{\tl_m}\t$ are an essential ingredient in the rest of this paper: they will be the building blocks of all the complexes we construct in order to prove our main results, in particular the complex of planar injective words~$W(n)$.
The rest of this section will study them in some detail, in particular finding a basis for them analogous to the Jones basis.

\begin{remark}[$\tl_n(a)\otimes_{\tl_m(a)}\t$ via diagrams]
\label{remark: black boxes}
     The elements of~$\tl_n(a)\otimes_{\tl_m(a)} \t$ can be regarded as diagrams, just like the elements of~$\tl_n(a)$, except that now the first~$m$ dots on the right are encapsulated within a \emph{black box}, and if any cups can be absorbed into the black box, then the diagram is identified with~$0$.
     For example, some elements of~$\tl_4(a)\otimes_{\tl_3(a)}\t$ are depicted as follows:
    \[
        	\begin{tikzpicture}[scale=0.45]
        	\foreach \x in {1,2,3,4}
        	\foreach \y in {0,3,6,9,12,15,18,21}
        	\draw[fill=black, line width=1] (\y,\x) circle [radius=0.15] (-1,\x)node {$\scriptstyle{\x}$} (\y,.5)--(\y,4.5);
        	\foreach \x\y in {0/1, 0/2, 6/1, 18/1,18/2,18/3,18/4 }
        	\draw[black] (\x,\y) --(\x+3,\y);
        	\foreach \x\y in {0/3, 6/2, 12/1}
        	\draw (\x,\y) to[out=0,in=-90] (\x+1,\y+.5) to[out=90,in=0] (\x,\y+1);
        	\foreach \x\y in {3/3, 9/3, 15/3}
        	\draw (\x,\y) to[out=180,in=-90] (\x-1,\y+.5) to[out=90,in=180] (\x,\y+1);
        	\foreach \x\y in {6/4, 12/4, 12/3}
        	\draw (\x,\y) to[out=0, in=120] (\x+1.5, \y-1) to[out=300, in=180] (\x+3,\y-2);
        	\foreach \x in {3,9,15,21}
        	\draw[line width=0.2cm] (\x,.8)--(\x,3.2); 
        	\end{tikzpicture}
    \]
    The structure of~$\tl_n(a)\otimes_{\tl_m(a)} \t$ as a left module for~$\tl_n(a)$ is given by pasting diagrams on the left, and then simplifying, as in the following example for~$n=4$ and~$m=2$:
    \[
        	\begin{tikzpicture}[scale=0.5]
        	\foreach \x in {1,2,3,4}
        	\foreach \y in {3,6,9,12,15,18,21}
        	\draw[fill=black,line width=1]  (\y,\x) circle [radius=0.15]  
        	(\y,.5)--(\y,4.5);
        	\foreach \x\y in {3/1,12/1 }
        	\draw[black] (\x,\y) --(\x+3,\y);
        	\foreach \x\y in { 12/2,3/2, 9/1,9/3, 18/1,18/3}
        	\draw (\x,\y) to[out=0,in=-90] (\x+1,\y+.5) to[out=90,in=0] (\x,\y+1);
        	\foreach \x\y in {15/3,6/3, 12/1,12/3, 21/3,21/1}
        	\draw (\x,\y) to[out=180,in=-90] (\x-1,\y+.5) to[out=90,in=180] (\x,\y+1);
        	\foreach \x\y in {3/4, 12/4}
        	\draw (\x,\y) to[out=0, in=120] (\x+1.5, \y-1) to[out=300, in=180] (\x+3,\y-2);
        	\foreach \x in {6,15,21}
        	\draw[line width=0.2cm] (\x,.8)--(\x,2.2); 
        	\draw (1,2.5) node {$U_1U_3\,\cdot$} (7.5,2.5) node {$=$} (16.5,2.5) node {$=$}(22.5,2.5) node {$=0.$};
        	\end{tikzpicture}
    \]
\end{remark}

\begin{defn}[The ideal~$I_m$]
    Given~$0\leqslant m\leqslant n$, let~$I_m$ denote the left ideal of~$\tl_n(a)$ generated by the elements~$U_1,\ldots,U_{m-1}$.
\end{defn}

\begin{lemma}\label{lemma: tensor over subalgebra iso quotient by ideal}
   ~$\tl_n(a)\otimes_{\tl_m(a)}\t$ and~$\tl_n(a)/I_m$ are isomorphic as left~$\tl_n(a)$-modules via the maps
    \[
        \tl_n(a)\otimes_{\tl_m(a)}\t\longrightarrow\tl_n(a)/I_m,
        \qquad
        y\otimes r\longmapsto yr+I_m
    \]
    and
    \[
        \tl_n(a)/I_m\longrightarrow\tl_n(a)\otimes_{\tl_m(a)}\t,
        \qquad
        y+I_m\longmapsto y\otimes 1.
    \]    
\end{lemma}

\begin{proof}
    Observe that the generators~$U_1,\ldots,U_{m-1}$ of the left-ideal~$I_m$ in~$\tl_n$ are precisely the generators of the subalgebra~$\tl_m$ of~$\tl_n$.  
    Thus the map $y\otimes r\mapsto yr+I_m$ is well defined because if~$i=1,\ldots,m-1$ then elements of the form~$yU_i\otimes r$ and~$y\otimes U_ir$ both map to~$0$ in~$\tl_n/I_m$.
    And~$y+I_m\mapsto y\otimes 1$ is well defined because elements of~$I_m$ are linear combinations of ones of the form~$x\cdot U_i$ for~$i=1,\ldots,m-1$, and~$(x\cdot U_i)\otimes 1 = x\otimes (U_i\cdot 1) = x\otimes 0 = 0$ for~$i=1,\ldots,m-1$.
    One can now check that the two maps are inverses of one another.
\end{proof}

\begin{remark}
    Lemma~\ref{lemma: tensor over subalgebra iso quotient by ideal} justifies the description of $\tl_n(a)\otimes_{\tl_m(a)}\t$ in terms of diagrams with `black boxes' that we gave in Remark~\ref{remark: black boxes}.  
    Indeed,~$I_m$ is precisely the span of those diagrams which have a  cup on the right between the dots~$i$ and~$i+1$ for some~$i=1,\ldots,m-1$.  
    But these are precisely the diagrams which are made to vanish by having a cup fall into the black box.
    Thus~$\tl_n(a)/I_m$ has basis given by the remaining diagrams, i.e.~the ones that are not rendered~$0$ by the black box.
\end{remark}

\begin{lemma}\label{lem-basis for Im}
    For~$m\leqslant n$, the ideal~$I_m$ of~$\tl_n(a)$ has basis consisting of those elements of~$\tl_n(a)$ written in Jones normal form~$x_{\ul{a},\ul{b}}$, which have terminus~$b_1\leqslant m-1$ (and~$k\neq 0$).
\end{lemma}

\begin{proof}
    Recall that words of the form~$x_{\ul{a},\ul{b}}$ give a basis for~$\tl_n$. Then by definition any word~$w\in I_m$ is of the form~$w=x_{\ul{a},\ul{b}}v$ for~$v\in \langle U_{1}, \ldots , U_{m-1}\rangle$ and~$v\neq e$. Then we have that~$t(w)\leqslant m-1$. Now apply Lemma~\ref{lem-terminus of JNF} to~$w$ to complete the proof.
\end{proof}

\begin{lemma} \label{lem-basis for tensor product over smaller tl}
    For~$m \leqslant n$,~$\tl_n(a)\otimes_{\tl_m(a)}\t$ has basis given by~$x_{\ul{a},\ul{b}}\otimes \t$ such that the terminus~$b_1>m-1$.
\end{lemma}

\begin{proof}
    From Lemma~\ref{lemma: tensor over subalgebra iso quotient by ideal}~$\tl_n\otimes_{\tl_m}\t$ is isomorphic to~$\tl_n/I_m$. Then elements of the form~$x_{\ul{a},\ul{b}}$ give a basis for~$\tl_n$ and elements of the form~$x_{\ul{a},\ul{b}}$, which have terminus~$b_1\leqslant m-1$ give a basis for~$I_m$ by Lemma~\ref{lem-basis for Im}. Therefore a basis for the quotient is given by~$x_{\ul{a},\ul{b}}$ such that the terminus~$b_1>m-1$, and under the isomorphism in Lemma~\ref{lemma: tensor over subalgebra iso quotient by ideal} this gives the required basis.
\end{proof}

\begin{example}
    The Jones basis of~$\tl_3(a)$ is:
    \[
        1,
        \quad 
        U_2,
        \quad
        U_1U_2,
        \quad
        U_1,
        \quad 
        U_2U_1
        \quad
    \]
    So~$\tl_3(a)\otimes_{\tl_2(a)}\t$ has basis consisting of those elements whose terminus is strictly greater than~$1$, namely:
    \[
        1,
        \quad 
        U_2,
        \quad
        U_1U_2.
    \]
    (Recall that by convention the terminus of~$1$ is~$\infty$.) 
\end{example}

\begin{lemma}\label{lemma-cups}
    For~$m\leqslant n$, suppose that~$y\in\tl_n(a)$ and that~$y\cdot U_{m-1}$ lies in~$I_{m-1}$.  Then~$y\cdot U_{m-1}$ lies in~$I_{m-2}$.
\end{lemma}

\begin{proof}
   The product $y\cdot U_{m-1}$ is a linear combination of words ending with~$U_{m-1}$, i.e.~of words~$w$ with~$t(w)=m-1$.  
    By Lemma~\ref{lem-terminus of JNF}, this can be rewritten as a linear combination of Jones basis elements~$x_{\ul{a},\ul{b}}$ whose terminus satisfies~$t(x_{\ul{a},\ul{b}}) = m-1$ or~$t(x_{\ul{a},\ul{b}}) \leqslant m-3$.
    Since~$y\cdot U_{m-1}\in I_{m-1}$, this means that in fact no basis elements with terminus~$m-1$ remain after cancellation, and therefore all remaining words have terminus~$m-3$ or less, and so lie in~$I_{m-2}$.
\end{proof}

\subsection{Iwahori-Hecke algebras}\label{section-ih}

\begin{defn}[The Iwahori-Hecke algebra]
	Let~$n\geqslant 0$ and let~$q\in R^\times$.  The \emph{Iwahori-Hecke algebra}~$\H_n(q)$ of type~$A_{n-1}$ is  the algebra with generators
	\[
	    T_1,\ldots,T_{n-1}
	\]
	satisfying the following relations:
	\begin{itemize}
	    \item
	   ~$T_i T_j = T_jT_i$ for~$i\neq j\pm 1$
	    \item
	   ~$T_iT_jT_i = T_jT_iT_j$ for~$i=j\pm 1$ 
	    \item
	   ~$T_i^2 = (q-1)T_i + q$
	\end{itemize}
\end{defn}

\begin{defn}[From Iwahori-Hecke to Temperley-Lieb]\label{defn-IH to TL}
	Now suppose that there is~$v\in R^\times$ such that~$q=v^2$.
	Then there are two natural homomorphisms 
	\[
		\theta_1,\theta_2\colon
		\H_n(q)\longrightarrow \tl_n(v+v^{-1}),
	\]
	defined by~$\theta_1(T_i) = vU_i-1$ and~$\theta_2(T_i)=v^2-vU_i$
	for~$i=1,\ldots,n-1$.
	They induce isomorphisms
	\[
	    \bar\theta_1\colon
	    \H_n(q)/I_1\xrightarrow{\ \cong\ } \tl_n(v+v^{-1}),
	    \qquad
	    \bar\theta_2\colon
	    \H_n(q)/I_2\xrightarrow{\ \cong\ } \tl_n(v+v^{-1}),
	\]
	where~$I_1$ is the two-sided ideal generated by elements of the form
	\[
	    T_iT_jT_i+T_iT_j + T_jT_i + T_i + T_j + 1
	\]
	for~$i=j\pm 1$, and~$I_2$ is the two-sided ideal generated by elements of the form
	\[
	    T_iT_jT_i - q T_iT_j  -q T_jT_i + q^2T_i + q^2T_j -q^3
	\]
	for~$i=j\pm 1$.
	See~\cite{FanGreen}, Theorem~5.29 of~\cite{KasselTuraev},
	and Section~2.3 of~\cite{HMR}, though unfortunately conventions change from author to author. Another standard convention of setting~$a=-(v+v^{-1})$ can easily be accounted for by swapping~$v$ with~$-v^{\pm 1}$.
\end{defn}

We will take an agnostic approach to the homomorphisms~$\theta_1$,~$\theta_2$.
We will choose one of them and denote it by simply
\[
	\theta\colon \H_n(q)\longrightarrow\tl_n(v+v^{-1}),
\]
and denote by~$\lambda$ the constant term in~$\theta(T_i)$, and by~$\mu$ the coefficient of~$U_i$ in~$\theta(T_i)$, so that
\[
    \theta(T_i) = \lambda + \mu U_i.
\]
Then~$\theta$ induces an isomorphism
\[
	\bar\theta\colon\H_n(q)/I\xrightarrow{\ \cong\ }\tl_n(v+v^{-1})
\]
where~$I$ is the two-sided ideal generated by elements of the form
\[
	T_iT_jT_i - \lambda T_iT_j-\lambda T_jT_i 
	+\lambda^2 T_i + \lambda^2 T_j-\lambda^3
\]
for~$i=j\pm 1$.  
And moreover, the elements~$\theta(T_i)$ act on the trivial module~$\t$ as multiplication by~$\lambda$.

\begin{defn}\label{definition-si}
    Let~$v\in R^{\times}$. We define~$s_1,\ldots,s_{n-1}\in\tl_n(v+v^{-1})$ by setting
    \[s_i=\theta(T_i)=\lambda+\mu U_i\] and note that these elements satisfy the following properties:
    \begin{itemize}
    	\item
    	$s_i^2 = (v^2-1)s_i + v^2$ for all~$i$,
    	\item
    	$s_i s_j = s_js_i$ for~$i\neq j\pm 1$,
    	\item
    	$s_is_js_i = s_js_is_j$ for~$i=j\pm 1$,
    	\item
    	$s_is_js_i - \lambda s_is_j-\lambda s_js_i +\lambda^2 s_i + \lambda^2 s_j-\lambda^3=0$ for~$i=j\pm 1$,
    	\item
    	$s_i$ acts on~$\t$ as multiplication by~$\lambda$.
    \end{itemize}
\end{defn}

\begin{remark}\label{remark-smoothing}
    There is a homomorphism from (the group algebra of) the braid group into~$\tl_n(v+v^{-1})$ given on generators by~$s_i\mapsto s_i$.  
    This is the content of the second and third bullet points above, together with the fact that the~$s_i$ are invertible, \red{which follows from the first bullet point (and the fact that~$v$ is a unit)}.
    Diagrammatically, this homomorphism can be viewed as a \emph{smoothing expansion} from braided diagrams to planar diagrams: take a braid diagram, and then smooth each crossing in turn in the two possible ways, using appropriate weightings for each smoothing.
    For example, we can visualise the image of~$s_i$ in $\tl_n(v+v^{-1})$ as the standard braid group generator crossing strand~$i$ over strand~$i+1$. There are two ways this crossing can be resolved to a planar diagram, and we equate~$s_i$ to the sum of these two states. They are the identity and~$U_i$, as shown in Figure~$\ref{fig: smoothing}$.
    \begin{figure}
        \centering
        	\begin{tikzpicture}[scale=0.4]
        	\foreach \x in {1,3,4,5,6,8}
        	\foreach \y in {0,3,6,9,12,15}
        	\draw[fill=black, line width=1] (\y,\x) circle [radius=0.15] (\y, 0.5)--(\y, 8.5);
        	\foreach \x in {1,3,8}
        	\draw[black] (0,\x) --(3,\x) (6,\x)--(9,\x) (12,\x)--(15,\x);
        	\foreach \x in {1.5,7.5,13.5}
        	\draw (\x,2.2) node {$\scriptstyle{\vdots}$} (\x,7.2) node {$\scriptstyle{\vdots}$};
        	\draw (-1,1)node {$\scriptstyle{1}$} 
        	(-1,4) node {$\scriptstyle{i}$}  (-1,5) node {$\scriptstyle{i+1}$}  (-1,8) node {$\scriptstyle{n}$};
        	\foreach \x in {12}
        	\draw (\x,4) to[out=0,in=-90] (\x+1,4.5) to[out=90,in=0] (\x,5);
        	\foreach \x in {15}
        	\draw (\x,4) to[out=180,in=-90] (\x-1,4.5) to[out=90,in=180] (\x,5);
        	\draw (6,6) -- (9,6) (12,6)--(15,6) (9,4)--(6,4) (0,6)--(3,6) (6,5)--(9,5);
        	\draw (10.5, 4.5) node[scale=1] {$+\,\mu$};
        	\draw (4.5, 4.5) node[scale=1] {$=\, \lambda$};    	
        	\draw (0,5) to[out=0, in=140] (1.5, 4.5) to[out=320, in=180] (3,4);
        	\draw[white,fill=white] (1.5,4.5) circle [radius=0.15];
        	\draw (0,4) to[out=0, in=220] (1.5, 4.5) to[out=40, in=180] (3,5);
        	\draw (1.5,-.8) node[] {$s_i$}(4.5,-.8) node[] {$= \, \lambda$}(10.5,-.8) node[] {$+\mu$}(13.5,-.8) node[] {$U_i$};
        	\end{tikzpicture}
        \caption{Smoothings of~$s_i$}
        \label{fig: smoothing}
    \end{figure}
    The coefficient of the identity is~$\lambda$ and the coefficient of~$U_i$ is~$\mu$, simply because we defined~$s_i=\lambda+\mu U_i$.
    Similarly, we consider the image of~$s_i^{-1}$ as strand~$i$ crossing under strand~$i+1$, and when this is smoothed the coefficient of the identity is~$\lambda^{-1}$ and the coefficient of~$U_i$ is~$\mu^{-1}$, precisely because one can verify that~$s_i^{-1} = \lambda^{-1}+\mu^{-1}U_i$ in~$\tl_n(v+v^{-1})$.
    
    In principle we could describe how various Reidemeister moves affect the smoothing expansion, but it will not be necessary for the rest of the paper.
    Moreover, we will only encounter positive powers of~$s_i$.
\end{remark}

\section{Inductive resolutions}
\label{section-inductive}
In this section we prove the following two theorems, which we recall from the introduction.

\setcounter{abcthm}{0}
\begin{abcthm}
    Let~$R$ be a commutative ring and let~$a$ be a unit in~$R$.  Then $\Tor^{\tl_n(a)}_d(\t,\t)$ and~$\Ext_{\tl_n(a)}^d(\t,\t)$ both vanish for~$d>0$.
\end{abcthm}

\setcounter{abcthm}{5}
\begin{abcthm}\label{thm-inductive res}
    Let~$R$ be a commutative ring and let~$a\in R$.
    Let~$0\leqslant m<n$.  
    Then~$\Tor^{\tl_n(a)}_d(\t,\tl_n(a)\otimes_{\tl_{m}(a)}\t)$ and~$\Ext_{\tl_n(a)}^d(\tl_n(a)\otimes_{\tl_{m}(a)}\t,\t)$ vanish for~$d>0$.
\end{abcthm}

In fact for Theorem~\ref{theorem-invertible} we will prove the following stronger claim:

\begin{claim}\label{claim-invertible}
    Suppose that the parameter~$a\in R$ is invertible.
    Then for any $0\leqslant m\leqslant n$, the groups~$\Tor^{\tl_n(a)}_d(\t,\tl_n(a)\otimes_{\tl_m(a)}\t)$ and~$\Ext^{\tl_n(a)}_d(\tl_n(a)\otimes_{\tl_m(a)}\t,\t)$ both vanish for~$d>0$.      
\end{claim}

The similarity between Theorem~\ref{thm-inductive res} and Claim~\ref{claim-invertible} is now clear.  Both will be proved by induction on~$m$, the initial cases~$m=0,1$ being immediate because then~$\tl_m$ is the ground ring~$R$ so that~$\tl_n\otimes_{\tl_m}\t\cong\tl_n$ is free.  In order to produce an inductive proof, we construct resolutions of the modules~$\tl_n\otimes_{\tl_m}\t$ whose terms are not free or projective or injective, but instead whose terms are the modules considered earlier in the induction, specifically~$\tl_n\otimes_{\tl_{m-1}}\t$ and~$\tl_n\otimes_{\tl_{m-2}}\t$.  For this reason we refer to these resolutions as \emph{inductive resolutions}.  This approach is inspired by homological stability arguments, in which one considers complexes whose building blocks are induced up from the earlier objects in the sequence.  The difference here is that our complexes are actual resolutions --- they are acyclic rather than just acyclic up to a point --- and because Shapiro's lemma is unavailable we do not change the algebra we are working over, rather we change the algebra from which we are inducing our modules.

\subsection{The inductive resolutions}

In this subsection we establish the resolutions~$C(m)$ and~$D(m)$ of~$\tl_n\otimes_{\tl_m}\t$ required to prove Claim~\ref{claim-invertible} and Theorem~\ref{thm-inductive res} above.

\begin{defn}[The complex~$C(m)$]
    \red{Let~$2\leqslant m\leqslant n$ and assume that~$a$ is invertible.
   	We define a chain complex of left~$\tl_n(a)$-modules as in Figure~\ref{fig: Cm}.}  The degree is indicated in the right-hand column.
    \begin{figure}
        \centering
    	$
            \xymatrix{
                \vdots
                \ar[d]^{(1-a^{-1}U_{m-1})}
                \\
                \tl_n\otimes_{\tl_{m-2}}\t
                \ar[d]^{a^{-1}U_{m-1}}
                &
                3
                \\
                \tl_n\otimes_{\tl_{m-2}}\t
                \ar[d]^{(1-a^{-1}U_{m-1})}
                &
                2
                \\
                \tl_n\otimes_{\tl_{m-2}}\t
                \ar[d]^{a^{-1}U_{m-1}}
                &
                1
                \\
                \tl_n\otimes_{\tl_{m-1}}\t
                \ar[d]^1
                &
                0
                \\
                \tl_n\otimes_{\tl_m}\t
                &
                -1
            }
       ~$
        \caption{The complex~$C(m)$.}
        \label{fig: Cm}
    \end{figure}
    The differentials of~$C(m)$ are all given by extending the algebra over which the tensor product is taken, by right multiplying in the first factor by the indicated element of~$\tl_n(a)$, or by a combination of the two.  So, for example, the differential originating in degree~$1$ sends~$x\otimes r\in\tl_n(a)\otimes_{\tl_{m-2}(a)}\t$ to~$(x\cdot a^{-1} U_{m-1})\otimes r\in\tl_n(a)\otimes_{\tl_{m-1}(a)}\t$. 
    The complex is periodic of period~$2$ in degrees~$1$ and above, so that all entries are~$\tl_n(a)\otimes_{\tl_{m-2}(a)}\t$ and the boundary maps between them alternate between~$a^{-1}U_{m-1}$ and~$(1-a^{-1}U_{m-1})$.
    The boundary maps are well defined because~$U_{m-1}$ commutes inside~$\tl_n(a)$ with all elements of~$\tl_{m-2}(a)$.
\end{defn}

\begin{defn}[The complex~$D(m)$]
    \red{Let~$2\leqslant m <n$, and do not assume that~$a$ is invertible.
   We define a chain complex of left~$\tl_n(a)$-modules as in Figure~\ref{fig: Dm}.}  The degree is indicated in the right-hand column.
    \begin{figure}
        \centering
    	$
            \xymatrix{
                \vdots
                \ar[d]^{(1-U_{m-1}U_m)}
                \\
                \tl_n\otimes_{\tl_{m-2}}\t
                \ar[d]^{U_{m-1}U_m}
                &
                3
                \\
                \tl_n\otimes_{\tl_{m-2}}\t
                \ar[d]^{(1-U_{m-1}U_m)}
                &
                2
                \\
                \tl_n\otimes_{\tl_{m-2}}\t
                \ar[d]^{U_{m-1}}
                &
                1
                \\
                \tl_n\otimes_{\tl_{m-1}}\t
                \ar[d]^1
                &
                0
                \\
                \tl_n\otimes_{\tl_m}\t
                &
                -1
            }
       ~$
        \caption{The complex~$D(m)$.}
        \label{fig: Dm}
    \end{figure}
    The differentials of~$D(m)$ are all given by extending the algebra over which the tensor product is taken, by right multiplying in the first factor by the indicated element of~$\tl_n(a)$, or by a combination of the two.  
    So, for example, the differential originating in degree~$1$ sends~$x\otimes r\in\tl_n(a)\otimes_{\tl_{m-2}(a)}\t$ to~$x\cdot U_{m-1}\otimes r\in\tl_n(a)\otimes_{\tl_{m-1}(a)}\t$. 
    The complex is periodic of period~$2$ in degrees~$1$ and above, so that in that range all terms are~$\tl_n(a)\otimes_{\tl_{m-2}(a)}\t$ and the boundary maps between them alternate between~$U_{m-1}U_m$ and~$(1-U_{m-1}U_m)$.
    The boundary maps are well defined because~$U_{m-1}$ and~$U_{m-1}U_m$ commute inside~$\tl_n(a)$ with all elements of~$\tl_{m-2}(a)$.
    Observe that the condition~$m<n$ is necessary in order to ensure that~$U_m$ is actually an element of~$\tl_n(a)$.
\end{defn}

\begin{lemma}\label{lemma-idempotent}\hfill
    \begin{enumerate}
        \item
        Let~$2\leqslant m\leqslant n$ and let~$a$ be invertible.
        Then~$a^{-1}U_{m-1}\in\tl_n(a)$ is idempotent.
        \item
        Let~$2\leqslant m <n$ and let~$a$ be arbitrary.
        Then~$U_{m-1}U_m\in \tl_n(a)$ is idempotent.
    \end{enumerate}
\end{lemma}

\begin{proof}
    We calculate
   \[(a^{-1}U_i)^2 = a^{-2}U_i^2 = a^{-2}aU_i = a^{-1}U_i\]
    and 
   \[U_{m-1}U_m\cdot U_{m-1}U_m = U_{m-1}U_mU_{m-1}\cdot U_m=U_{m-1}U_m.\qedhere\]
\end{proof}

From now on in this section, we will attempt to talk about~$C(m)$ and~$D(m)$ at the same time.  
When we refer to~$C(m)$, the relevant assumptions should be understood, namely that~$2\leqslant m\leqslant n$ and that~$a\in R$ is a unit.  
And when we refer to~$D(m)$, the assumptions that~$2\leqslant m<n$ but~$a\in R$ is arbitrary should be understood.
We trust that this will not be confusing.

\begin{lemma}
   $C(m)$ and~$D(m)$ are indeed chain complexes.
\end{lemma}

\begin{proof}
    We give the proof for~$C(m)$.  The proof for~$D(m)$ is similar.
    We must check that consecutive boundary maps of~$C(m)$ compose to~$0$.
    In the case of the composite from degree~$1$ to~$-1$, the composition is given by
    \[x\otimes r\mapsto (x\cdot a^{-1}U_{m-1})\otimes r = x\otimes (a^{-1}U_{m-1}\cdot r)=x\otimes 0 = 0\] this holds because the tensor product is over~$\tl_m$, which contains~$a^{-1}U_{m-1}$.       
    In the case of the remaining composites, this follows immediately from
    \[(a^{-1}U_{m-1})\cdot (1-a^{-1}U_{m-1}) = 0 = (1-a^{-1}U_{m-1})\cdot(a^{-1}U_{m-1}),\]
    which is a consequence of the fact that~$a^{-1}U_{m-1}$ is idempotent (from Lemma~\ref{lemma-idempotent}).
\end{proof}
    
\begin{lemma}\label{lemma-CD-acyclic}
    The complexes~$C(m)$ and~$D(m)$ are acyclic.
\end{lemma}
    
\begin{proof}
    In degree~$-1$ it is clear that the boundary map is surjective, for both~$C(m)$ and~$D(m)$.
    
    In degree~$0$, we will give the proof for~$C(m)$, the proof for~$D(m)$ being similar.
    Suppose that~$y\otimes 1\in\tl_n\otimes_{\tl_{m-1}}\t$ lies in the kernel of the boundary map, or in other words that~$y\otimes 1\in\tl_n\otimes_{\tl_m}\t$ vanishes.  This means that~$y$ lies in the left-ideal generated by the elements~$U_1,\ldots,U_{m-1}$.  
    Since all but the last of these generators lie in~$\tl_{m-1}$, and we started with~$y\otimes 1\in\tl_n\otimes_{\tl_{m-1}}\t$, we may assume without loss that~$y=y'\cdot U_{m-1}$ for some~$y'$.
    But then
    \[y\otimes 1 = y'\cdot U_{m-1}\otimes 1 = ay'\cdot (a^{-1}U_{m-1})\otimes 1\] does indeed lie in the image of the boundary map.
    
    In degree~$1$, we give the proof for both complexes.  First, for~$C(m)$, suppose that~$y\otimes 1 \in\tl_n\otimes_{\tl_{m-2}}\t$ lies in the kernel of the boundary map.  
    It follows that $y\cdot (a^{-1}U_{m-1})\otimes 1$ vanishes in~$\tl_n\otimes_{\tl_{m-1}}\t$, which means that $y\cdot (a^{-1}U_{m-1})$ lies in the left ideal~$I_{m-1}$ generated by~$U_1,\ldots,U_{m-2}$.  It follows from Lemma~\ref{lemma-cups} that~$y\cdot (a^{-1}U_{m-1})$ lies in the left ideal~$I_{m-2}$ generated by~$U_1,\ldots,U_{m-3}$, so that in~$\tl_n\otimes_{\tl_{m-2}}\t$ the element~$y\cdot (a^{-1}U_{m-1})\otimes 1$ vanishes.  Thus
    \[y\otimes 1 =y\cdot(1-a^{-1}U_{m-1})\otimes 1\] does indeed lie in the image of the boundary map.
    Second, for~$D(m)$, suppose that~$y\otimes 1\in\tl_n\otimes_{\tl_{m-2}}\t$ lies in the kernel of the boundary map.
    Then, as for~$C(m)$,~$y\cdot U_{m-1}$ lies in~$I_{m-2}$. So~$y\cdot U_{m-1}U_m$ also lies in the left ideal~$I_{m-2}$ since~$U_m$ commutes with the generators of~$I_{m-2}$.
    Thus~$y\cdot U_{m-1}U_m\otimes 1$ vanishes in~$\tl_n\otimes_{\tl_{m-2}}\t$, so that~$y\otimes 1 = y\cdot (1- U_{m-1}U_m)\otimes 1$ does indeed lie in the image of the boundary map.
    
    In degrees~$2$ and higher, acyclicity is an immediate consequence of the fact that~$a^{-1}U_{m-1}$ and~$U_{m-1}U_m$ are idempotents, by Lemma~\ref{lemma-idempotent}.
\end{proof}

\begin{lemma}\label{lemma-tensor-acyclic}
    The complexes~$\t\otimes_{\tl_n(a)}C(m)$,~$\t\otimes_{\tl_n(a)} D(m)$,~$\Hom_{\tl_n(a)}(C(m),\t)$ and~$\Hom_{\tl_n(a)}(D(m),\t)$  are acyclic.
\end{lemma}

\begin{proof}
    We give the proof for~$\t\otimes_{\tl_n}C(m)$, the proof for the other parts being similar.
    The terms of~$C(m)$ have the form~$\tl_n\otimes_{\tl_{m-i}}\t$ where~$i=0,1,2$ depending on the degree.  Thus~$\t\otimes_{\tl_n}C(m)$ has terms of the form~$\t\otimes_{\tl_n}(\tl_n\otimes_{\tl_{m-i}}\t)\cong\t\otimes_{\tl_{m-i}}\t\cong\t$.  Moreover, by tracing through this isomorphism, one sees that if a boundary map in~$C(m)$ is labelled by an element~$x\in\tl_n$, then the corresponding boundary map in~$\t\otimes_{\tl_n}C(m)$ is simply the map~$\t\to\t$ given by the action of~$x$ on~$\t$.  Thus~$\t\otimes_{\tl_n}C(m)$ is nothing other than the complex in Figure~\ref{fig: tensor complex}.  (The right hand column indicates the degree.)
    \begin{figure}
    \centering
    	\[
        \xymatrix{
            \vdots
            \ar[d]
            \\
            \t
            \ar[d]^{0}
            &
            3
            \\
            \t
            \ar[d]^{1}
            &
            2
            \\
            \t
            \ar[d]^{0}
            &
            1
            \\
            \t
            \ar[d]^{1}
            &
            0
            \\
            \t
            &
            -1
        }
    \]
    \caption{The complex~$\t\otimes C(m)$}
    \label{fig: tensor complex}
    \end{figure}
    This is visibly acyclic, and this completes the proof.
\end{proof}

\begin{remark}[Representation theory and the inductive resolutions]
	Schur-Weyl duality relates representations of $\tl_n$ 
	with representations of the quantum group $\uq$,
	and it is possible to use this to construct our inductive resolutions
	via the representation theory of $\uq$.
	We will try to describe this briefly.
	We are indebted to an anonymous referee for explaining this connection
	to us.
	
	One instance of Schur-Weyl duality is the following.
	Let $V$ denote the standard representation of 
	$\uq$.	
	Then there is an isomorphism
	$\tl_n\cong \mathrm{End}_{\uq}(V^{\otimes n})$,
	and more generally there are isomorphisms 
	$\tl(n,m)\cong\mathrm{Hom}_\uq(V^{\otimes n},V^{\otimes m})$
	that assemble into a monoidal functor on the 
	\emph{Temperley-Lieb category} $\tl$.
	(The objects of $\tl$ are the non-negative integers, the morphism
	space
	$\tl(n,m)$ is the $R$-module spanned by planar diagrams with $n$
	marked points on the left and $m$ marked points on the right,
	and composition is defined just like multiplication in $\tl_n$.)
	See Webster's appendix to~\cite{EliasLibendinsky}.

	One can write down exact sequences of $\uq$-modules that,
	after applying Schur-Weyl duality, yield the inductive resolutions
	$C(m)$ and $D(m)$.
	We will not detail the construction of these sequences,
	except to say that each one relies on the construction of 
	an appropriate splitting of some tensor power of $V$.
	The relevant splittings are constructed in each
	case as follows:
	\begin{itemize}
		\item	
		In the case where $a$ is invertible, the morphisms
		\[
			a^{-1}\,
			\begin{tikzpicture}[scale=0.6,baseline=23]
			\foreach \x in {1,2}
			\foreach \y in {2}
				\draw[fill=black,line width=1]  (\y,\x) circle [radius=0.1];
			\draw[fill=black, line width=1]
				(2,.5) -- (2,2.5);
			\draw[fill=black, line width=1]
				(1,.5) -- (1,2.5);
			\draw[line width=1]
				(2,2) arc (90:270:0.5);	
			\end{tikzpicture}
			\quad\text{and}\quad
			\begin{tikzpicture}[scale=0.6,baseline=23]
			\foreach \x in {1,2}
			\foreach \y in {1}
				\draw[fill=black,line width=1]  (\y,\x) circle [radius=0.1]  
				(\y,.5)--(\y,2.5);
			\draw[fill=black, line width=1]
				(2,.5) -- (2,2.5);
			\draw[line width=1]
				(1,1) arc (-90:90:0.5);	
			\end{tikzpicture}
		\]
		in $\tl$ compose to give the identity morphism 
		in $\tl(0,0)$.
		(The two semicircles compose to the circle morphism 
		from $0$ to itself,
		and by the usual rule for composing diagrams, 
		the circle morphism is $a$ times the identity.)
		This then corresponds to a pair of maps 
		$R=V^{\otimes 0}\to V^{\otimes 2}$ 
		and $V^{\otimes 2}\to V^{\otimes 0}=R$
		that compose to the identity, showing that 
		$V^{\otimes 2}$ splits off a copy of $R$.
		Note that the map
		$V^{\otimes 2}\to V^{\otimes 2}$
		that projects onto this copy of $R$ is 
		represented by the morphism
		\[
			a^{-1}
			\begin{tikzpicture}[scale=0.6,baseline=23]
			\foreach \x in {1,2}
			\foreach \y in {2}
				\draw[fill=black,line width=1]  (\y,\x) circle [radius=0.1];
			\draw[fill=black, line width=1]
				(2,.5) -- (2,2.5);
			\draw[line width=1]
				(2,2) arc (90:270:0.5);	
			\draw[line width=1]
				(0,1) arc (270:450:0.5);	
			\foreach \x in {1,2}
			\foreach \y in {0}
				\draw[fill=black,line width=1]  (\y,\x) circle [radius=0.1]  
				(\y,.5)--(\y,2.5);
			\end{tikzpicture}
		\]
		in $\tl$.  Compare this with the
		idempotent $a^{-1}U_{m-1}$ appearing in $C(n)$.

		\item
		When $a$ is not invertible, we consider the morphisms
		\[
			\begin{tikzpicture}[scale=0.6,baseline=32]
			\foreach \x in {2}
			\foreach \y in {1}
				\draw[fill=black,line width=1]  (\y,\x) circle [radius=0.1];
			\foreach \x in {1,2,3}
			\foreach \y in {3}
				\draw[fill=black,line width=1]  (\y,\x) circle [radius=0.1];
			\draw[fill=black, line width=1]
				(3,.5) -- (3,3.5);
			\draw[fill=black, line width=1]
				(1,.5) -- (1,3.5);
			\draw[line width=1]
				(3,3) arc (90:270:0.5);	
			\draw[line width=1]
				(1,2) to[out=0,in=180] (3,1);
			\end{tikzpicture}
			\quad\text{and}\quad
			\begin{tikzpicture}[scale=0.6,baseline=32]
			\foreach \x in {2}
			\foreach \y in {3}
				\draw[fill=black,line width=1]  (\y,\x) circle [radius=0.1]  
				(\y,.5)--(\y,3.5);
			\foreach \x in {1,2,3}
			\foreach \y in {1}
				\draw[fill=black,line width=1]  (\y,\x) circle [radius=0.1]  
				(\y,.5)--(\y,3.5);
			\draw[line width=1]
				(1,1) arc (-90:90:0.5);	
			\draw[line width=1]
				(1,3) to[out=0,in=180] (3,2);
			\end{tikzpicture}
		\]
		which compose to give the identity morphism
		in $\tl(1,1)$.
		These diagrams correspond to a pair of maps
		$V\to V^{\otimes 3}\to V$ that compose to
		the identity, showing that $V^{\otimes 3}$
		splits off a copy of $V$.
		Observe that the map 
		$V^{\otimes 3}\to V^{\otimes 3}$ that projects
		to this copy of $V$ is represented by the
		morphism
		\[
			\begin{tikzpicture}[scale=0.6,baseline=44]
			\foreach \x in {1,2,3}
			\foreach \y in {5}
				\draw[fill=black,line width=1]  (\y,\x) circle [radius=0.1]  
				(\y,.5)--(\y,3.5);
			\foreach \x in {1,2,3}
			\foreach \y in {1}
				\draw[fill=black,line width=1]  (\y,\x) circle [radius=0.1]  
				(\y,.5)--(\y,3.5);
			\draw[line width=1]
				(1,1) arc (-90:90:0.5);	
			\draw[line width=1]
				(5,3) arc (90:270:0.5);	
			\draw[line width=1]
				(1,3) to[out=0,in=180] (5,1);
			\end{tikzpicture}
		\]
		which can be compared to the idempotent
		$U_{m-1}U_m$ appearing in $D(n)$.
	\end{itemize}
\end{remark}

\subsection{The spectral sequence of a double complex}\label{sec: double complex ss}
Since the spectral sequence of a particular kind of double complex is used several times during this paper, we introduce and discuss it in this subsection.

We begin with the homological version.
Suppose we have a chain complex~$Q_\ast$ of left~$\tl_n$-modules, such as~$C(m)$ or~$D(m)$, or the complex of planar injective words~$W(n)$ to be introduced later.  Then we choose a projective resolution~$P$ of~$\t$ as a right module over~$\tl_n$, and we consider the double complex~$P_\ast\otimes_{\tl_n}Q_\ast$.
This is a homological double complex in the sense that both differentials reduce the grading.
Associated to this double complex are two spectral sequences,~$\{\IE^r\}$ and~$\{\IIE^r\}$, which both converge to the homology of the totalisation, $H_\ast(\mathrm{Tot}(P_\ast\otimes_{\tl_n}Q_\ast))$ as in Section~5.6 of~\cite{Weibel}.
The first spectral sequence has 
$E^1$-term given by 
$
    \IE^1_{i,j} 
    = H_j(P_i\otimes_{\tl_n}Q_\ast)
    \cong P_i\otimes_{\tl_n} H_j(Q_\ast)
$
with~$d^1\colon \IE^1_{i,j}\to \IE^1_{i-1,j}$ induced by the differential~$P_i\to P_{i-1}$.
The isomorphism above holds because each~$P_i$ is projective and therefore flat.
It follows that the~$E^2$-term is
\[
    \IE^2_{i,j} = \Tor_i^{\tl_n}(\t,H_j(Q_\ast)).
\]
The second spectral sequence has~$E^1$-term given by~$\IIE^1_{i,j} = H_j(P_\ast\otimes_{\tl_n}Q_i)$, i.e.
\[
    \IIE^1_{i,j} =  \Tor^{\tl_n}_j(\t,Q_i)
\]
with~$d^1\colon \IIE^1_{i,j}\to \IIE^1_{i-1,j}$ induced by the boundary maps of~$Q_\ast$.

We now consider the cohomological version.
Suppose we have a chain complex~$Q_\ast$ of left~$\tl_n$-modules, again such as~$C(m)$,~$D(m)$ or~$W(n)$ (the latter to be introduced later).
Then we choose an injective resolution~$I^\ast$ of~$\t$ as a left module over~$\tl_n$, and we consider the double complex~$\Hom_{\tl_n}(Q_\ast,I^\ast)$.
This is a cohomological double complex in the sense that both differentials increase the grading.
Associated to this double complex are two spectral sequences,~$\{\IE_r\}$ and~$\{\IIE_r\}$, both converging to the cohomology of the totalisation,~$H^\ast(\mathrm{Tot}(\Hom_{\tl_n}(Q_\ast,I^\ast)))$ as in Section~5.6 of~\cite{Weibel}.
The first spectral sequence has~$E_1$-term given by 
$
    \IE_1^{i,j} 
    = H^j(\Hom_{\tl_n}(Q_\ast,I^i))
    \cong 
    \Hom_{\tl_n}(H_j(Q_\ast),I^i)
$
with~$d^1\colon \IE^1_{i,j}\to \IE^1_{i+1,j}$ induced by the differential of~$I^\ast$.
The isomorphism above holds because each~$I^i$ is injective, so that the functor~$\Hom_{\tl_n}(-,I^i)$ is exact.
It follows that the~$E_2$-term is
\[
    \IE_2^{i,j} = \Ext_{\tl_n}^i(\red{H_j(Q_\ast),\t}).
\]
The second spectral sequence has~$E_1$-term~$\IIE_1^{i,j} = H^j(\Hom_{\tl_n}(Q_i,I^\ast))$, i.e.
\[
    \IIE_1^{i,j}  =\Ext_{\tl_n}^j(\red{Q_i, \t})
\]
with differential~$d^1\colon \IIE_1^{i,j}\to \IIE_1^{i+1,j}$ induced by the differential of~$Q_\ast$.

\subsection{Proof of Theorems~\ref{theorem-invertible} and~\ref{theorem-shapiro}}
We can now prove Claim~\ref{claim-invertible} (which implies Theorem~\ref{theorem-invertible}) and Theorem~\ref{theorem-shapiro}.
The proofs of the two results will be almost identical except that the former uses the complex~$C(m)$ and the latter uses the complex~$D(m)$.
Moreover, each result has a homological and cohomological part, referring to~$\Tor$ and~$\Ext$ respectively.
In each case the two parts are proved similarly, by using either the homological or cohomological spectral sequence from Section~\ref{sec: double complex ss} above.
We will therefore only prove the homological part of Claim~\ref{claim-invertible}, i.e.~we will prove that $\Tor^{\tl_n}_\ast(\t,\tl_n\otimes_{\tl_m}\t)$ vanishes in positive degrees, leaving the details of the other parts to the reader.

\begin{proof}[{Proof of Claim~\ref{claim-invertible}, Tor case.}]
We prove the claim by fixing~$n$ and using strong induction on~$m$ in the range~$n\geqslant m\geqslant 0$.
As noted before, the initial cases~$m=0,1$ of the induction are immediate since then~$\tl_m$ is the ground ring and~$\tl_n\otimes_{\tl_n}\t\cong\tl_n$ is free.
We therefore fix~$m$ in the range~$2\leqslant m\leqslant n$.

We now employ the homological spectral sequences~$\{\IE^r\}$ and~$\{\IIE^r\}$ of Section~\ref{sec: double complex ss}, in the case~$Q=C(m)$.  
Then~$\IE^2_{i,j}=\Tor^{\tl_n}_i(\t,H_j(C(m)))=0$ for all~$i$ and~$j$, since~$C(m)$ is acyclic by Lemma~\ref{lemma-CD-acyclic}.
Thus~$\{\IE^r\}$ converges to zero, and the same must therefore be true of~$\{\IIE^r\}$, since both spectral sequences have the same target.
In the second spectral sequence the~$E^1$-page
\[
    \IIE^1_{i,j} = \Tor^{\tl_n}_j(\t,C(m)_i)
\]
is largely known to us.  
The bottom~$j=0$ row of~$\IIE^1$ is precisely the complex $\t\otimes_{\tl_n}C(m)$, which is acyclic by Lemma~\ref{lemma-tensor-acyclic}.
And when~$i\geqslant 0$, the term~$C(m)_i$ is either~$\tl_n\otimes_{\tl_{m-1}}\t$ or~$\tl_n\otimes_{\tl_{m-2}}\t$, and our inductive hypothesis applies to these ($(m-1)<m$ and~$(m-2)<m$) to show that~$\IIE^1_{i,j} = \Tor^{\tl_n}_j(\t,C(m)_i)=0$ when~$j>0$.
See Figure~\ref{fig:second spectral sequence E1 page INVERTIBLE} for a visualisation of the~$E^1$ page.
Altogether, this tells us that~$\IIE^2_{i,j}$ vanishes except for the groups 
\[
    \IIE^2_{-1,j}
    =
    \Tor^{\tl_n}_j(\t,C(m)_{-1})
    =
    \Tor^{\tl_n}_j(\t,\tl_n\otimes_{\tl_m}\t)
\]
for~$j>0$, which are concentrated in a single column and therefore not subject to any further differentials.  
Thus~$\IIE^2=\IIE^\infty$. 
But we know that~${}^{II}E^\infty$ vanishes identically, so that the inductive hypothesis is proved, and so, therefore, is the proof of the homological part of Claim~\ref{claim-invertible}.
\end{proof}
\begin{figure}
		\begin{tikzpicture}[scale=0.7]
			\draw[<->] (0,6)--(0,.5)--(12,.5);
			\draw[line width=0.2cm, gray!20, <->] (-4.5,6)--node[black, above, pos=0] {$j$}(-4.5,-1)--(12,-1) node[black, right] {$i$};
			\draw (-2,-1) node {$-1$};
			\draw (1.5,-1) node {$0$};	
			\draw (4.75,-1) node {$1$};
			\draw (8,-1) node {$2$};
			\foreach \x in {0,1,2,3}	
			\draw (-4.5,\x) node {$\x$};
			\foreach \x in {(-4.5,4.5),(-2,4.5)}	
			\draw \x node {$\vdots$};
			\foreach \x in {(11,-1), (11,0)}
			\draw \x node {$\cdots$};
			\foreach \x in {1,2,3}
			\draw (-2,\x) node {$\scriptstyle{\Tor^{\tl_n}_{\x}(\t,C(m)_{-1})}$};
			\draw (-2,0) node {$\scriptstyle{\t \otimes_{\tl_n} C(m)_{-1}}$};
			\draw (1.25,0) node {$\scriptstyle{\t \otimes_{\tl_n} C(m)_{0}}$};
			\draw (4.5,0) node {$\scriptstyle{\t \otimes_{\tl_n} C(1)_{1}}$};
			\draw (7.75,0) node {$\scriptstyle{\t \otimes_{\tl_n} C(m)_{2}}$};
			\draw (4,2) node[black!80,scale=1.2] {$\Tor^{\tl_n}_j(\t,C(m)_i)=0$};
			\foreach \x in { (10,0),(3,0),(6.4,0), (-.2, 0)}
	        \draw[->] \x -- +(-0.35,0);
		\end{tikzpicture}
\caption{The page~$\IIE^1$. The only differentials that affect the~$\IIE^2$ page are shown on the~$j=0$ row.}
\label{fig:second spectral sequence E1 page INVERTIBLE}
\end{figure}

\section{Planar injective words}
\label{section-Wn}

Throughout this section we will consider the Temperley-Lieb algebra~$\tl_n(a)=\tl_n(v+v^{-1})$, where~$v\in R^\times$.
We will make use of the elements~$s_1,\ldots,s_{n-1}$ of Definition~\ref{definition-si}.

\begin{defn}\label{defn: W(n) and boundary maps}
    For~$n\geqslant 0$ we define a chain complex~$W(n)_\ast$ of left~$\tl_n(a)$-modules as follows.
    For~$i$ in the range~$-1\leqslant i\leqslant n-1$, the degree-$i$ part of~$W(n)_\ast$ is defined by
    \[
        W(n)_i = \tl_n(a)\otimes_{\tl_{n-i-1}(a)}\t
    \]
    and in all other degrees we set~$W(n)_i=0$.
    Note that
    \[W(n)_{-1}=\tl_n(a)\otimes_{\tl_{n}(a)}\t=\t.\]
    For~$i\geqslant 0$ the boundary map~$d^i\colon W(n)_i\to W(n)_{i-1}$ is defined to be the alternating sum~$\sum_{j=0}^i (-1)^jd^i_j$, where~$d^i_j\colon W(n)_i\to W(n)_{i-1}$ is the map
    \[
        d^i_j\colon\tl_n(a)\otimes_{\tl_{n-i-1}(a)}\t\longrightarrow\tl_n(a)\otimes_{\tl_{n-i}(a)}\t
    \]
    defined by
    \[
        d^i_j(x\otimes r)= (x\cdot s_{n-i+j-1}\cdots s_{n-i})\otimes\lambda^{-j}r.
    \]
    In the expression~$s_{n-i+j-1}\cdots s_{n-i}$, the indices decrease from left to right.  Thus, for example, the product is~$s_{n-i+1}s_{n-i}$ when~$j=2$, it is~$s_{n-i}$ when~$j=1$, and it is trivial (the unit element) when~$j=0$ (the latter point can be regarded as a convention if one wishes).
    \red{Recall that~$\lambda$ is indeed invertible since~$\lambda=-1$ or $v^2$ and~$v$ is a unit.} For notational purposes we will write~$W(n)$ and only use a subscript when identifying a particular degree.
    \end{defn}

    Observe that~$d_j$ is well-defined because the elements~$s_{n-i},\ldots,s_{n-i+j-1}$ all commute with all generators of~$\tl_{n-i-1}$.
    We have depicted~$W(n)$ in Figure~\ref{figure-wn}.
    \begin{figure}
     	\[
            \xymatrix{
                \tl_n\ootimes{0}\t
                \ar[d]_{d^{n-1}_0-d^{n-1}_1+\cdots+(-1)^{n-1} d^{n-1}_{n-1}}
                &
                n-1
                \\
                \tl_n\ootimes{1}\t
                \ar[d]_{d^{n-2}_0-d^{n-2}_1+\cdots+(-1)^{n-2} d^{n-2}_{n-2}}
                &
                n-2
                \\
                \vdots
                \ar[d]
                &
                {}
                \\
                \tl_n\otimes_{\tl_{n-3}}\t
                \ar[d]_{d^2_0-d^2_1+d^2_2}
                &
                2
                \\
                \tl_n\otimes_{\tl_{n-2}}\t
                \ar[d]_{d^1_0-d^1_1}
                &
                1
                \\
                \tl_n\otimes_{\tl_{n-1}}\t
                \ar[d]
                &
                0
                \\
                \t
                &
                -1
            }
        \]   
        \caption{The complex~$W(n)$.}
        \label{figure-wn}
    \end{figure}

\begin{lemma}\label{lemma-complex}
    The boundary maps of~$W(n)$ satisfy~$d^{i-1}\circ d^i=0$.
\end{lemma}

\begin{proof}
    We will show that if~$i\geqslant 1$ and~$0\leqslant j<k\leqslant i$, then the composite maps~$d^{i-1}_jd^i_k,d^{i-1}_{k-1}d^{i}_j\colon W(n)_i\to W(n)_{i-2}$ coincide. (Thus the~$d^i_j$ satisfy the semi-simplicial identities, so~$W(n)$ is a semi-simplicial $R$-module.)
    The fact that~$d\circ d$ vanishes then follows.
    We have
    \[
        d^{i-1}_jd^i_k(x\otimes r) = [x\cdot(s_{n-i+k-1}\cdots s_{n-i})\cdot (s_{n-i+j}\cdots s_{n-i+1})]\otimes \lambda^{-(j+k)}r
    \]
    and
    \[
        d^{i-1}_{k-1}d^{i}_j(x\otimes r) = [x\cdot(s_{n-i+j-1}\cdots s_{n-i})\cdot (s_{n-i+k-1}\cdots s_{n-i+1})]\otimes \lambda^{-(j+k-1)}r.
    \]
    Now we have:
    \begin{align*}
        (s_{n-i+k-1}\cdots s_{n-i})\cdot& (s_{n-i+j}\cdots s_{n-i+1})
        \\
        =&
        (s_{n-i+j-1}\cdots s_{n-i})\cdot (s_{n-i+k-1}\cdots s_{n-i})
        \\
        =&
        (s_{n-i+j-1}\cdots s_{n-i})\cdot (s_{n-i+k-1}\cdots s_{n-i+1})\cdot s_{n-i}
    \end{align*}
    Here, the first equality follows by taking the letters of the second parenthesis in turn, and `passing through' the first parenthesis, using a single braid relation, with the result that the letter's index is reduced by $1$.  Thus:
    \begin{align*}
        d^{i-1}_jd^i_k(x\otimes r) 
        &= [x\cdot(s_{n-i+k-1}\cdots s_{n-i})\cdot (s_{n-i+j}\cdots s_{n-i+1})]\otimes \lambda^{-(j+k)}r
        \\
        &= [x\cdot (s_{n-i+j-1}\cdots s_{n-i})\cdot (s_{n-i+k-1}\cdots s_{n-i+1})\cdot s_{n-i}]\otimes \lambda^{-(j+k)}r
        \\
        &= [x\cdot (s_{n-i+j-1}\cdots s_{n-i})\cdot (s_{n-i+k-1}\cdots s_{n-i+1})]\otimes s_{n-i}\cdot(\lambda^{-(j+k)}r)
        \\
        &= [x\cdot (s_{n-i+j-1}\cdots s_{n-i})\cdot (s_{n-i+k-1}\cdots s_{n-i+1})]\otimes \lambda^{-(j+k-1)}r
        \\
        &= d^{i-1}_{k-1}d^i_j(x\otimes r)
    \end{align*}
    where the third equality holds because this computation takes place in $W(n)_{i-2}=\tl_n\ootimes{n-i+1}\t$ and~$s_{n-i}\in\tl_{n-i+1}$.
\end{proof}

\begin{remark}\label{remark-Cn}
    Let us explain the motivation for the definition of~$W(n)$.
    Let~$\frakS_n$ denote the symmetric group on~$n$ letters.
    The \emph{complex of injective words} is the chain complex~$\calC(n)$ of~$\frakS_n$-modules, concentrated in degrees~$-1$ to~$(n-1)$, that in degree~$i$ is the free~$R$-module with basis given by tuples~$(x_0,\ldots,x_i)$ where~$x_0,\ldots,x_i\in\{1,\ldots,n\}$ and no letter appears more than once.  
    We allow the empty word~$()$, which lies in degree~$-1$.
    The differential of~$\calC(n)$ sends a word~$(x_0,\ldots,x_i)$ to the alternating sum~$\sum_{j=0}^i(-1)^j(x_0,\ldots,\widehat{x_j},\ldots,x_i)$.
    A theorem of Farmer~\cite{Farmer} shows that the homology of~$\calC(n)$ vanishes in degrees~$i\leqslant (n-2)$, and the same result has been proved since then by many authors \cite{Maazen, BjornerWachs, Kerz, RandalWilliamsConfig}.
    The complex of injective words has been used by several authors to prove homological stability for the symmetric groups~\cite{Maazen,Kerz,RandalWilliamsConfig}.   
    
    For this paragraph only, let us abuse our established notation and denote by $s_1,\ldots,s_{n-1}\in\frakS_n$ the elements defined by~$s_i=(i\ \ i+1)$, the transposition of~$i$ with~$i+1$.
    Then these elements satisfy the braid relations, i.e.~the second and third identities of Definition~\ref{definition-si}.
    The complex of injective words~$\calC(n)$ can be rewritten in terms of the group ring~$R\frakS_n$ and the elements~$s_i$.  
    Indeed, it is shown in \cite{HepworthIH} that~$\calC(n)_i\cong R\frakS_n\otimes_{R\frakS_{n-i-1}}\t$, where~$\t$ is the trivial module of~$R\frakS_{n-i-1}$, and that under this isomorphism the differential~$d^i\colon \calC(n)_i\to\calC(n)_{i-1}$ becomes the map
    \[
        d^i
        \colon 
        R\frakS_n\otimes_{R\frakS_{n-i-1}}\t
        \longrightarrow
        R\frakS_n\otimes_{R\frakS_{n-i}}\t
    \]
    defined by~$d^i(x\otimes 1) = \sum_{j=0}^i(-1)^jx\cdot (s_{n-i+j-1}\cdots s_{n-i})\otimes 1$.
    (There are no constants~$\lambda$ in this expression). Comparing this description of~$\calC(n)$ with our definition of~$W(n)$, we see that our complex of planar injective words is precisely analogous to the original complex of injective words, after systematically replacing the group algebras of symmetric groups with the Temperley-Lieb algebras. The lack of constants in the differential for~$\calC(n)$ is explained by the fact that the effect of~$s_i$ on~$\t$ is multiplication by~$\lambda$ in the Temperley-Lieb setting, and multiplication by~$1$ in the symmetric group setting.
    
    Since we regard the Temperley-Lieb algebra as the planar analogue of the symmetric group, we chose the name \emph{planar injective words} for our complex~$W(n)$. This seemed the least discordant way of giving our complex an appropriate name.  See the next remark for a means of picturing the complex.
\end{remark}

\begin{remark} Let us describe a method for visualising~$W(n)$.
    Recall from the diagrammatic description of $\tl_n(a)\otimes_{\tl_m(a)}\t$ when~$m\leqslant n$ given in Remark~\ref{remark: black boxes} that elements of~$W(n)_i$ can be regarded as diagrams where the first~$n-i-1$ dots on the right are encapsulated within a black box, and if any cups can be absorbed into the black box, then the diagram is identified with~$0$.
    The differential~$d^i\colon W(n)_i\to W(n)_{i-1}$ is then given by pasting special elements onto the right of a diagram, followed by taking their signed and weighted sum. These special elements each enlarge the black box by an extra strand, and plumb one of the free strands into the new space in the black box: Here is an example for~$n=4$ and~$i=2$.
        
        \begin{center}
    	\begin{tikzpicture}[scale=0.38]
    	\foreach \x in {1,2,3,4}
    	\foreach \y in {0,3,6,9,12,15,18,21,24,27,30}
    	\draw[fill=black, line width=1] (\y,\x) circle [radius=0.15] (\y,0.5)--(\y,4.5);
    	\foreach \x\y in {0/1,  6/1, 15/1, 24/1,9/1,9/2,9/3,9/4,18/4, 18/1, 27/1}
    	\draw[black] (\x,\y) --(\x+3,\y);
    	\foreach \x\y in {0/2, 6/2, 15/2, 24/2}
    	\draw (\x,\y) to[out=0,in=-90] (\x+1,\y+.5) to[out=90,in=0] (\x,\y+1);
    	\foreach \x\y in {3/3, 9/3, 18/3, 27/3}
    	\draw (\x,\y) to[out=180,in=-90] (\x-1,\y+.5) to[out=90,in=180] (\x,\y+1);
    	\foreach \x\y in {0/4,6/4, 15/4, 24/4,27/4}
    	\draw (\x,\y) to[out=0, in=120] (\x+1.5, \y-1) to[out=300, in=180] (\x+3,\y-2);
    	\draw (18,3) to[out=0, in=140] (19.5, 2.5) to[out=320, in=180] (21,2);
    	\foreach \x\y in {19.5/2.5, 28.7/2.65, 28.3/3.35}
    	\draw[white,fill=white] (\x,\y) circle [radius=0.15];
    	\foreach \x\y in {18/2,27/2,27/3}
    	\draw (\x,\y) to[out=0, in=220] (\x+1.5, \y+0.5) to[out=40, in=180] (\x+3,\y+1);
    	\foreach \x in {9,18,27}
    	\draw[line width=0.2cm] (\x+3,.8)--(\x+3,2.2); 
    	\draw[line width=0.2cm] (2.8,1)--(3.2,1);
    	\draw (-1,2.5) node {$d^2\colon$} (4.5,2.5) node {$\mapsto$} (13.5,2.5) node {
    	$\scriptstyle{-\lambda^{-1}}
    	$} (22.5,2.5) node {
    	$\scriptstyle{+\lambda^{-2}}
    	$};
    	\end{tikzpicture}
        \end{center}
    	
    \noindent The resulting diagrams can be simplified using the smoothing rules for diagrams with crossings described in Remark~\ref{remark-smoothing}.
    We leave it to the reader to make this description as precise as they wish, and note here that this is where the notion of \emph{braiding}, so often seen in homological stability arguments, fits into our set up.
\end{remark}

\begin{remark}
    Readers who are familiar with the theory will recognise that~$W(n)$ is the chain complex associated to an augmented semi-simplicial~$\tl_n(a)$-module.
\end{remark}

The main result about the complex of planar injective words is the following, which we recall from the introduction.  It is analogous to the homological vanishing property of the complex of injective words first proved by Farmer~\cite{Farmer}.

\setcounter{abcthm}{4}
\begin{abcthm}
    The homology of~$W(n)$ vanishes in degrees~$d\leqslant (n-2)$.
\end{abcthm}

The proof of Theorem~\ref{theorem-high-acyclicity} is the most technical part of this work, and will be given in Section~\ref{section-high-acyclicity}.

The complex of injective words on~$n$ letters has rich combinatorial features: its Euler characteristic is the number of derangements of~$\{1,\ldots,n\}$; when one works over~$\C$, its top homology has a description as a virtual representation that categorifies a well-known alternating sum formula for the number of derangements; and again when one works over~$\C$, its top homology has a compact description in terms of Young diagrams and counts of standard Young tableaux.
In the associated paper \cite{BoydHepworthComb} we establish analogues of these for the complex of planar injective words. 
In particular we show that \red{when the ring~$R$ is Noetherian} the rank of~$H_{n-1}(W(n))$ is the~$n$-th \emph{Fine number}~\cite{DeutschShapiro}. 
(The rank of the Temperley-Lieb algebra is the~$n$\nobreakdash-th \emph{Catalan number}, which is the number of Dyck paths of length~$2n$.  The~$n$-th Fine number is the number of Dyck paths of length~$2n$ whose first peak occurs at an even height, and as we explain in~\cite{BoydHepworthComb}, it is an analogue of the number of derangements.) 
We also discover a new feature of the complex: the differentials have an alternate expression in terms not of the~$s_i$ but of the~$U_i$. This expression demonstrates a connection with the \emph{Jacobsthal numbers}, and we will briefly explain the result for the top differential below.  
The top homology of the Tits building is known as the \emph{Steinberg module}. This inspires the name in the following definition.

\begin{defn}
    We define the \emph{$n$-th Fineberg module} to be the~$\tl_n(a)$-module $\F_n(a) = H_{n-1}(W(n))$.  We often suppress the~$a$ and simply write~$\F_n$. 
\end{defn}

The Fineberg module is an important ingredient in the full statement of our stability result, Theorem~\ref{theorem-fineberg}. In order to detect the non-zero homology group appearing in Theorem~\ref{theorem-sharpness} we need to study it in more detail using the connection with Jacobsthal numbers from \cite{BoydHepworthComb}.

The~$n$-th \emph{Jacobsthal number}~$J_n$~\cite{JacobsthalOEIS} is (among other things) the number of sequences~$n>a_1>a_2>\cdots>a_r>0$ whose initial term has the opposite parity to~$n$. Some examples, when~$n=4$, are~$3$,~$1$,~$3>2$,~$3>1$ and~$3>2>1$. (We allow the empty sequence, and say that by convention its initial term is~$a_1=0$ and~$r=0$. Of course this only occurs when~$n$ is odd.)
Another viewpoint of~$J_n$ in terms of compositions of~$n$ is given in~\cite{BoydHepworthComb}.

\begin{defn}\label{defn-jacobsthal}
   Let~$a = v+v^{-1}$ where~$v\in R^\times$ is a unit. We define the \emph{Jacobsthal} element in~$\tl_n(a)$ as follows:

\[
    \calJ_n=(-1)^{n-1}\sum_{\substack{n>a_1>\cdots>a_r>0\\ n-a_1\text{ odd}}}
    \left(\frac{\mu}{\lambda}\right)^rU_{a_1}\cdots U_{a_r}
\]
   Recall we allow the empty sequence ($a_1=0$ and~$r=0$) when~$n$ is odd. This corresponds to a constant summand~$1$ in~$\calJ_n$ for odd~$n$. Note that the number of irreducible terms in~$\calJ_n$ is~$J_{n}$.
\end{defn}

\begin{example}\label{example-Jn}
	In the cases $n=1,2,3,4$, and choosing $\theta=\theta_1$ so that 
	$(\lambda,\mu)=(-1,v)$, we have:
	\begin{align*}
		\calJ_1 &= 1
		\\
		\calJ_2 &= vU_1
		\\
		\calJ_3 &= v^2 U_2U_1 - vU_2 + 1
		\\
		\calJ_4 &= v^3 U_3U_2U_1 - v^2 U_3U_2 - v^2 U_3U_1
		+ v U_3 + v U_1
	\end{align*}
Spencer has computed the Jacobsthal elements~$\calJ_n$ up to~$n=9$, these can be viewed at \cite{SpencerJ}.
\end{example}

Since~$\F_n$ is the homology of~$W(n)$ in the top degree, it is simply the kernel of the top differential~$d^{n-1}\colon W(n)_{n-1}\to W(n)_{n-2}$.  There are identifications
\[W(n)_{n-1}=\tl_n(a)\otimes_{\tl_0(a)}\t\cong\tl_n(a)\text{ and }W(n)_{n-2}\cong \tl_n(a)\otimes_{\tl_1(a)}\t\cong \tl_n(a).\]  The following is shown in Theorem~D of~\cite{BoydHepworthComb}:

\begin{prop}\label{proposition-fineberg}
    Under the above identifications, the top differential of~$W(n)$ is right-multiplication by~$\calJ_n$.  In particular, there is an exact sequence
    \[
        0\longrightarrow \F_n(a)
        \longrightarrow \tl_n(a)\xrightarrow{-\cdot\calJ_n}\tl_n(a).
    \]
\end{prop}

\begin{remark}
	Note that Definition~\ref{defn-jacobsthal} gives a different value
	for the element $\calJ_n$ than the one that appears in 
	Definition~8.1 and Theorem~D of~\cite{BoydHepworthComb}.
	This is because the proof of Theorem~D of~\cite{BoydHepworthComb}
	contains a sign error: it assumes that $s_i=(\lambda-\mu U_i)$
	rather than $s_i=(\lambda + \mu U_i)$ as it should have done.
	This error is fixed by replacing $\mu$ with $-\mu$ in the formula
	in Definition~8.1 of~\cite{BoydHepworthComb}.
	It is possible to check Example~4.8 by hand to confirm
	that the signs in the present formula for $\calJ_n$ are the correct
	ones.
\end{remark}

The Fineberg module $\F_n$ appears to be a new and interesting representation,
and looks likely to be highly nontrivial for each choice of $n$.
Let us illustrate this by computing $\F_2$, $\F_3$ and $\F_4$.
We will continue with the choice $\theta=\theta_1$
so that $(\lambda,\mu)=(-1,v)$.

Our description will be phrased in terms of the cell modules
of $\tl_n$, which we describe briefly.
A \emph{half-diagram} 
(or \emph{link state} in the language of~\cite{RidoutStAubin})
consists of a vertical line in the plane
decorated with dots labelled $1,\ldots, n$ from bottom to top, 
together with a collection of
arcs in the plane, each of which either connects two dots,
or is connected to a dot at one end, 
in such a way that each dot lies on precisely one arc.
The arcs must lie to the right of the vertical line,
they must be disjoint, and the half-diagrams are taken up to isotopy.
Thus the half-diagrams on $4$ dots are as follows.
\[
	\begin{tikzpicture}[scale=0.4,baseline=32]
	\foreach \x in {1,2,3,4}
	\foreach \y in {3}
		\draw[fill=black,line width=1]  (\y,\x) circle [radius=0.15]  
		(\y,.5)--(\y,4.5);
	\draw (3,1) to[out=0,in=-90] (3+1,1+.5) to[out=90,in=0] (3,2);
	\draw (3,3) to[out=0,in=-90] (3+1,3+.5) to[out=90,in=0] (3,4);
	\end{tikzpicture}
	\qquad\qquad
	\quad
	\begin{tikzpicture}[scale=0.4,baseline=32]
	\foreach \x in {1,2,3,4}
	\foreach \y in {3}
		\draw[fill=black,line width=1]  (\y,\x) circle [radius=0.15]  
		(\y,.5)--(\y,4.5);
	\draw (3,2) to[out=0,in=-90] (3+1,2+.5) to[out=90,in=0] (3,3);
	\draw (3,1) to[out=0,in=-90] (3+2,1+1.5) to[out=90,in=0] (3,4);
	\end{tikzpicture}
	\qquad\qquad
	\quad
	\begin{tikzpicture}[scale=0.4,baseline=32]
	\foreach \x in {1,2,3,4}
	\foreach \y in {3}
		\draw[fill=black,line width=1]  (\y,\x) circle [radius=0.15]  
		(\y,.5)--(\y,4.5);
	\draw (3,1) to[out=0,in=-90] (3+1,1+.5) to[out=90,in=0] (3,2);
	\draw (3,3) -- (4,3);
	\draw (3,4) -- (4,4);
	\end{tikzpicture}
	\qquad\qquad
	\quad
	\begin{tikzpicture}[scale=0.4,baseline=32]
	\foreach \x in {1,2,3,4}
	\foreach \y in {3}
		\draw[fill=black,line width=1]  (\y,\x) circle [radius=0.15]  
		(\y,.5)--(\y,4.5);
	\draw (3,2) to[out=0,in=-90] (3+1,2+.5) to[out=90,in=0] (3,3);
	\draw (3,1) -- (4,1);
	\draw (3,4) -- (4,4);
	\end{tikzpicture}
	\qquad\qquad
	\quad
	\begin{tikzpicture}[scale=0.4,baseline=32]
	\foreach \x in {1,2,3,4}
	\foreach \y in {3}
		\draw[fill=black,line width=1]  (\y,\x) circle [radius=0.15]  
		(\y,.5)--(\y,4.5);
	\draw (3,3) to[out=0,in=-90] (3+1,3+.5) to[out=90,in=0] (3,4);
	\draw (3,1) -- (4,1);
	\draw (3,2) -- (4,2);
	\end{tikzpicture}
	\qquad\qquad
	\begin{tikzpicture}[scale=0.4,baseline=32]
	\foreach \x in {1,2,3,4}
	\foreach \y in {3}
		\draw[fill=black,line width=1]  (\y,\x) circle [radius=0.15]  
		(\y,.5)--(\y,4.5);
	\draw (3,1) -- (4,1);
	\draw (3,2) -- (4,2);
	\draw (3,3) -- (4,3);
	\draw (3,4) -- (4,4);
	\end{tikzpicture}
\]
The \emph{cell module} $S(n,m)$ is the $\tl_n$-module with $R$-basis consisting
of the half-diagrams on $n$ dots in which $m$ arcs have free ends.
The $\tl_n$-module structure on $S(n,m)$ is obtained by pasting planar diagrams
onto the left of half-diagrams and simplifying the result exactly as with
composition in $\tl_n$, with the extra condition that if pasting produces
an arc with two free ends, then the resulting diagram is set to $0$.
In $S(4,2)$, for example, we have:
\[
	U_1\cdot 
	\begin{tikzpicture}[scale=0.4,baseline=25]
	\foreach \x in {1,2,3,4}
	\foreach \y in {3}
		\draw[fill=black,line width=1]  (\y,\x) circle [radius=0.15]  
		(\y,.5)--(\y,4.5);
	\draw (3,1) to[out=0,in=-90] (3+1,1+.5) to[out=90,in=0] (3,2);
	\draw (3,3) -- (4,3);
	\draw (3,4) -- (4,4);
	\end{tikzpicture}
	\ =\ 
	a\cdot
	\begin{tikzpicture}[scale=0.4,baseline=25]
	\foreach \x in {1,2,3,4}
	\foreach \y in {3}
		\draw[fill=black,line width=1]  (\y,\x) circle [radius=0.15]  
		(\y,.5)--(\y,4.5);
	\draw (3,1) to[out=0,in=-90] (3+1,1+.5) to[out=90,in=0] (3,2);
	\draw (3,3) -- (4,3);
	\draw (3,4) -- (4,4);
	\end{tikzpicture}
	\qquad\qquad
	U_2\cdot 
	\begin{tikzpicture}[scale=0.4,baseline=25]
	\foreach \x in {1,2,3,4}
	\foreach \y in {3}
		\draw[fill=black,line width=1]  (\y,\x) circle [radius=0.15]  
		(\y,.5)--(\y,4.5);
	\draw (3,1) to[out=0,in=-90] (3+1,1+.5) to[out=90,in=0] (3,2);
	\draw (3,3) -- (4,3);
	\draw (3,4) -- (4,4);
	\end{tikzpicture}
	\ =\ 
	\begin{tikzpicture}[scale=0.4,baseline=25]
	\foreach \x in {1,2,3,4}
	\foreach \y in {3}
		\draw[fill=black,line width=1]  (\y,\x) circle [radius=0.15]  
		(\y,.5)--(\y,4.5);
	\draw (3,2) to[out=0,in=-90] (3+1,2+.5) to[out=90,in=0] (3,3);
	\draw (3,1) -- (4,1);
	\draw (3,4) -- (4,4);
	\end{tikzpicture}
	\qquad\qquad
	U_3\cdot 
	\begin{tikzpicture}[scale=0.4,baseline=25]
	\foreach \x in {1,2,3,4}
	\foreach \y in {3}
		\draw[fill=black,line width=1]  (\y,\x) circle [radius=0.15]  
		(\y,.5)--(\y,4.5);
	\draw (3,1) to[out=0,in=-90] (3+1,1+.5) to[out=90,in=0] (3,2);
	\draw (3,3) -- (4,3);
	\draw (3,4) -- (4,4);
	\end{tikzpicture}
	\ =\ 
	0
\]
(The reader is reminded that we label the dots from bottom to top.)
Observe that $S(n,n)=\t$ is the trivial module for each $n$,
and that $S(n,m)$ is nonzero only when $n-m$ is even.

\begin{example}[The Fineberg module $\F_2$]\label{example - F2}
	The module $\F_2$ is the kernel of the map
	$\tl_2\to\tl_2$, $x\mapsto x\cdot\calJ_2$.
	Now $\calJ_2 = vU_1$ as in Example~\ref{example-Jn},
	so that $\F_2$ is the $R$-module of rank $1$ spanned by the element
	$a-U_1$.
	This is a copy of the trivial module $\t = S(2,2)$.
\end{example}

\begin{example}[The Fineberg module $\F_3$]
	The module $\F_3$ is the kernel of the map
	$\tl_3\to\tl_3$, $x\mapsto x\cdot\calJ_3$,
	where $\calJ_3 = v^2U_2U_1-vU_2+1$ as in Example~\ref{example-Jn}.
	Thus $\F_3$ is the $R$-module of rank 2 with basis elements
	\begin{align*}
		\alpha &= U_1U_2 - vU_1,
		\\
		\beta &= U_2 - vU_2U_1.
	\end{align*}
	One can now check that there is an isomorphism of $\tl_3$-modules
	$\F_3\cong S(3,1)$ given by:
	\[
		\F_3\xrightarrow{\ \cong\ } S(3,1)
		\qquad
		\qquad
    		\alpha\longmapsto
        	\begin{tikzpicture}[scale=0.5,baseline=25]
        	\foreach \x in {1,2,3}
        	\foreach \y in {3}
        		\draw[fill=black,line width=1]  (\y,\x) circle [radius=0.15]  
        		(\y,.5)--(\y,3.5);
        	\draw (3,1) to[out=0,in=-90] (3+1,1+.5) to[out=90,in=0] (3,2);
		\draw (3,3) -- (4,3);
		\end{tikzpicture}
		\qquad
		\qquad
		\beta\longmapsto
        	\begin{tikzpicture}[scale=0.5,baseline=25]
        	\foreach \x in {1,2,3}
        	\foreach \y in {3}
        		\draw[fill=black,line width=1]  (\y,\x) circle [radius=0.15]  
        		(\y,.5)--(\y,3.5);
        	\draw (3,2) to[out=0,in=-90] (3+1,2+.5) to[out=90,in=0] (3,3);
		\draw (3,1) -- (4,1);
		\end{tikzpicture}
	\]
\end{example}

\begin{example}[The Fineberg module $\F_4$]
	The module $\F_4$ is the kernel of the map
	$\tl_4\to\tl_4$, $x\mapsto x\cdot\calJ_4$,
	where 		
	$\calJ_4 = v^3 U_3U_2U_1 - v^2 U_3U_2 - v^2 U_3U_1
	+ v U_3 + v U_1$
	as in Example~\ref{example-Jn}.
	It is now possible to check (at length) that $\F_4$ is a free
	$R$-module of rank $6$ with the following basis:
	\begin{align*}
		A 
		&= 
		{U_3U_1} - a U_3 U_1U_2
		\\
		B
		&=
		{U_2 U_3 U_1} - a U_2 U_3 U_1 U_2
		\\
		X
		&=
		{U_1 U_2 U_3} - U_3 U_1 U_2 - a U_1U_2 + U_1
		\\
		Y
		&=
		U_2 U_3 - {U_2 U_3 U_1 U_2} - a U_2 + U_2 U_1
		\\
		Z
		&=
		{U_3 U_2U_1} - U_3 U_1 U_2 - a U_3 U_2 + U_3
		\\
		P
		&=
		{U_3U_1U_2} - U_1 - U_3 + a.
	\end{align*}
	If we now define
	\[
		M_0 = \Span(A,B),
		\quad
		M_1 = \Span(A,B,X,Y,Z),
		\quad
		M_2= \Span(A,B,X,Y,Z,P)
	\]
	so that 
	$
		M_0\subseteq M_1\subseteq M_2=\F_4,
	$
	then one can check directly (by computing the effect of 
	multiplying on the left by $U_1,U_2,U_2$)
	that $M_0$ and $M_1$ are submodules of $\F_4$, and moreover,
	that we have isomorphisms:
	\begin{align*}
		&M_0\xrightarrow{\ \cong\ } S(4,0),
		\qquad
    		&&A\longmapsto
        	\begin{tikzpicture}[scale=0.5,baseline=32]
        	\foreach \x in {1,2,3,4}
        	\foreach \y in {3}
        		\draw[fill=black,line width=1]  (\y,\x) circle [radius=0.15]  
        		(\y,.5)--(\y,4.5);
        	\draw (3,1) to[out=0,in=-90] (3+1,1+.5) to[out=90,in=0] (3,2);
        	\draw (3,3) to[out=0,in=-90] (3+1,3+.5) to[out=90,in=0] (3,4);
		\end{tikzpicture}
		\qquad\qquad
    		B\longmapsto
        	\begin{tikzpicture}[scale=0.5,baseline=32]
        	\foreach \x in {1,2,3,4}
        	\foreach \y in {3}
        		\draw[fill=black,line width=1]  (\y,\x) circle [radius=0.15]  
        		(\y,.5)--(\y,4.5);
        	\draw (3,2) to[out=0,in=-90] (3+1,2+.5) to[out=90,in=0] (3,3);
        	\draw (3,1) to[out=0,in=-90] (3+2,1+1.5) to[out=90,in=0] (3,4);
		\end{tikzpicture}
		\\
		&M_1/M_0 \xrightarrow{\ \cong\ }
		S(4,2),
		\qquad
		&&X\longmapsto
        	\begin{tikzpicture}[scale=0.5,baseline=32]
        	\foreach \x in {1,2,3,4}
        	\foreach \y in {3}
        		\draw[fill=black,line width=1]  (\y,\x) circle [radius=0.15]  
        		(\y,.5)--(\y,4.5);
        	\draw (3,1) to[out=0,in=-90] (3+1,1+.5) to[out=90,in=0] (3,2);
		\draw (3,3) -- (4,3);
		\draw (3,4) -- (4,4);
		\end{tikzpicture}
		\qquad\qquad
		Y\longmapsto
        	\begin{tikzpicture}[scale=0.5,baseline=32]
        	\foreach \x in {1,2,3,4}
        	\foreach \y in {3}
        		\draw[fill=black,line width=1]  (\y,\x) circle [radius=0.15]  
        		(\y,.5)--(\y,4.5);
        	\draw (3,2) to[out=0,in=-90] (3+1,2+.5) to[out=90,in=0] (3,3);
		\draw (3,1) -- (4,1);
		\draw (3,4) -- (4,4);
		\end{tikzpicture}
		\qquad\qquad
		Z\longmapsto
        	\begin{tikzpicture}[scale=0.5,baseline=32]
        	\foreach \x in {1,2,3,4}
        	\foreach \y in {3}
        		\draw[fill=black,line width=1]  (\y,\x) circle [radius=0.15]  
        		(\y,.5)--(\y,4.5);
        	\draw (3,3) to[out=0,in=-90] (3+1,3+.5) to[out=90,in=0] (3,4);
		\draw (3,1) -- (4,1);
		\draw (3,2) -- (4,2);
		\end{tikzpicture}
		\\
		&M_2/M_1\xrightarrow{\ \cong\ }\t,
		\qquad
		&&P\longmapsto 1
	\end{align*}
	Thus $\F_4$ has a filtration in which each of the three
	cell modules appears as precisely one of the filtration quotients.
	We emphasise that this result holds with \emph{no}
	further assumptions on the ground ring $R$ or on the parameter $v$.
\end{example}

\section{Homological stability and stable homology}
\label{section-stability}

The aim of this section is to prove the following result.
Theorem~\ref{theorem-vanishing-range} is an immediate consequence, and Theorem~\ref{theorem-sharpness} will be proved in the next section as a corollary of it.

\begin{thm}\label{theorem-fineberg}
    Let~$R$ be a commutative ring, let~$v\in R$ be a unit, and let $a=v+v^{-1}$.
    Then for~$n$ odd we have:
	\[
		\Tor_i^{\tl_n(a)}(\t,\t)\cong
		\begin{cases}
			R & i=0
			\\
			0 & 1\leqslant i\leqslant (n-1)
			\\
			\Tor^{\tl_n(a)}_{i-n}(\t,\F_n(a)) & i\geqslant n
		\end{cases}
	\]
    and for~$n$ even we have
	\[
		\Tor_i^{\tl_n(a)}(\t,\t)\cong
		\begin{cases}
			R & i=0
			\\
			0 & 1\leqslant i\leqslant (n-2)
			\\
			\Tor^{\tl_n(a)}_{i-n}(\t,\F_n(a)) & i \geqslant (n+1)
		\end{cases}
	\]
	for~$i\neq n-1,n$, while in degrees~$(n-1)$ and~$n$ there is an exact sequence
    \begin{equation}\label{equation-tor-exact-sequence}
        0
        \to
        \Tor_n^{\tl_n(a)}(\t,\t)
        \to 
        \t\otimes_{\tl_n(a)}\F_n(a)
        \overset{\mathcal{Q}_n}{\longrightarrow} 
        \t
        \to
        \Tor_{n-1}^{\tl_n(a)}(\t,\t)
        \to
        0.
    \end{equation}
    Analogous results hold for the~$\Ext$-groups.
    For~$n$ odd we have:
	\[
		\Ext_{\tl_n(a)}^i(\t,\t)\cong
		\begin{cases}
			R & i=0
			\\
			0 & 1\leqslant i\leqslant (n-1)
			\\
			\Ext_{\tl_n(a)}^{i-n}(\F_n(a),\t) & i\geqslant n
		\end{cases}
	\]
    and for~$n$ even we have
	\[
		\Ext_{\tl_n(a)}^i(\t,\t)\cong
		\begin{cases}
			R & i=0
			\\
			0 & 1\leqslant i\leqslant (n-2)
			\\
			\Ext_{\tl_n(a)}^{i-n}(\F_n(a),\t) & i \geqslant (n+1)
		\end{cases}
	\]
	for~$i\neq n-1,n$, while in degrees~$(n-1)$ and~$n$ there is an exact sequence
    \begin{equation}\label{equation-ext-exact-sequence}
      \hspace{-2mm}  0
        \to
        \Ext^{n-1}_{\tl_n(a)}(\t,\t)
        \to 
        \t
        \overset{\mathcal{Q}^n}{\longrightarrow}  
        \Hom_{\tl_n(a)}(\F_n(a),\t)
        \to
        \Ext_{\tl_n(a)}^n(\t,\t)
        \to
        0.
    \end{equation}
    The central maps $\mathcal{Q}_n$ and $\mathcal{Q}^n$ of~\eqref{equation-tor-exact-sequence} and~\eqref{equation-ext-exact-sequence} respectively are described as follows.
    Regard $\F_n(a)$ as a left-submodule of $\tl_n(a)$ as in Proposition~\ref{proposition-fineberg}. 
    Then the maps are
    \[
        \mathcal{Q}_n\colon\t\otimes_{\tl_n(a)}\F_n(a)\longrightarrow\t,
        \qquad\qquad
        x\otimes f\longmapsto x\cdot f
    \]
    and
    \[
        \mathcal{Q}^n\colon\t\longrightarrow\Hom_{\tl_n(a)}(\F_n(a),\t),
        \qquad\qquad
        x\longmapsto (f\mapsto f\cdot x)
    \]
    where $x\cdot f$ and $f\cdot x$ denote the action of $f\in\F_n(a)\subseteq\tl_n(a)$ on the right and left of $\t$, respectively.
\end{thm}

In order to prove this theorem, we will use the complex of planar injective words~$W(n)$ introduced in the previous section.  Recall that the Fineberg module~$\F_n$ appearing in the statement is the top homology group~$H_{n-1}(W(n))$.

\begin{lemma}\label{lemma-tensor-homology}
	The homology groups of both the complex $\t\otimes_{\tl_n(a)}W(n)$ and the complex $\Hom_{\tl_n(a)}(W(n),\t)$ are concentrated in degree~$(n-1)$, where in both cases they are given by~$\t$ if~$n$ is even and~$0$ if~$n$ is odd.
\end{lemma}

\begin{proof}
	We have~$W(n)_i = \tl_n\otimes_{\tl_{n-i-1}}\t$, and the boundary map~$d^i\colon W(n)_i\to W(n)_{i-1}$ is given by~$x\otimes r\mapsto x\cdot D_i\otimes r$, where~$D_i = \sum_{j=0}^i(-1)^j s_{n-i+j-1}\cdots s_{n-i}\lambda^{-j}$.
	
	By regarding~$\t$ as both a left and right~$\tl_n$-module, we may regard~$\t\otimes_{\tl_n} W(n)_i$ as a left~$\tl_n$-module. With this~$\tl_n$-module structure, we obtain~$\t\otimes_{\tl_n} W(n)_i = \t\otimes_{\tl_n}(\tl_n\otimes_{\tl_{n-i-1}}\t) \cong\t$. Under these isomorphisms, the boundary map originating in degree $i$ becomes the action on~$\t$ of the element $D_i$. 
	Similarly, $\Hom_{\tl_n}(W(n)_i,\t)=\Hom_{\tl_n}(\tl_n\otimes_{\tl_{n-i-1}}\t,\t)\cong\t$, and under these isomorphisms the boundary map originating in degree $(i-1)$ becomes the action of the element $D_i$ on $\t$.
	
	The action of~$s_{n-i+j-1}\cdots s_{n-i}$ on~$\t$ is simply multiplication by~$\lambda^j$, with one factor of~$\lambda$ for each~$s$ term (recall~$s_i=\mu U_i+\lambda$).  Thus the action of~$D_i$ on~$\t$ is nothing other than multiplication by~$\sum_{j=0}^i(-1)^j$, which is~$0$ for~$i$ odd and~$1$ for~$i$ even.
	
	So altogether~$\t\otimes_{\tl_n}W(n)$ and $\Hom_{\tl_n}(W(n),\t)$ are isomorphic to complexes with a copy of~$R$ in each degree~$i=-1,\ldots,(n-1)$ and with boundary maps alternating between the identity map and~$0$.
	In $\t\otimes_{\tl_n}W(n)$ the identity maps originate in even degrees, and in $\Hom_{\tl_n}(W(n),\t)$ they originate in odd degrees.
	The claim now follows.
\end{proof}

\begin{proof}[Proof of Theorem~\ref{theorem-fineberg}]
    We begin with the $\Tor$-case.
    
    In degree~$d=0$ the theorem holds trivially.
    Recall that~$P_\ast$ is a projective resolution of~$\t$ as a right~$\tl_n$-module. 
    We use the two homological spectral sequences~$\{\IE^r\}$ and~$\{\IIE^r\}$ associated to $W(n)$ as described in Section~\ref{sec: double complex ss}.
    
    Let us consider~$\{\IE^r\}$. 
    We have 
    \[
            \IE^2_{i,j}=
            \begin{cases}
                \Tor_i^{\tl_n}(\t,\F_n) & j=(n-1)
                \\
                0 & j\neq(n-1)
            \end{cases}
    \]
    and consequently the spectral sequence converges to~$\Tor_{\ast-n+1}^{\tl_n}(\t,\F_n)$, for~$*=i+j$.
    The same is therefore true of~$\{\IIE^r\}$.
    
    Let us write~$\varepsilon_n=H_{n-1}(\t \otimes_{\tl_n}W(n))$, so that by Lemma~\ref{lemma-tensor-homology},~$\varepsilon_n$ is trivial for~$n$ odd and~$\t$ for~$n$ even. Since~$\F_n$ consists of the cycles in~$W(n)_{n-1}$, the map
    \[\t\otimes_{\tl_n}\F_n\to \t\otimes_{\tl_n}W(n)_{n-1}\] again lands in the cycles, giving us a map
    \[\t\otimes_{\tl_n}\F_n\to H_{n-1}(\t\otimes_{\tl_n}W(n))=\varepsilon_n.\]  When~$n$ is even and~$\varepsilon_n$ is identified with~$\t$ as in the lemma, then this map simply becomes~$\mathcal{Q}_n$ as described in the statement of the theorem. 
    
    We now know that~$\{{}^{II}E^r\}$ converges to~$\Tor_{\ast-n+1}^{\tl_n}(\t,\F_n)$.  Its~$E^1$-page~${}^{II}E^1_{i,j} = \Tor^{\tl_n}_j(\t, W(n)_i)$ is largely known to us.  Indeed, when~$j=0$ the terms are $\Tor^{\tl_n}_0(\t, W(n)_i) =\t\otimes_{\tl_n} W(n)_i$, with~$d^1$-maps between them induced by the boundary maps of~$W(n)$.  In other words, the~$j=0$ part of~$\IIE^1_{i,j}$ is precisely the complex~$\t\otimes_{\tl_n} W(n)$.  When~$0\leqslant i\leqslant (n-1)$, the term~$W(n)_i=\tl_n\otimes_{\tl_{n-i-1}}\t$ satisfies~$0\leqslant (n-i-1)< n$, so that by Theorem~\ref{theorem-shapiro} we have
    \[\IIE^1_{i,j}=\Tor_j^{\tl_n}(\t,\tl_n\otimes_{\tl_{n-i-1}}\t)=0\] for~$j>0$.  When~$i=-1$ we have~$W(n)_{-1}=\t$ so that ${}^{II}E^1_{-1,j}=\Tor_j^{\tl_n}(\t,\t)$ for~$j>0$. This is depicted in Figure~\ref{fig:second spectral sequence E1 page}.
    
    \begin{figure}
		\begin{tikzpicture}[scale=0.7]
			\draw[<->] (-.5,9)--(-.5,.5)--(12,.5);
			\draw[line width=0.2cm, gray!20, <->] (-4,9)--node[black, above, pos=0] {$j$}(-4,-1)--(12,-1) node[black, right] {$i$};
			\draw (-4,0) node {$0$};
			\draw (-2,-1) node {$-1$};
			\draw (1,-1) node {$0$};		
			\draw (-4,1) node {$1$};
			\draw (-4,2) node {$2$};
			\foreach \x in {(-4,4),(-2,4)}	
			\draw \x node {$\vdots$};
			\foreach \x in {(3,-1), (3,0)}
			\draw \x node {$\cdots$};			
			\draw (5,-1) node[rotate=90]  {$\scriptstyle{n-2}$};	\draw (11,-1) node[rotate=90]  {$\scriptstyle{n}$};					
			\draw(8.5,-1) node[rotate=90]  {$\scriptstyle{n-1}$};
			\draw (-4.1,6) node {$\scriptstyle{n -2}$};
			\draw (-4.1,7) node {$\scriptstyle{n-1}$};
			\draw (-4,8) node {$\scriptstyle{n }$};
			\foreach \x in {1,2}
			\draw (-2,\x) node {$\scriptstyle{\Tor^{\tl_n}_{\x}(\t,\t)}$};
			\draw (-2,6) node {$\scriptstyle{\Tor^{\tl_n}_{n-2}(\t,\t)}$};
			\draw (-2,7) node {$\scriptstyle{\Tor^{\tl_n}_{n-1}(\t,\t)}$};
			\draw (-2,8) node {$\scriptstyle{\Tor^{\tl_n}_{n}(\t,\t)}$};
			\draw (-2,0) node {$\scriptstyle{\t \otimes_{\tl_n} W(n)_{-1}}$};
			\draw (1.25,0) node {$\scriptstyle{\t \otimes_{\tl_n} W(n)_{0}}$};
			\draw (5,0) node {$\scriptstyle{\t \otimes_{\tl_n} W(n)_{n-2}}$};
			\draw (8.75,0) node {$\scriptstyle{\t \otimes_{\tl_n} W(n)_{n-1}}$};
			\draw (11,0) node {$0$};
			\draw (5,2) node[black!80,scale=1.2] {$\Tor_j^{\tl_n}(\t, \tl_n\otimes_{\tl_{n-i-1}}\t)=0$};
			\foreach \x in { (10.8,0),(7,0), (-.2, 0)}
	        \draw[->] \x -- +(-0.35,0);
		\end{tikzpicture}
\caption{The page~$\IIE^1$. The only differentials that affect the~$\IIE^2$ page are shown on the~$j=0$ row.}
\label{fig:second spectral sequence E1 page}
\end{figure}

    By the description in the previous paragraph, we can now identify~$\IIE^2_{\ast,\ast}$.  The only possible differentials are in the~$j=0$ part, which is~$\t\otimes_{\tl_n}W(n)$, and whose homology is~$\varepsilon_n$ concentrated in degree~$(n-1)$.
    Thus~$\IIE^2_{\ast,\ast}$ is zero except for the following groups:
    \[
    	\IIE^2_{i,j}=\begin{cases}
    		\Tor^{\tl_n}_j(\t,\t) 
    		& 
    		i=-1,\ j>0
    		\\
    		\varepsilon_n
    		&
    		i=(n-1),\ j=0
    	\end{cases}
    \]
    as depicted in Figure~\ref{fig:second spectral sequence E2 page}.

\begin{figure}
		\begin{tikzpicture}[scale=0.75]
			\draw[<->] (-.5,9)--(-.5,.5)--(9,.5);	
			\draw[line width=0.2cm, gray!20, <->] (-4,9)--node[black, above, pos=0] {$j$}(-4,-1)--(10,-1) node[black, right] {$i$};
			\draw (-4,0) node {$0$};
			\draw (-2,-1) node {$-1$};
			\draw (0,-1) node {$0$};		
			\draw (-4,1) node {$1$};					
			\draw (1,-1) node {$1$};
			\draw (-4,2) node {$2$};					
			\draw (2,-1) node {$2$};
			\foreach \x in {(-4,4),(-2,4)}	
			\draw \x node {$\vdots$};
			\foreach \x in {(4,-1), (4,0)}
			\draw \x node {$\cdots$};			
			\draw (6,-1) node[rotate=90]  {$\scriptstyle{n-3}$};	\draw (7,-1) node[rotate=90]  {$\scriptstyle{n-2}$};	\draw (9,-1) node[rotate=90]  {$\scriptstyle{n}$};					\draw(8,-1) node[rotate=90]  {$\scriptstyle{n-1}$};
			\draw (-4.1,6) node {$\scriptstyle{n -2}$};
			\draw (-4.1,7) node {$\scriptstyle{n-1}$};
			\draw (-4,8) node {$\scriptstyle{n }$};
		
			\draw[cyan!40,ultra thick] (7,0) -- (-1.5,7) node[black,left, pos=0.1] {$\scriptstyle{i+j=n-2}$};
			\draw[cyan!40, ultra thick] (8,0) -- (-1.5,8) node[black, right, pos=0.9] {$\scriptstyle{i+j=n-1}$};	
			\draw[red, dashed, ->, thick] (7.8,0.2) -- (-0.85,6.8) node[black,above, pos=0.5] {$d^n$};
			\foreach \x in {1,2}
			\draw (-2,\x) node {$\scriptstyle{\Tor^{\tl_n}_{\x}(\t,\t)}$};
			\draw (8,0) node {$\varepsilon_n$};
			\draw (-2,6) node {$\scriptstyle{\Tor^{\tl_n}_{n-2}(\t,\t)}$};
			\draw (-2,7) node {$\scriptstyle{\Tor^{\tl_n}_{n-1}(\t,\t)}$};
			\draw (-2,8) node {$\scriptstyle{\Tor^{\tl_n}_{n}(\t,\t)}$};	\foreach \x in {(-2,0),(0,0),(1,0), (2,0), (6,0),(7,0),(9,0)}
			\draw \x node {$0$};
			\draw (1,2) node[black!80,scale=3] {$0$};
		\end{tikzpicture}
\caption{The page~$\IIE^2$. This page stays constant until~$\IIE^n$ where the only possible further differential lies: this is shown in red. The~$i+j=n-1$ and~$i+j=n-2$ diagonals are indicated in blue.}
\label{fig:second spectral sequence E2 page}
\end{figure}
    
    From the~$E^2$-page onwards there is precisely one possible differential, namely $d^n\colon E^n_{n-1,0}\to E^n_{-1,n-1}$, which is a map~$d^n\colon \varepsilon_n\to\Tor_{n-1}^{\tl_n}(\t,\t)$.  It forms part of an exact sequence
    \[
        0
        \to
        \IIE^\infty_{n-1,0}
        \to
        \varepsilon_n
        \xrightarrow{d^n}
        \Tor_{n-1}^{\tl_n}(\t,\t)
        \to
        \IIE^\infty_{-1,n-1}
        \to 
        0
    \]
    In~$\IIE^\infty_{\ast,\ast}$, each total degree has only one non-zero group, except (possibly) for total degree~$(n-1)$, where we have the two groups~$\IIE^\infty_{-1,n}$ and~$\IIE^\infty_{n-1,0}$.
    The relationship between the infinity-page of a spectral sequence and the sequence's target now give us a short exact sequence:
    \[
        0
        \to
        \IIE^\infty_{-1,n}
        \to 
        \Tor_0^{\tl_n}(\t,\F_n)
        \to 
        \IIE^\infty_{n-1,0}
        \to
        0
    \]
    The last two exact sequences combine to give us:
    \[
        0
        \to
        \IIE^\infty_{-1,n}
        \to 
        \Tor_0^{\tl_n}(\t,\F_n)
        \to 
        \varepsilon_n
        \to
        \Tor_{n-1}^{\tl_n}(\t,\t)
        \to
        \IIE^\infty_{-1,n-1}
        \to 
        0
    \]
    The leftmost term is~$\IIE^\infty_{-1,n} = \IIE^2_{-1,n}=\Tor^{\tl_n}_n(\t,\t)$. 
    And~$\IIE^\infty_{-1,n-1}$ is the only group in total degree~$(n-2)$, and therefore coincides with~$\Tor^{\tl_n}_{(n-2)-n+1}(\t,\F_n) =\Tor^{\tl_n}_{-1}(\t,\F_n)=0$.  
    And~$\Tor^{\tl_n}_0(\t,\F_n) = \t\otimes_{\tl_n}\F_n$.  So the last exact sequence becomes:
    \[
        0
        \to
        \Tor_n^{\tl_n}(\t,\t)
        \to 
        \t\otimes_{\tl_n}\F_n
        \to 
        \varepsilon_n
        \to
        \Tor_{n-1}^{\tl_n}(\t,\t)
        \to
        0
    \]
    When~$n$ is even, we claim that the map~$\t\otimes_{\tl_n}\F_n\to\varepsilon_n$ in this sequence is~$\mathcal{Q}_n$.  Let~$\F_n[n-1]$ be the complex consisting of a copy of~$\F_n$ concentrated in degree~$n-1$. There is a natural inclusion of chain complexes~$\F_n[n-1]\hookrightarrow W(n)$, and this leads to a map of double complexes and then of spectral sequences.
    The map~$\t\otimes_{\tl_n}\F_n\to\varepsilon_n$ can be identified using this map of spectral sequences.
    
    It follows from the sequence that in the case~$n$ odd, when~$\varepsilon_n=0$, the final term satisfies~$\Tor_{n-1}^{\tl_n}(\t,\t)=0$, and the first two terms satisfy
    \[\Tor_n^{\tl_n}(\t,\t)\cong\t\otimes_{\tl_n}\F_n=\Tor^{\tl_n}_0(\t,\F_n)\] as required. 
    
    The previous discussion determines what happens in total degrees~$(n-1)$ and $(n-2)$.  In total degrees~$d$ other than~$(n-1)$ and~$(n-2)$, and when~$j>0$, the only term on the~$E^\infty$ page is~$\IIE^{\infty}_{-1,d+1}=\Tor^{\tl_n}_{d+1}(\t,\t)$, which must therefore equal~$\Tor^{\tl_n}_{d-n+1}(\t,\F_n)$.
    Thus~$\Tor^{\tl_n}_d(\t,\t)\cong \Tor^{\tl_n}_{d-n}(\t,\F_n)$ for~$d\neq n,n-1$.
    This completes the proof.
    
    For the $\Ext$-case we use the two cohomological spectral sequences associated to $W(n)$ as in Section~\ref{sec: double complex ss}, and then proceed dually to the above.
    We leave the details to the reader.
\end{proof}

\section{Sharpness}
\label{section-sharpness}

We recall the statement of Theorem~\ref{thm-sharpness} from the introduction.

\setcounter{abcthm}{2}
\begin{abcthm}\label{thm-sharpness}
    Let~$n$ be even and suppose that~$a$ is not a unit.
    Then~$\Tor^{\tl_n(a)}_{n-1}(\t,\t)$ is non-zero.
\end{abcthm}

Let~$\calI\subseteq\tl_n$ denote the left-ideal generated by all diagrams which have a cup on the right in positions other than~$1$, together with all multiples of~$a$.  Thus
\[
    \calI 
    = 
    (\tl_n\cdot a)+(\tl_n\cdot U_2)+\cdots+(\tl_n\cdot U_{n-1}).
\]

\begin{lemma}\label{lemma-mults-of-Jn}
    Let~$n$ be even or odd, and let~$1\leqslant p\leqslant n-1$.
    Then~$U_p\cdot \calJ_n\in\calI$.
\end{lemma}

\begin{proof}
    Recall from Definition~\ref{defn-jacobsthal} that the monomials appearing in~$\calJ_n$ are those of the form~$U_{i_1}\cdots U_{i_r}$ where~$(n-1)\geqslant i_1>i_2\cdots>i_r\geqslant 1$ and~$i_1\equiv (n-1)\mod 2$, and that such a monomial appears in~$\calJ_n$ with coefficient~$(-1)^{n-1}(\frac{\mu }{ \lambda})^r$.
    We write~$\calJ_n = K_n + L_n$ where~$K_n$ is the part of~$\calJ_n$ featuring monomials of the form~$U_iU_{i-1}\cdots U_1$ for~$i\equiv n-1\text{ mod }2$ in the range~$1\leqslant i\leqslant n-1$, and~$L_n$ is the part of~$\calJ_n$ featuring the remaining monomials.
    
    If~$U_{i_1}\cdots U_{i_r}$ is a monomial appearing in~$L_n$, then it must either end in~$U_{i_r}$ for~$n-1\geqslant i_r>1$ or end in a monomial of the form~$U_{i_j}\cdot U_{i_{j-1}}\cdots U_{1}= (U_{i_{j-1}}\cdots U_1) \cdot U_{i_j}$ for some~$i_j\geqslant i_{j-1}+2$,~$i_{j-1}\geqslant 1$ and hence must lie in~$\calI$.  Thus~$L_n\in\calI$, and to prove the lemma it will be sufficient to show that~$U_p\cdot K_n\in\calI$. 
    
    Now observe that 
    \[
        K_n 
        = 
        (-1)^{n-1}\sum_{
            \substack{0\leqslant i\leqslant (n-1)\\ i\equiv n-1\text{ mod }2}
        } \left(\frac{\mu }{ \lambda}\right)^i\cdot U_iU_{i-1}\cdots U_1.
    \]
    (In the case~$i=0$ the product~$U_i\cdots U_1$ is empty and therefore equal to~$1$. This term only appears in~$K_n$ when~$n$ is odd.)
    Suppose that~$U_i\cdots U_1$ is a monomial appearing in the above sum.  Then:
    \[
        U_p\cdot (U_i\cdots U_1)
        =
        \begin{cases}
            (U_p\cdots U_1)\cdot (U_i\cdots U_{p+2})
            &
            p\leqslant i-2
            \\
            U_{i-1}\cdots U_1
            & 
            p=i-1
            \\
            U_i\cdots U_1\cdot a
            & 
            p=i 
            \\
            U_{i+1}\cdots U_1
            &
            p=i+1
            \\
            (U_i\cdots U_1)\cdot U_p
            &
            p\geqslant i+2
        \end{cases}
    \]
    Thus~$U_p\cdot(U_i\cdots U_1)\in\calI$ except for the cases
   ~$i=p-1$,~$i=p+1$.  
    When $p\equiv (n-1)\text{ mod }2$ these exceptional cases never occur, since we have assumed $i\equiv (n-1)\text{ mod }2$, and so~$U_p\cdot K_n\in\calI$ as required.  
    And when~$p\equiv n\text{ mod }2$, we can compute the contribution from the two exceptional cases to find that, modulo~$\calI$,~$U_p\cdot \calJ_n$ is equal to 
    \begin{eqnarray*}
        &&(-1)^{n-1}\left(\frac{\mu }{ \lambda}\right)^{p-1} U_p\cdot (U_{p-1}\cdots U_1)
        +
        (-1)^{n-1}\left(\frac{\mu }{ \lambda}\right)^{p+1} U_p\cdot (U_{p+1}\cdots U_1)
        \\
        &=&
        (-1)^{n-1}\left(\frac{\mu }{ \lambda}\right)^{p-1}\cdot (U_p\cdots U_1)
        +
        (-1)^{n-1}\left(\frac{\mu }{ \lambda}\right)^{p+1}\cdot (U_p\cdots U_1)
        \\
        &=&
        (-1)^{n-1}\left(\frac{\mu }{ \lambda}\right)^p\Big[ \left(\frac{\mu }{ \lambda}\right)^{-1}+\left(\frac{\mu }{ \lambda}\right)^1  \Big]\cdot (U_p\cdots U_1)
        \in\calI.
    \end{eqnarray*}
    Now from Definition~\ref{defn-IH to TL} we have either~$(\mu, \lambda)=(v,-1)$ or~$(\mu, \lambda)=(-v,v^2)$. In both cases the square bracket above evaluates to~$-a$ (recall~$a=v+v^{-1}$). Thus~$U_p\cdot K_n$ is a multiple of~$a$ and therefore in~$\calI$ as required.
\end{proof}

\begin{lemma}
    Let~$n$ be even.
    Let~$x\in\F_n(a)$, so that~$x\cdot \calJ_n=0$.  Then the constant term of~$x$ is a multiple of~$a$.
\end{lemma}

\begin{proof}
    Let~$b$ be the constant term of~$x$, so that~$x$ is equal to~$b$ plus a linear combination of left-multiples of the elements~$U_1,\ldots,U_{n-1}$.  Thus~$x\cdot \calJ_n$ is equal to~$b\cdot \calJ_n$ plus a linear combination of left-multiples of~$U_1\cdot \calJ_n,\ldots,U_{n-1}\cdot \calJ_n$, all of which lie in~$\calI$ by Lemma~\ref{lemma-mults-of-Jn}.  Thus~$x\cdot \calJ_n = b\cdot \calJ_n$ modulo~$\calI$.  
    
    As an~$R$-module, the quotient~$\tl_n/\calI$ is isomorphic to the direct sum of copies of~$R/aR$, with one summand for each monomial whose Jones normal form ends with~$U_1$. We have that
    \[\calJ_n=(-1)^{n-1}\Big[\left(\frac{\mu }{ \lambda}\right) U_1+\left(\frac{\mu }{ \lambda}\right)^3 U_3U_2U_1+\cdots\Big]\text{ in }\tl_n/\calI\]
    and it follows that
    \[b\cdot \calJ_n = (-1)^{n-1}\Big[b\left(\frac{\mu }{ \lambda}\right)U_1+b\left(\frac{\mu }{ \lambda}\right)^3U_3U_2U_1+\cdots\Big]\text{ in }\tl_n/\calI,\] so~$b$ must vanish in~$R/aR$.  
\end{proof}

\begin{lemma}
    Let~$n$ be even.  Then the image of the map~
    \[\t\otimes_{\tl_n(a)}\F_n(a)\to\t,\qquad1\otimes x\mapsto 1\cdot x,\] is contained in the ideal generated by~$a$.
\end{lemma}

\begin{proof}
    Since the elements~$U_p$ act on~$\t$ as multiplication by~$0$, the map above simply sends~$1\otimes x$ to the constant term of~$x$.  But the previous lemma tells us that the constant term of~$x$ is a multiple of~$a$.
\end{proof}

\begin{proof}[Proof of Theorem~\ref{theorem-sharpness}]
    Let~$n$ be even. From Theorem~\ref{theorem-fineberg}, we have the (fairly short) exact sequence
    \[
        0
        \to
        \Tor_n^{\tl_n}(\t,\t)
        \to 
        \t\otimes_{\tl_n}\F_n
        \to 
        \t
        \to
        \Tor_{n-1}^{\tl_n}(\t,\t)
        \to
        0
    \]
    and the image of~$\t\otimes_{\tl_n}\F_n\to\t$ is contained in the ideal generated by~$a$, and in particular does not contain the element~$1$, so that~$\Tor^{\tl_n}_{n-1}(\t,\t)\neq 0$.
\end{proof}

\section{The case of~\texorpdfstring{$\tl_2(a)$}{n=2}}
\label{section-tl-two}

In this section we briefly consider the case~$n=2$, and fully compute the~$\Tor$ and~$\Ext$ groups. We do this first by a straightforward computation using an explicit free resolution. Then, in order to illustrate the theory developed in the paper, we re-prove the same result by explicitly computing the Fineberg module~$\F_2$ and applying Theorem~\ref{theorem-fineberg}.

\begin{prop}\label{proposition-tltwo}
    The homology and cohomology of~$\tl_2(a)$ are as follows.
    \begin{align*}
        \Tor^{\tl_2(a)}_i(\t,\t)
        &=
        \left\{\begin{array}{ll}
            R, & i=0,
            \\
            R/aR, & i>0,\ i\text{ odd},
            \\
            R_a, & i>0,\ i\text{ even},
        \end{array}\right.
        \\
        \Ext_{\tl_2(a)}^i(\t,\t)
        &=
        \left\{\begin{array}{ll}
            R, & i=0,
            \\
            R_a, & i>0,\ i\text{ odd},
            \\
            R/aR, & i>0,\ i\text{ even},
        \end{array}\right.
    \end{align*}
    where~$R_a$ denotes the kernel of the map~$R\xrightarrow{a}R$.
    This holds for any choice of ground ring~$R$ and any choice of parameter~$a\in R$.
\end{prop}

\begin{proof}

We define a chain complex of left~$\tl_2$-modules as follows.
The degree is indicated in the right-hand column.  The boundary maps are given by right-multiplication by the indicated element of~$\tl_2$, except for the last, which is the map~$\tl_2\to\t$,~$x\mapsto x\cdot 1$.
\[
    \xymatrix{
        \vdots
        \ar[d]^{(a-U_1)}
        \\
        \tl_2
        \ar[d]^{U_1}
        &
        3
        \\
        \tl_2
        \ar[d]^{(a-U_1)}
        &
        2
        \\
        \tl_2
        \ar[d]^{U_1}
        &
        1
        \\
        \tl_2
        \ar[d]
        &
        0
        \\
        \t
        &
        -1
    }
\]
The composite of consecutive boundary maps is~$0$, due to the computation\[U_1\cdot(a-U_1) = 0 = (a-U_1)\cdot U_1,\] and the fact that~$U_1$ acts by~$0$ on~$\t$.  Moreover, this complex is acyclic, as one sees by considering the bases~$1,U_1$ and~$1,(a-U_1)$ of~$\tl_2$.
Thus the non-negative part of the complex above, which we denote by~$P_\ast$, is a free resolution of the left~$\tl_2$-module~$\t$. 
Thus~$\Tor^{\tl_2}_\ast(\t,\t)$ and~$\Ext^\ast_{\tl_2}(\t,\t)$ are the homology of~$\t\otimes_{\tl_2}P_\ast$ and the cohomology of~$\Hom_{\tl_2}(P_\ast,\t)$ respectively. 
Using the isomorphisms~$\t\otimes_{\tl_2}\tl_2\cong\t$,~$a\otimes x\mapsto a\cdot x$ and~$\Hom^{\tl_2}(\tl_2,\t)\cong \t$,~$f\mapsto f(1)$ in every degree, and working out the induced boundary maps, we see that ~$\t\otimes_{\tl_2}P_\ast$ and~$\Hom_{\tl_2}(P_\ast,\t)$ are isomorphic to the complexes depicted below.
\[
    \xymatrix{
        \vdots
        \ar[d]^{a}
        &&
        \vdots
        \ar@{<-}[d]^a
        &&
        \vdots
        \\
        R
        \ar[d]^{0}
        &&
        R
        \ar@{<-}[d]^0
        &&
        3
        \\
        R
        \ar[d]^{a}
        &&
        R
        \ar@{<-}[d]^a
        &&
        2
        \\
        R
        \ar[d]^{0}
        &&
        R
        \ar@{<-}[d]^0
        &&
        1
        \\
        R
        &&
        R
        &&
        0
    }
\]
The homology and cohomology of these complexes are easily computed, and give the claim.
\end{proof}

\begin{prop}\label{prop-F2}
    When~$n=2$ the Fineberg module satisfies~$\F_2(a)\cong\t$, and the map~$\t\otimes_{\tl_2(a)}\F_2(a)\to\varepsilon_2\cong \t$ is multiplication by~$a$.
\end{prop}

\begin{proof}
    We compute~$\F_2$ explicitly in Example~\ref{example - F2}:
    $
        \F_2\cong \langle a-U_1\rangle\cong \t.
    $
    The map~$\t\otimes_{\tl_2}\F_2\to\varepsilon_2\cong \t$ is the composite map \[\t\otimes_{\tl_2}\F_2\to\t\otimes_{\tl_2} W(2)_{1}=\t\otimes_{\tl_2}\left(\tl_2\otimes_{\tl_0}\t\right)\cong\t.\]
    Under the central equality the basis element~$a-U_1$ of~$\F_2\subset W(2)_1$ gets mapped to~$a-U_1=a$ in the tensor product. Therefore the composite map is given by multiplication by~$a$, as required. 
\end{proof}

\begin{corollary}
    Suppose that~$v\in R$ is a unit and that~$a=v+v^{-1}$.
    Then the groups~$\Tor_i^{\tl_2(a)}(\t,\t)$ and~$\Ext_i^{\tl_2(a)}(\t,\t)$ are as described in Proposition~\ref{proposition-tltwo}.
\end{corollary}
\begin{proof}
    In the light of Proposition~\ref{prop-F2}, the exact sequence from Theorem~\ref{theorem-fineberg}
       \[
        0
        \to
        \Tor_2^{\tl_2}(\t,\t)
        \to 
        \t\otimes_{\tl_2}\F_2
        \to 
        \t
        \to
        \Tor_{1}^{\tl_2}(\t,\t)
        \to
        0
    \]
    now becomes
       \[
        0
        \to
        \Tor_2^{\tl_2}(\t,\t)
        \to 
        \t\otimes_{\tl_2}\t
        \overset{ a}{\longrightarrow} 
        \t
        \to
        \Tor_{1}^{\tl_2}(\t,\t)
        \to
        0
    \]
    from which one can compute~$\Tor_2^{\tl_2}(\t,\t)=R_a$ and~$\Tor_1^{\tl_2}(\t,\t)=R/aR$. For $i\geqslant3$ we have the recursive formula
    \[
    \Tor_{i}^{\tl_2}(\t,\t)=\Tor_{i-2}^{\tl_2}(\t,\F_2)\cong \Tor_{i-2}^{\tl_2}(\t,\t)
    \]
    which completes the proof. The~$\Ext$ results similarly follow from Theorem~\ref{theorem-fineberg}.
\end{proof}

\section{High-acyclicity}\label{section-high-acyclicity}

In this final section we prove high connectivity of~$W(n)$, Theorem~\ref{theorem-high-acyclicity}.

\setcounter{abcthm}{4}
\begin{abcthm}
    $H_d(W(n))$ vanishes in degrees~$d\leqslant (n-2)$.
\end{abcthm}

\subsection{A filtration}
\label{section-filtration}

In this subsection we introduce a filtration of~$W(n)$. We state a theorem relating the filtration quotients to~$W(n-1)$ (the proof of which is the topic of the next 3 subsections) and therefore by induction prove Theorem~\ref{theorem-high-acyclicity}.

\begin{defn}[The filtration]\label{definition-filtration}
	We define a filtration~$F$ of~$W(n)$,
	\[
		F^0\subseteq F^1 \subseteq
		\cdots
		\subseteq F^n = W(n)
	\]
	as follows.
	\begin{itemize}
		\item
		$F^0$ is defined to be the span of the elements
		of two kinds.
		We call elements of the first kind \emph{basic elements} and these are of the form 
		\[
			x\otimes 1
		\]
		in degrees~$i$ such that~$-1\leqslant i\leqslant n-2$, 
		where~$x$ is represented by a monomial in the~$s_j$
		not involving the letter~$s_1$.
		Elements of the second kind are those of the form 
		\[
			x\cdot (s_1\cdots s_{n-i-1})\otimes 1
		\]
		in degrees~$i$ such that~$0\leqslant i\leqslant n-1$,
		where again~$x$ is represented by a monomial 
		not involving the letter~$s_1$.

		\item
		$F^k$ for~$k\geqslant 1$ is defined to be the span of
		$F^{k-1}$ together with terms of the form 
		\[
			x\cdot(s_1\cdots s_{n-i-1+k})\otimes 1
		\]
		in degrees~$i$ such that~$k\leqslant i\leqslant n-1$,
		where again~$x$ is 
		represented by a monomial not involving~$s_1$.
	\end{itemize}
\end{defn}

\begin{rem}
	Note that in the description of~$F^0$, it is possible 
	for the product~$s_1\cdots s_{n-i-1}$
	to be empty, i.e.~the unit element, if the final index~$(n-i-1)$
	is zero ($i=n-1$). 
	In contrast, in the description of~$F^k$ for~$k\geqslant 1$,
	the product~$s_1\cdots s_{n-i-1+k}$ is never empty.
	This is one reason why it is important for us to treat
	$F^0$ quite separately from the other~$F^k$, as is done in the remainder of this paper.
\end{rem}

\red{In Theorem \ref{theorem-filtration identification} we show that each~$F^k$ is a subcomplex of~$W(n)$. The fact that~$F^n=W(n)$ will follow from Lemma~\ref{lem-basis for F^k}.}

\begin{defn}\label{defn-cone sus trunc}
	Recall that the \emph{cone} on a chain complex~$X$ (or, more precisely, the cone on the identity map of~$X$)
	is the chain complex~$CX$ defined by~$(CX)_i = X_i\oplus X_{i-1}$,
	and with differential defined by 
	\[
		d^i_{CX}(x,y) = (d^i_X(x)+y,-d^{i-1}_X(y)).
	\]
	The \emph{suspension} of a chain complex~$X$ is the complex
	$\Sigma X$ defined by
	\[(\Sigma X)_i = X_{i-1}\] and with 
	the same differential as~$X$.
	The \emph{truncation to degree~$p$} of a chain complex~$X$
	is the chain complex~$\tau_pX$ defined by
	\[
		(\tau_pX)_i
		=
		\left\{\begin{array}{ll}
			X_i, & i\leqslant p\\
			0, & i>p
		\end{array}\right.
	\]
	and with the same differential as~$X$ (in the relevant degrees).
\end{defn}

\begin{rem}
	Note that our definition of cone and suspension do not seem
	to match up very well.  However, we have chosen our conventions
	in order to make the proof of the next theorem as direct as possible,
	and we believe that our choices are the best fit for this 
	purpose.
\end{rem}

\begin{defn}\label{defn-sigma shift map}
    Define the \emph{shift map}~$\sigma$ to be the map
    \[
    \sigma \colon \tl_{n-1}(a)\to\tl_n(a) 
    \]
    which sends each~$U_i$ to~$U_{i+1}$ for~$1\leqslant i \leqslant n-2$, and hence each~$s_i$ to~$s_{i+1}$.
\end{defn}

\begin{lemma}
    Each~$F^k$ consists of~$\tl_{n-1}(a)$-submodules of~$W(n)$, where~$\tl_{n-1}(a)$ acts via the shift map~$\sigma$.
    
    \begin{proof}
        Definition~\ref{definition-filtration} defines each~$F^k$ as the span of certain `base elements' of the form~$y\otimes 1$ where~$y\in\tl_n$ is represented by a monomial in the~$s_j$ subject to certain restrictions.
        Multiplying any such~$y$ on the left by any~$s_j$ for~$1<j\leqslant n-1$ does not affect whether it meets these restrictions.
        Since~$s_j = \sigma(s_{j-1})$ for~$1<j\leqslant n-1$,
        this shows that the generators of~$\tl_{n-1}$ send the base elements of each~$F^k$ to other base elements of~$F^k$, and therefore~$F^k$ itself is stable under the action of~$\tl_{n-1}$.
    \end{proof}
\end{lemma}

Here is the main result of this section.

\begin{thm}\label{theorem-filtration identification}
	Each~$F^k$ is a subcomplex of~$W(n)$. We identify
	\[
		F^0\cong C(W(n-1)).
	\]
	And for~$k\geqslant 1$, we have
	\[
		F^k/F^{k-1}\cong\tau_{n-1}\Sigma^{k+1}W(n-1).
	\]
\end{thm}

\begin{cor}[Theorem~\ref{theorem-high-acyclicity}]\label{corollary-theorem-E}
	For each~$n\geqslant 0$ the complex~$W(n)$ is~$(n-2)$-acyclic, or in other words, its homology vanishes up to and including degree~$(n-2)$.
\end{cor}

\begin{proof}
    We prove this by induction on~$n\geqslant 0$.
	One can verify the claim directly in the case
	$n=0$.  Fix~$n\geqslant 1$ and suppose that the
	theorem has been proved for the previous case.
	Now~$W(n)$ has the filtration~$F^0\subseteq F^1
	\subseteq \cdots\subseteq F^n$.
	We prove below that~$F^0$ and all filtration quotients
	$F^k/F^{k-1}$ are~$(n-2)$-acyclic, and then it follows (for example by using the short exact sequences
	$0\to F^{k-1}\to F^k\to F^k/F^{k-1}\to 0$, or by using the spectral sequence of the filtration) that the same holds for~$W(n)$ itself.

	Observe that~$F^0\cong C(W(n-1))$, being isomorphic to a cone, is acyclic.
	Next, for~$k\geqslant 1$ we have~$F^k/F^{k-1} \cong\tau_{n-1}\Sigma^{k+1}W(n-1)$.  
	The induction hypothesis states that~$W(n-1)$ is~$(n-3)$-acyclic, \red{so that~$\Sigma^{k+1}W(n-1)$ is~$(n-2+k)$-acyclic} and in particular~$(n-2)$-acyclic, so that~$\tau_{n-1}\Sigma^{k+1}W(n-1)$ is also~$(n-2)$-acyclic.
	This completes the proof.
\end{proof}

\red{
\begin{remark}[Intuitions and motivations]

The complex of planar injective words $W(n)$ is an analogue of
the complex of injective words $\calC(n)$, 
and
Theorem~\ref{theorem-high-acyclicity}
is the analogue for $W(n)$ of the well-known vanishing result for the homology
of $\calC(n)$; see Remark~\ref{remark-Cn}.

Our starting point in proving Theorem~\ref{theorem-high-acyclicity} was
Kerz's proof~\cite{Kerz} of the vanishing theorem for the homology of
$\calC(n)$.
Kerz identifies within $\calC(n)$ a subcomplex $F^0$ that is isomorphic to 
the cone $C(\calC(n-1))$.
This is then extended to a filtration 
$F^0\subseteq F^1\subseteq\cdots\subseteq F^{n-1}\subseteq \calC(n)$
in which each subsequent filtration quotient $F^k/F^{k-1}$ is isomorphic
to a direct sum of copies of the suspension $\Sigma^{k+1}\calC(n-k-1)$.
(In fact Kerz does not explicitly mention filtrations, but this is one way of 
framing his proof.)
This permits an inductive proof of high-acyclicity as in
Corollary~\ref{corollary-theorem-E}.

Our proof of Theorem~\ref{theorem-high-acyclicity} began as an attempt to 
mimic Kerz's approach.
There is an evident way to embed $W(n-1)$ into $W(n)$ 
--- this is the span of the basic elements of $F^0$ --- 
and this can be extended to an embedding of the cone $C(W(n-1))$ 
into $W(n)$ by considering the elements of the second kind in $F^0$.
The remainder of our proof is the result of attempting 
to extend this embedding to a complete filtration of $W(n)$.
At this stage the parallels with~\cite{Kerz} begin to fail, but the Jones normal form gives us an extra tool. Using this we characterise the basis elements of~$W(n)$ that are not in the image of the cone $C(W(n-1))$, and this characterisation gives a surprising separation into subcomplexes which `look like' suspended and truncated copies of~$W(n-1)$ --- we build our filtration such that these are our filtration quotients~$F^k/F^{k-1}$. 
\end{remark}
}

The final three subsections prove Theorem~\ref{theorem-filtration identification}, by first setting up the required chain map for~$F^0$, then for~$F^k$ and then in the final section proving these chain maps are isomorphisms.

\subsection{Proofs for~$F^0$}
\label{section-F-zero}

In this subsection we prove~$F^0$ is a subcomplex of~$W(n)$. We define a map from the cone~$C(W(n-1))$ to~$F^0$ and prove this is a well defined chain map. 

\begin{lemma}
	$F^0$ is a subcomplex of~$W(n)$.
\end{lemma}

\begin{proof}
	To prove the claim, we must take a generator of 
	$F^0$ in degree~$i$, and show that
	under the boundary map~$d^i\colon W(n)_i\to W(n)_{i-1}$
	this generator is mapped into~$F^0$.
	Since~$d^i$ is the alternating sum ~$d^i_0-d^i_1+\cdots +(-1)^i d^i_i$,
	it will suffice to fix~$j$ in the range~$0\leqslant j\leqslant i$,
	and show that~$d^i_j$ sends our generator into~$F^0$.
	Recall from Definition~\ref{defn: W(n) and boundary maps} the definition of~$d^i_j$:
	\[
		d^i_j(y\otimes r)
		=
		y\cdot (s_{n-i+j-1}\cdots s_{n-i}) \otimes \lambda^{-j} r.
	\]

	Generators of~$F^0$ come in two kinds.
	The first kind are the basic elements~$x\otimes 1$
	in degrees~$-1\leqslant i\leqslant n-2$ where~$x$ is represented by
	a monomial not featuring the letter~$s_1$.
	The map~$d^i_j$ only introduces a letter~$s_1$ in the case~$i=n-1$, which is excluded here, so that~$d^i_j(x\otimes 1)$ \red{is again a basic element and therefore} also lies
	in~$F^0$.

	The second kind of generators of~$F^0$ are elements
	\[
		x\cdot (s_1\cdots s_{n-i-1})\otimes 1
	\]
	in degrees~$0\leqslant i\leqslant n-1$,
	where~$x$ is represented by a monomial not involving~$s_1$.
	In the case~$j=0$, we have
	\[
		d^i_0(x\cdot (s_1\cdots s_{n-i-1})\otimes 1)
		= x\cdot(s_1\cdots s_{n-i-1})\otimes 1
	\]
	but this lies in~$W(n)_{i-1}=\tl_n\otimes_{\tl_{n-i}}\t$,
	hence is equal to~$x\otimes\lambda^{n-i-1}$, and since~$x$ is 
	represented by a monomial not involving~$s_1$, this does indeed
	lie in~$F^0$.
	(This argument includes the special case~$i=n-1$, where the product
	$s_1\cdots s_{n-i-1}$ is empty, but this clearly creates no issues.)
	In the case~$j\geqslant 1$, we have
	\begin{align*}
		d^i_j(x\cdot (s_1\cdots s_{n-i-1})\otimes 1)
		&= 
		x
		\cdot
		(s_1\cdots s_{n-i-1})
		\cdot 
		(s_{n-i+j-1}\cdots s_{n-i})
		\otimes \lambda^{-j}
		\\
		&= 
		x
		\cdot
		(s_1\cdots s_{n-i-1})
		\cdot 
		(s_{n-i+j-1}\cdots s_{n-i+1})
		\cdot s_{n-i}
		\otimes \lambda^{-j}
		\\
		&= 
		x
		\cdot
		(s_{n-i+j-1}\cdots s_{n-i+1})\cdot
		(s_1\cdots s_{n-i-1})
		\cdot s_{n-i}
		\otimes \lambda^{-j}
		\\
		&= 
		x
		\cdot
		(s_{n-i+j-1}\cdots s_{n-i+1})
		\cdot 
		(s_1\cdots s_{n-i})
		\otimes 
		\lambda^{-j}
		\\
		&= 
		(x
		\cdot
		(s_{n-i+j-1}\cdots s_{n-i+1}))
		\cdot 
		(s_1\cdots s_{n-(i-1)-1})
		\otimes 
		\lambda^{-j}
	\end{align*}
	which lies in~$F^0$ since~$x\cdot(s_{n-i+j-1}\cdots s_{n-i+1})$, does not involve the letter~$s_1$ \red{and so $d^i_j(x\cdot (s_1\cdots s_{n-i-1})\otimes 1)$ is a scalar multiple of a generator of $F^0$, and thus in~$F^0$, as required.}
\end{proof}

\begin{defn}\label{defn-phi for k=0}
	Define a map
	\[
		\Phi^0\colon C(W(n-1))\longrightarrow F^0
	\]
	as follows.
	Recall that
	\begin{align*}
		C(W(n-1))_i 
		&= W(n-1)_i\oplus W(n-1)_{i-1}
		\\
		&= \left(\tl_{n-1}(a)\otimes_{\tl_{n-i-2}(a)}\t\right)
		\oplus
		\left(\tl_{n-1}(a)\otimes_{\tl_{n-i-1}(a)}\t\right)
	\end{align*}
	and that
	\[
		F^0_i\subseteq W(n)_i=\tl_n(a)\otimes_{\tl_{n-i-1}(a)}\t.
	\]
	We define~$\Phi^0$ in degree~$i$ by the rule	
	\[
		\Phi^0_i(x\otimes\alpha,y\otimes\beta) 
		= 
		\xi_i(x\otimes\alpha) + \eta_i(y\otimes\beta)
	\]
	where
    \begin{align*}
        \xi_i\colon W(n-1)_i &\to W(n)_i\\    
      x\otimes\alpha &\mapsto \sigma(x)\otimes\lambda^{n-1}\alpha  
    \end{align*}
	and
	\begin{align*}
        \eta_i\colon W(n-1)_{i-1} &\to W(n)_i\\    	y\otimes\beta &\mapsto \sigma(y)\cdot(s_1\cdots s_{n-i-1})
		\otimes\lambda^i\beta.
	\end{align*}
	It is simple to check that the image of both maps lies in~$F^0_i$.
\end{defn}

\begin{lemma}
	The maps~$\xi_i$ and~$\eta_i$ are well defined.
\end{lemma}

\begin{proof}
	In the case of~$\xi_i$ this is simple to verify, as the
	map~$\sigma\colon\tl_{n-1}\to\tl_n$ is in fact a map of
	right-modules
	with respect to the map of algebras
	$\sigma\colon \tl_{n-i-2}\to\tl_{n-i-1}$.

	In the case of~$\eta_i$, the definition of~$\eta_i(y\otimes\beta)$
	as presented depends on~$y$ and~$\beta$ themselves, and we must
	check that it depends only on~$y\otimes\beta$.
	Thus we must show that
	\[
		\eta_i(ys_j\otimes\beta) = \eta_i(y\otimes\lambda\beta)
	\]
	whenever~$1\leqslant j\leqslant n-i-2$.
	And indeed
	\begin{align*}
		\eta_i(ys_j\otimes\beta)
		&=
		\sigma(ys_j)\cdot(s_1\cdots s_{n-i-1})\otimes\lambda^i\beta
		\\
		&=
		\sigma(y)\cdot s_{j+1}
		\cdot(s_1\cdots s_{n-i-1})\otimes\lambda^i\beta
		\\
		&=
		\sigma(y)
		\cdot(s_1\cdots s_{n-i-1})\cdot s_j\otimes\lambda^i\beta
		\\
		&=
		\sigma(y)
		\cdot(s_1\cdots s_{n-i-1})\otimes\lambda^{i+1}\beta
		\\
		&=
		\eta_i(y\otimes\lambda\beta)
	\end{align*}
	where the third equality holds since~$2\leqslant j+1\leqslant n-i-1$ \red{(a simple way to see this is to draw the~$s_i$ as braids)},
	and the fourth holds since~$j\leqslant n-i-2$ and the tensor product
	is over~$\tl_{n-i-1}$.
\end{proof}

\begin{lemma}
	The~$\xi_i$ and~$\eta_i$ interact with the boundary
	maps of~$W(n)$ in the following way:
	\begin{enumerate}
		\item
		$d^i_j\circ\xi_i = \xi_{i-1}\circ d^{i}_j$
		for~$i$ in the range~$-1\leqslant i \leqslant n-2$
		and~$j$ in the range $0\leqslant j\leqslant i$.

		\item
		$d^i_0\circ\eta_i = \xi_{i-1}$ for~$i$ in the
		range~$0\leqslant i\leqslant n-1$.

		\item
		$d^i_{j+1}\circ \eta_i = \eta_{i-1}\circ d^{i-1}_j$
		for~$i$ in the range~$0\leqslant i\leqslant n-1$
		and~$j$ in the range~$0\leqslant j \leqslant i-1$.
	\end{enumerate}
\end{lemma}

\begin{proof}
	For the first point, we have:
	\begin{align*}
		d_j^i(\xi_i(x\otimes\alpha))
		&=
		d_j^i(\sigma(x)\otimes\lambda^{n-1}\alpha)
		\\
		&=
		\sigma(x)\cdot(s_{n-i+j-1}\cdots s_{n-i})
		\otimes\lambda^{-j}\lambda^{n-1}\alpha
		\\
		&=
		\sigma(x\cdot(s_{n-i+j-2}\cdots s_{n-i-1}))
		\otimes\lambda^{-j}\lambda^{n-1}\alpha
		\\
		&=
		\xi_{i-1}(x\cdot(s_{n-i+j-2}\cdots s_{n-i-1})
		\otimes\lambda^{-j}\alpha)
		\\
		&=
		\xi_{i-1}(x\cdot(s_{(n-1)-i+j-1}\cdots s_{(n-1)-i})
		\otimes\lambda^{-j}\alpha)
		\\
		&=
		\xi_{i-1}(d^i_j(x\otimes\alpha)).
	\end{align*}
	For the second point, we have:
	\begin{align*}
		d^i_0(\eta_i(y\otimes\beta))
		&=
		d^i_0(\sigma(y)\cdot (s_1\cdots s_{n-i-1})\otimes\lambda^i\beta)
		\\
		&=
		\sigma(y)\cdot (s_1\cdots s_{n-i-1})\otimes\lambda^i\beta
		\\
		&=
		\sigma(y)\otimes\lambda^{n-i-1}\lambda^i\beta
		\\
		&=
		\sigma(y)\otimes\lambda^{n-1}\beta
		\\
		&=
		\xi_{i-1}(y\otimes\beta),
	\end{align*}
	where the third equality holds because the terms lie in
	$W(n)_{i-1} = \tl_n\ootimes{n-i}\t$.
	And for the third point we have:
	\begin{align*}
		d^i_{j+1}\eta_i(y\otimes\beta)
		&=
		d^i_{j+1}(\sigma(y)\cdot(s_1\cdots s_{n-i-1})
		\otimes\lambda^i\beta)
		\\
		&=
		\sigma(y)
		\cdot(s_1\cdots s_{n-i-1})
		\cdot(s_{n-i+(j+1)-1}\cdots s_{n-i})
		\otimes\lambda^{-j-1}\lambda^i\beta
		\\
		&=
		\sigma(y)
		\cdot(s_1\cdots s_{n-i-1})
		\cdot(s_{n-i+j}\cdots s_{n-i+1})
		\cdot s_{n-i}
		\otimes\lambda^{i-j-1}\beta
		\\
		&=
		\sigma(y)
		\cdot(s_{n-i+j}\cdots s_{n-i+1})
		\cdot(s_1\cdots s_{n-i})
		\otimes\lambda^{i-j-1}\beta
		\\
		&=
		\sigma(y
		\cdot (s_{n-i+j-1}\cdots s_{n-i}))
		\cdot(s_1\cdots s_{n-(i-1)-1})
		\otimes\lambda^{i-1}\lambda^{-j}\beta
		\\
		&=
		\eta_{i-1}(y\cdot (s_{n-i+j-1}\cdots s_{n-i}) 
		\otimes\lambda^{-j}\beta)
		\\
		&=
		\eta_{i-1}(y\cdot (s_{(n-1)-(i-1)+j-1}\cdots s_{(n-1)-(i-1)}) 
		\otimes\lambda^{-j}\beta)
		\\
		&=
		\eta_{i-1}(d^{i-1}_j(y \otimes\beta))
	\end{align*}
	where for the final equality we recall that the source of~$\eta_{i-1}$ is~$W(n-1)_{i-2}$.
\end{proof}

\begin{lemma}
	$\Phi^0$ is a chain map.
\end{lemma}

\begin{proof}
	Referring to the definition of the differential on
	$C(W(n-1))$ (Definition~\ref{defn-cone sus trunc}), we see that in order to check that 
	$d^i\circ\Phi^0_i = \Phi^0_{i-1}\circ d^i$,
	it is enough to show that 
	$d^i\circ\xi_i(x\otimes\alpha)=\xi_{i-1}(d^i(x\otimes\alpha))$
	and
	$d^i\circ\eta_i(y\otimes\beta) = \xi_{i-1}(y\otimes\beta)
	-
	\eta_{i-1}(d^{i-1}(y\otimes\beta))$.
	Using the previous lemma, for the first we have
	\begin{align*}
		d^i\circ\xi_i(x\otimes\alpha)
		&=
		\sum_{j=0}^i(-1)^j d^i_j(\xi_i(x\otimes\alpha))
		\\
		&=
		\sum_{j=0}^i(-1)^j \xi_{i-1}(d^i_j(x\otimes\alpha))
		\\
		&=
		\xi_{i-1}\left(\sum_{j=0}^i(-1)^j d^i_j(x\otimes\alpha)\right)
		\\
		&=\xi_{i-1}(d^i(x\otimes\alpha)).
	\end{align*}
	And for the second we have
	\begin{align*}
		d^i\circ\eta_i(y\otimes\beta) 
		&=
		\sum_{j=0}^i (-1)^j d^i_j(\eta_i(y\otimes\beta))
		\\
		&=
		d^i_0(\eta_i(y\otimes\beta))
		-
		\sum_{j=0}^{i-1}(-1)^j d^i_{j+1}\eta_i(y\otimes\beta)
		\\
		&=
		\xi_{i-1}(y\otimes\beta)
		-
		\sum_{j=0}^{i-1}(-1)^j\eta_{i-1}d^{i-1}_j(y\otimes\beta)
		\\
		&=
		\xi_{i-1}(y\otimes\beta)
		-
		\eta_{i-1}\left(\sum_{j=0}^{i-1}(-1)^jd^{i-1}_j(y\otimes\beta)\right)
		\\
		&= \xi_{i-1}(y\otimes\beta)
		-
		\eta_{i-1}(d^{i-1}(y\otimes\beta)).
		\qedhere
	\end{align*}
\end{proof}

\subsection{Proofs for~$F^k$,~$k\geqslant 1$}
\label{section-F-k}

In this subsection we prove, for~$k\geqslant 1$, that~$F^k$ is a subcomplex of~$W(n)$. We define a map from~$\tau_{n-1}\Sigma^{k+1}W(n-1)$ to~$F^k/F^{k-1}$ and prove this is a well defined chain map. We start off with some elementary lemmas involving the~$s_j$, which we require for later proofs.

\begin{lemma} \label{lemma-towers}
	Let~$m\geqslant 1$ and~$p\leqslant m$.
	Then
	\[
		s_1\cdots s_m
		\cdots s_p
		=
		(s_m \cdots s_{p+1})
		\cdot
		(s_1\cdots s_m).
	\]
	In the case~$m=p$ the product~$s_m\cdots s_{p+1}$ is empty and therefore equal to~$1$.
\end{lemma}

\begin{lemma}\label{lemma-products}
	Let~$p\geqslant 1$,~$q\geqslant r\geqslant 1$.
	Then the product~$(s_1\cdots s_p) \cdot (s_q\cdots s_r)$
	can be described as follows.
	\begin{enumerate}
		\item
		When~$r-1\leqslant p\leqslant q-1$, 
		\[
			(s_1\cdots s_p) \cdot (s_q\cdots s_r)
			=
			(s_q\cdots s_{r+1}) \cdot (s_1\cdots s_{p+1}).
		\]

		\item
		When~$p=q$,~$(s_1\cdots s_p) \cdot (s_q\cdots s_r)$
		is a linear combination of terms of the form
		$(s_t\cdots s_{r+1})\cdot (s_1\cdots s_t)$
		for~$p\geqslant t\geqslant r+1$, as well as
		$s_1\cdots s_r$ and~$s_1\cdots s_{r-1}$.

		\item
		When~$p\geqslant q+1$,
		\[
			(s_1\cdots s_p) \cdot (s_q\cdots s_r)
			=
			(s_{q+1}\cdots s_{r+1})\cdot (s_1\cdots s_p).
		\]
	\end{enumerate}
\end{lemma}

\begin{proof}
	When~$r-1\leqslant p\leqslant q-1$, 
	\begin{align*}
		(s_1\cdots s_p) \cdot (s_q\cdots s_r)
		&=
		(s_1\cdots s_p) 
		\cdot
		(s_q\cdots s_{p+2})
		\cdot 
		(s_{p+1}\cdots s_r)
		\\
		&=
		(s_q\cdots s_{p+2})\cdot 
		(s_1\cdots s_p) \cdot (s_{p+1}\cdots s_r)
		\\
		&=
		(s_q\cdots s_{p+2})\cdot 
		(s_1\cdots s_{p+1}\cdots s_r)
		\\
		&=
		(s_q\cdots s_{p+2})\cdot 
		(s_{p+1}\cdots s_{r+1})\cdot
		(s_1\cdots s_{p+1})
		\\
		&=
		(s_q\cdots s_{r+1}) \cdot (s_1\cdots s_{p+1}),
	\end{align*}
	where we used Lemma~\ref{lemma-towers} to obtain the 
	fourth equality.

	When~$p=q$, we claim that 
	\[
		(s_1\cdots s_p) \cdot (s_q\cdots s_r)
		=
		(s_1\cdots s_p) \cdot (s_p\cdots s_r)
	\]
	is a linear combination of terms of the form
	$(s_t\cdots s_{r+1})\cdot (s_1\cdots s_t)$
	for~$p\geqslant t\geqslant r+1$, as well as
	$s_1\cdots s_r$ and~$s_1\cdots s_{r-1}$.
	We will prove this claim by induction on the difference
	$p-r$.
	When~$p-r=0$, we have
	\[
		(s_1\cdots s_p) \cdot (s_p\cdots s_r)
		=
		s_1\cdots s_p \cdot s_p.
	\]
	Now since~$s_p^2$ is a linear combination of~$s_p$ and~$1$, this is a linear combination of
	$s_1\cdots s_p=s_1\cdots s_r$
	and~$s_1\cdots s_{p-1}=s_1\cdots s_{r-1}$ as required.
	Now let~$p-r\geqslant 1$, and assume that the claim
	holds for all smaller values.  Then
	\[
		(s_1\cdots s_p) \cdot (s_p\cdots s_r)
		=
		(s_1\cdots s_{p-1})\cdot s_p^2\cdot (s_{p-1}\cdots s_r)
	\]
	is a linear combination of
	\begin{align*}
		(s_1\cdots s_{p-1})\cdot s_p\cdot (s_{p-1}\cdots s_r)
		&=
		s_1\cdots s_p\cdots s_r
		\\
		&=
		(s_p\cdots s_{r+1})\cdot (s_1\cdots s_p)
	\end{align*}
	(where we used Lemma~\ref{lemma-towers})
	and
	\[
		(s_1\cdots s_{p-1})\cdot (s_{p-1}\cdots s_r).
	\]
	The former is 
	$(s_t\cdots s_{r+1})\cdot (s_1\cdots s_t)$
	in the case~$t=p$, while the induction hypothesis
	tells us that the latter is a linear combination of
	$(s_t\cdots s_{r+1})\cdot (s_1\cdots s_t)$
	for~$p-1\geqslant t\geqslant r+1$, as well as
	$s_1\cdots s_r$ and~$s_1\cdots s_{r-1}$.
	This completes the proof of the claim.

	When~$p\geqslant q+1$,
	\begin{align*}
		(s_1\cdots s_p) \cdot (s_q\cdots s_r)
		&=
		(s_1\cdots s_{q+1})
		\cdot
		(s_{q+2}\cdots s_p)	
		\cdot 
		(s_q\cdots s_r)
		\\
		&=
		(s_1\cdots s_{q+1})
		\cdot 
		(s_q\cdots s_r)
		\cdot
		(s_{q+2}\cdots s_p)	
		\\
		&=
		(s_1\cdots s_{q+1}\cdots s_r)
		\cdot
		(s_{q+2}\cdots s_p)	
		\\
		&=
		(s_{q+1}\cdots s_{r+1})\cdot (s_1\cdots s_{q+1})
		\cdot
		(s_{q+2}\cdots s_p)	
		\\
		&=
		(s_{q+1}\cdots s_{r+1})\cdot (s_1\cdots s_p)
	\end{align*}
	(where we again used Lemma~\ref{lemma-towers} 
	to obtain the fourth equality) as required.
\end{proof}

\begin{lemma}\label{lemma-subcomplex}
	For~$k\geqslant 1$,~$F^k$ is a subcomplex of~$W(n)$.
\end{lemma}

\begin{proof}
	We fix~$k\geqslant 1$ and
	take a generator of 
	\red{$F^k/ F^{k-1}$} in degree~$i$, 
	where $k\leqslant i\leqslant n-1$,
	and show that
	the boundary map $d^i\colon W(n)_i\to W(n)_{i-1}$
	sends our generator into~$F^k$.
	Since~$d$ is the alternating sum $d^i_0-d^i_1+\cdots +(-1)^i d^i_i$,
	it will suffice to fix~$j$ in the range~$0\leqslant j\leqslant i$,
	and show that~$d^i_j$ sends our generator into~$F^k$.
	\red{Recall from Definition~\ref{definition-filtration} that} our generator of~\red{$F^k/ F^{k-1}$} in degree~$i$
	is~$x\cdot(s_1\cdots s_{n-i-1+k})\otimes 1$, where~$x$ does not involve the letter~$s_1$.
	Note that
	\[(n-i-1+k) = (n-1)-i+k\geqslant (n-1)-(n-1)+1=1,\]
	so that the product~$(s_1\cdots s_{n-i-1+k})$ is not empty.
	We have
	\[
		d^i_j(x\cdot(s_1\cdots s_{n-i-1+k})\otimes 1)
		=
		x\cdot (s_1\cdots s_{n-i-1+k})\cdot (s_{n-i-1+j}\cdots s_{n-i})
		\otimes
		\lambda^{-j},
	\]
	where the factor~$(s_{n-i-1+j}\cdots s_{n-i})$ can be empty,
	in the case~$j=0$.
	
	\begin{itemize}
		\item
		First we consider the case~$j=0$.  We find that
		\begin{align*}
			d^i_0(x\cdot(s_1\cdots s_{n-i-1+k})\otimes 1)
			&=
			x\cdot(s_1\cdots s_{n-i-1+k})\otimes 1
			\\
			&=
			x\cdot(s_1\cdots s_{n-(i-1)-1+(k-1)})\otimes 1
		\end{align*}
		lies in~$F^{k-1}$, and therefore in~$F^k$ as required.

		\item
		Now we consider the case~$1\leqslant j \leqslant (k-1)$.
		Then~$(n-i-1+k)\geqslant (n-i-1+j)+1$, so that the
		third item of Lemma~\ref{lemma-products} applies and
		shows that
		\begin{align*}
			d^i_j(x\cdot&(s_1\cdots s_{n-i-1+k})\otimes 1)
			\\
			&=
			x
			\cdot (s_1\cdots s_{n-i-1+k})
			\cdot (s_{n-i-1+j}\cdots s_{n-i})
			\otimes \lambda^{-j}
			\\
			&=
			x
			\cdot (s_{n-i+j}\cdots s_{n-i+1})
			\cdot (s_1\cdots s_{n-i-1+k})
			\otimes \lambda^{-j}
			\\
			&=
			x
			\cdot (s_{n-i+j}\cdots s_{n-i+1})
			\cdot (s_1\cdots s_{n-(i-1)-1+(k-1)})
			\otimes \lambda^{-j}
		\end{align*}
		Since~$n-i+1\geqslant n-(n-1)+1=2$, the word
		$(s_{n-i+j}\cdots s_{n-i+1})$ does not involve~$s_1$,
		and consequently the element above lies in~$F^{k-1}$,
		and therefore in~$F^k$.

		\item
		Now we consider the case~$j=k$.
		Then~$(n-i-1+k) = (n-i-1+j)$ and so the second item of
		Lemma~\ref{lemma-products} applies and shows that
		\begin{align*}
			d^i_k(x\cdot(s_1\cdots s_{n-i-1+k})\otimes 1)
			&=
			\\
			x
			\cdot (s_1\cdots s_{n-i-1+k})
			&\cdot (s_{n-i-1+k}\cdots s_{n-i})
			\otimes \lambda^{-k}
		\end{align*}
		is a linear combination of terms
		\[
			x
			\cdot (s_t\cdots s_{n-i+1})
			\cdot (s_1\cdots s_t)
			\otimes \lambda^{-k}
		\]
		for~$t$ in the range 
		\[
			(n-i+1)\leqslant t\leqslant (n-i-1+k)
			= (n-(i-1)-1 + (k-1))
		\]
		together with
		\[
			x
			\cdot (s_1\cdots s_{n-(i-1)-1})
			\otimes \lambda^{-k}
		\]
		and
		\[
			x
			\cdot (s_1\cdots s_{n-(i-1)-2})
			\otimes \lambda^{-k}
			=
			x\otimes \lambda^{-k}.
		\]
		Now~$(s_t\cdots s_{n-i+1})$ does not involve~$s_1$,
		so the first of these terms lies in~$F^{k-1}$,
		while the second and third lie in~$F^0$.
		So altogether we have the required result.

		\item
		Now we consider the case~$k+1\leqslant j$.
		Here we have
		\[(n-i-1)\leqslant (n-i-1+k)+1\leqslant (n-i-1+j),\]
		so that the first item of Lemma~\ref{lemma-products}
		applies and shows that 
		\begin{align*}
			d^i_j(x\cdot(s_1\cdots s_{n-i-1+k})\otimes 1)
			&=
			x
			\cdot (s_1\cdots s_{n-i-1+k})
			\cdot (s_{n-i-1+j}\cdots s_{n-i})
			\otimes
			\lambda^{-j}
			\\
			&=
			x
			\cdot (s_{n-i-1+j}\cdots s_{n-i+1})
			\cdot (s_1\cdots s_{n-i-1+k+1})
			\otimes
			\lambda^{-j}
			\\
			&=
			x
			\cdot (s_{n-i-1+j}\cdots s_{n-i+1})
			\cdot (s_1\cdots s_{n-(i-1)-1+k})
			\otimes
			\lambda^{-j}.
		\end{align*}
		Since~$(s_{n-i-1+j}\cdots s_{n-i+1})$ does not involve~$s_1$,
		the element above lies in~$F^k$ as required.\qedhere
	\end{itemize}
\end{proof}

\begin{defn}\label{defn-phi for k>0}
	Define a map
	\[
		\Psi^k
		\colon
		\tau_{n-1}\Sigma^{k+1}W(n-1)
		\longrightarrow 
		F^k/F^{k-1}
	\]
	as follows.
	Note that for~$i$ in the range~$k\leqslant i\leqslant (n-1)$,
	\begin{align*}
		[\tau_{n-1}\Sigma^{k+1}W(n-1)]_i
		&=
		W(n-1)_{i-k-1} 
		\\
		&= 
		\tl_{n-1}(a)\otimes_{\tl_{(n-1)-(i-k-1)-1}(a)}\t
		\\
		&=
		\tl_{n-1}(a)\otimes_{\tl_{n-i-1+k}(a)}\t,
	\end{align*}
	while~$(F^k/F^{k-1})_i$ is a quotient of~$\tl_n(a)\otimes_{\tl_{n-i-1}(a)}\t$.
	Define the degree~$i$ part of~$\Psi$ to be the map
	\[
		\Psi^k_i\colon \tl_{n-1}(a)\otimes_{\tl_{n-i-1+k}(a)}\t
		\longrightarrow
		(F^k/F^{k-1})_i
	\]
	given by	
	\[
		\Psi^k_i\colon x\otimes\alpha
		\longmapsto
		(-1)^{-i(k+1)}\sigma(x)
		\cdot (s_1\cdots s_{n-i-1+k})\otimes\lambda^i\alpha.
	\]
	For later convenience, we will
	denote by~$\psi^k_i$ the map
	\[
		\psi^k_i\colon x\otimes\alpha
		\longmapsto
		\sigma(x)
		\cdot (s_1\cdots s_{n-i-1+k})\otimes\lambda^i\alpha,
	\]
	so that~$\Psi^k_i=(-1)^{-i(k+1)}\psi^k_i$.
\end{defn}

\begin{lemma}
	The map $\psi^k_i$ is well defined
	(and the same therefore holds for~$\Psi^k_i$).
\end{lemma}

\begin{proof}
	As presented above, 
	the value of~$\psi^k_i(x\otimes\alpha)$ depends on the choices
	of~$x$ and~$\alpha$, rather than on~$x\otimes \alpha$.  
	So to check that~$\psi^k_i$ is well-defined, we must
	check that~$\psi^k_i(xs_p\otimes \alpha)=\psi^k_i(x\otimes\lambda\alpha)$
	whenever~$p\leqslant (n-i-1+k)-1$.
	Let us write~$q=(n-i-1+k)$, so that~$p\leqslant q-1$.
	(In particular we are assuming that~$q\geqslant 2$.)
	Now
	\begin{align*}
		\psi^k_i(xs_p\otimes \alpha)
		&=
		\sigma(xs_p)\cdot(s_1\cdots s_q)\otimes\lambda^i\alpha
		\\
		&=
		\sigma(x)\cdot s_{p+1}\cdot(s_1\cdots s_q)\otimes\lambda^i\alpha
		\\
		&=
		\sigma(x)
		\cdot s_{p+1}
		\cdot (s_1\cdots s_{p-1})
		\cdot(s_p s_{p+1})
		\cdot(s_{p+2}\cdots s_q)
		\otimes\lambda^i\alpha
		\\
		&=
		\sigma(x)
		\cdot (s_1\cdots s_{p-1})
		\cdot(s_{p+1}s_p s_{p+1})
		\cdot(s_{p+2}\cdots s_q)
		\otimes\lambda^i\alpha.
	\end{align*}
	Recall from Definition~\ref{definition-si} that
	\[
		s_{p+1}s_ps_{p+1}
		=
		\lambda s_ps_{p+1}
		+
		\lambda s_{p+1}s_p
		-
		\lambda^2 s_p
		-
		\lambda^2 s_{p+1}
		+
		\lambda^3.
	\]
	Now 
	\begin{align*}
		(s_1\cdots s_{p-1})
		&\cdot(s_p s_{p+1})
		\cdot(s_{p+2}\cdots s_q)
		= (s_1\cdots s_q)
		\\
		(s_1\cdots s_{p-1})
		&\cdot(s_{p+1}s_p)
		\cdot(s_{p+2}\cdots s_q)
		=
		(s_{p+1}\cdots s_q)\cdot(s_1\cdots s_p)
		\\
		(s_1\cdots s_{p-1})
		&\cdot s_p 
		\cdot(s_{p+2}\cdots s_q)
		=
		(s_{p+2}\cdots s_q)\cdot (s_1\cdots s_p)
		\\
		(s_1\cdots s_{p-1})
		&\cdot s_{p+1} 
		\cdot(s_{p+2}\cdots s_q)
		=
		(s_{p+1}\cdots s_q)
		\cdot
		(s_1\cdots s_{p-1})
		\\
		(s_1\cdots s_{p-1})
		&\cdot 1
		\cdot(s_{p+2}\cdots s_q)
		=
		(s_{p+2}\cdots s_q)
		\cdot
		(s_1\cdots s_{p-1})
	\end{align*}
	so it follows that
	\begin{align*}
		\psi^k_i(xs_p\otimes \alpha)
		=&
		\sigma(x)
		\cdot 
		(s_1\cdots s_q)
		\otimes\lambda^{i+1}\alpha
		\\
		+&
		\sigma(x)
		\cdot 
		(s_{p+1}\cdots s_q)\cdot(s_1\cdots s_p)
		\cdot
		\otimes
		\lambda^{i+1}\alpha
		\\
		-&
		\sigma(x)
		\cdot 
		(s_{p+2}\cdots s_q)\cdot (s_1\cdots s_p)
		\otimes
		\lambda^{i+2}\alpha
		\\
		-&
		\sigma(x)
		\cdot 
		(s_{p+1}\cdots s_q) \cdot (s_1\cdots s_{p-1})
		\otimes
		\lambda^{i+2}\alpha
		\\
		+&
		\sigma(x)
		\cdot 
		(s_{p+2}\cdots s_q) \cdot (s_1\cdots s_{p-1})
		\otimes
		\lambda^{i+3}\alpha.
	\end{align*}
	Now~$p<n-i-1+k$, which means that the final four terms above
	all lie in~$F^{k-1}$, so that in~$F^k/F^{k-1}$ we have
	\begin{align*}
		\psi^k_i(xs_p\otimes \alpha)
		&=
		\sigma(x)
		\cdot 
		(s_1\cdots s_q)
		\otimes\lambda^{i+1}\alpha
		\\
		&=
		\sigma(x)
		\cdot 
		(s_1\cdots s_{n-i-1+k})
		\otimes\lambda^{i+1}\alpha
		\\
		&=
		\psi^k_i(x\otimes\lambda\alpha)
	\end{align*}
	as required.	
\end{proof}

\begin{lemma}
	Let~$k\geqslant 1$ and let~$k\leqslant i\leqslant n-1$.
	Then for~$j$ in the range~$j\geqslant k+1$ we have
	$\psi^k_{i-1}\circ d^{i-k-1}_{j-k-1} = d^i_j\circ \psi^k_i$.
\end{lemma}

\begin{proof}
	Let~$x\otimes\alpha\in W(n-1)_{i-k-1} =\tl_{n-1}\ootimes{n-i-1+k}\t$.
	Then
	\begin{align*}
		d^i_j(\psi^k_i(x\otimes\alpha))
		&=
		d^i_j(\sigma(x)\cdot (s_1\cdots s_{n-i-1+k})
		\otimes\lambda^i\alpha)
		\\
		&=
		\sigma(x)
		\cdot (s_1\cdots s_{n-i-1+k})
		\cdot (s_{n-i+j-1}\cdots s_{n-i})
		\otimes\lambda^{i-j}\alpha.
	\end{align*}
	Since~$(n-i-1)\leqslant(n-i-1+k)+1\leqslant (n-i+j-1)$,
	we may apply the first part of Lemma~\ref{lemma-products}
	to obtain
	\begin{align*}
		d^i _j(\psi^k_i(x\otimes\alpha))
		&=
		\sigma(x)
		\cdot (s_1\cdots s_{n-i-1+k})
		\cdot (s_{n-i+j-1}\cdots s_{n-i})
		\otimes\lambda^{i-j}\alpha
		\\
		&=
		\sigma(x)
		\cdot (s_{n-i+j-1}\cdots s_{n-i+1})
		\cdot (s_1\cdots s_{n-i+k})
		\otimes\lambda^{i-j}\alpha
		\\
		&=
		\sigma(x)
		\cdot (s_{n-i+j-1}\cdots s_{n-i+1})
		\cdot (s_1\cdots s_{n-(i-1)-1+k})
		\otimes\lambda^{(i-1)}\lambda^{1-j}\alpha
		\\
		&=
		\sigma(x \cdot (s_{n-i+j-2}\cdots s_{n-i}))
		\cdot (s_1\cdots s_{n-(i-1)-1+k})
		\otimes\lambda^{(i-1)}\lambda^{1-j}\alpha
		\\
		&=
		\psi^k_{i-1}(x\cdot (s_{n-i+j-2}\cdots s_{n-i})
		\otimes\lambda^{1-j}\alpha).
	\end{align*}
	In the last line of the above computation, 
	$x\cdot(s_{n-i+j-2}\cdots s_{n-i})\otimes\lambda^{1-j}\alpha$
	is an element of~$W(n-1)_{(i-1)-k-1}=\tl_{n-1}\ootimes{n-i+k}\t$,
	so we have
	\begin{align*}
		x
		\cdot
		(s_{n-i+j-2}&\cdots s_{n-i})
		 \otimes
		\lambda^{1-j}\alpha
		\\
		&=
		x
		\cdot
		(s_{n-i+j-2}\cdots s_{n-i+k})
		\cdot
		(s_{n-i+k-1}\cdots s_{n-i})
		\otimes
		\lambda^{1-j}\alpha
		\\
		&=
		x
		\cdot
		(s_{n-i+j-2}\cdots s_{n-i+k})
		\otimes
		\lambda^k\lambda^{1-j}\alpha
		\\
		&=
		x
		\cdot
		(s_{n-i+j-2}\cdots s_{n-i+k})
		\otimes
		\lambda^{-(j-k-1)}\alpha.
	\end{align*}
	Thus
	\begin{align*}
		d^i_j(\psi^k_i(x\otimes\alpha))
		&=
		\psi^k_{i-1}(x\cdot (s_{n-i+j-2}\cdots s_{n-i})
		\otimes\lambda^{1-j}\alpha)
		\\
		&=
		\psi^k_{i-1}(
			x
			\cdot
			(s_{n-i+j-2}\cdots s_{n-i+k})
			\otimes
			\lambda^{-(j-k-1)}\alpha.
		)
		\\
		&=
		\psi^k_{i-1}(
			x
			\cdot
			(s_{(n-1)-(i-k-1)+(j-k-1)-1}\cdots s_{(n-1)-(i-k-1)})
			\otimes
			\lambda^{-(j-k-1)}\alpha
		)
		\\
		&=
		\psi^k_{i-1}(d^{i-k-1}_{j-k-1}(x\otimes\alpha))
	\end{align*}
	as required.
\end{proof}

\begin{corollary}
	$\Psi^k$ is a chain map.
\end{corollary}

\begin{proof}
	The boundary map of~$\tau_{n-1}\Sigma^{k+1}W(n-1)$
	is given in degree~$i$ 
	by the boundary map~$d^{i-k-1}\colon W(n-1)_{i-k-1}\to W(n-1)_{i-k-2}$,
	which is itself given by the formula~$\sum_{j=0}^{i-k-1}(-1)^jd^{i-k-1}_j$.

	The boundary map of~$F^k/F^{k-1}$ is given in degree~$i$ by the boundary map of~$W(n)$ in degree~$i$, which is the alternating sum
	$\sum_{j=0}^i(-1)^jd^i_j$.  However, the proof of
	Lemma~\ref{lemma-subcomplex} shows that~$d^i_0,\ldots,d^i_k$ all
	send~$F^k$ into~$F^{k-1}$, and hence that they vanish in the
	quotient~$F^k/F^{k-1}$.  Thus the boundary map of~$F^k/F^{k-1}$
	is~$\sum_{j=k+1}^i(-1)^jd^i_j$. It follows that

	\begin{align*}
		d^i\circ \Psi^k_i
		&=
		\sum_{j=k+1}^i (-1)^j d^i_j\circ[(-1)^{-i(k+1)}\psi^k_i]
		\\
		&=
		\sum_{j=k+1}^i (-1)^{j-i(k+1)} \psi^k_{i-1}\circ d^{i-k-1}_{j-k-1}
		\\
		&=
		\sum_{j=0}^{i-k-1} (-1)^{j+(k+1)-i(k+1)} \psi^k_{i-1}\circ d^{i-k-1}_j
		\\
		&=
		[(-1)^{-(i-1)(k+1)}\psi^k_{i-1}]
		\circ \sum_{j=0}^{i-k-1} (-1)^{j} \psi^k_{i-1}\circ d^{i-k-1}_j
		\\
		&=
		\Psi^k_{i-1} \circ d^{i-k-1}
	\end{align*}
	as required.
\end{proof}

\subsection{Proof of Theorem~\ref{theorem-filtration identification}}
\label{section-Phi}

In this subsection we prove Theorem~\ref{theorem-filtration identification}, which in turn completes the proof of Theorem~\ref{theorem-high-acyclicity}.

We begin by finding a basis for each part of the filtration in terms of the Jones normal form.  This is done in Lemma~\ref{lem-basis for F^k} below, after some preliminary work.

\begin{lemma}\label{lem-siandui}
	Any word in the $s_i$ not containing $s_1$ is a linear combination
	of words in the $U_i$, none of which involve $U_1$.
	Conversely, any word in the $U_i$ not containing $U_1$ is a linear
	combination of words in the $s_i$ not containing $s_1$.
\end{lemma}

\begin{proof}
	Recall from Definition~\ref{definition-si} 
	that~$s_i=\lambda+\mu U_i$, where $\lambda$ and $\mu$ are both
	units in the ground ring, so that
	$U_i=-\mu^{-1}\lambda + \mu^{-1}s_i$.
	The claim follows immediately.	
\end{proof}

\begin{lemma}\label{lem-string s1 to sp lin comb}
    For~$1\leqslant p \leqslant n-1$, \red{the word~$s_1\ldots s_p$ written in terms of the~$U_i$ generators} is equal to~$\mu^p U_1\cdots U_p$, plus a linear combination of scalar multiples---by units---of words~$w$ in the $U_i$ with the following properties:
    \begin{itemize}
        \item~$i(w)\geqslant 2$ and~$t(w)\leqslant p$ or 
        \item~$i(w)=1$ and~$t(w)< p$.
    \end{itemize}
    In particular only the summand~$w=\mu^p U_1\cdots U_p$ satisfies~$i(w)=1$ and~$t(w)=p$.
    \begin{proof}\red{Using $s_i=\lambda+\mu U_i$ and multiplying out brackets gives the following:}
        \begin{align*}
        s_1\ldots s_p
        &= 
        \sum\limits_{r=0}^p \sum\limits_{(1\leqslant i_1\leqslant\cdots \leqslant i_r \leqslant p)} \lambda^{p-r}\mu^r U_{i_1}U_{i_2}\ldots U_{i_r}
        \\
        &= 
        \mu^p U_1\cdots U_p+
        \sum\limits_{r=0}^{p-1} \sum\limits_{(1\leqslant i_1\leqslant\cdots \leqslant i_r \leqslant p)} \lambda^{p-r}\mu^r U_{i_1}U_{i_2}\ldots U_{i_r}
        \end{align*}
        If~$r=0$ the term is a scalar, which has index~$\infty$ by convention (thus the first point is satisfied). Suppose~$0<r<p$. Then if~$i_1>1$ it follows that~$i(U_{i_1}\ldots U_{i_p})\geqslant 2$. Otherwise~$i_1=1$ and, since~$r<p$, there is some~$j\geqslant 2$ such that~$i_j\geqslant i_{j-1}+2$, so that~$U_{i_1}\cdots U_{i_r}$ can be written as a word with terminus~${i_{j-1}}$, and then the claim follows. Coefficients are given by powers of~$\lambda$ and~$\mu$, and products of these. The terms~$\lambda$ and~$\mu$ are defined via the homomorphisms in Definition~\ref{defn-IH to TL} and lie in the set~$\{-1, \pm v, v^2\}$. Since~$v$ is a unit it follows that all coefficients are units.
    \end{proof}
\end{lemma}

\red{
\begin{lemma}\label{lem-basis lies inside F^k}
	Let~$k\geqslant 0$ and~$-1\leqslant i \leqslant n-1$, and consider elements~$x_{\ul{a},\ul{b}} \otimes 1$, where~$x_{\ul{a},\ul{b}}$ is in Jones normal form and satisfies either:
	\begin{itemize}
		\item~$i(x_{\ul{a},\ul{b}})\geqslant 2$ and~$t(x_{\ul{a},\ul{b}})\geqslant n-i-1$, or
		\item~$i(x_{\ul{a},\ul{b}})=1$ and~$n-i-1\leqslant t(x_{\ul{a},\ul{b}})\leqslant n-i-1+k$.
	\end{itemize}
	Then these elements all lie in $F^k_i$.
	\begin{proof}
		The first type of element lies in~$F^0$, since in this case~$x_{\ul{a},\ul{b}}$ is a word in the~$U_i$s not containing~$U_1$, and by Lemma~\ref{lem-siandui}, this is a linear combination of words in the~$s_i$s not containing~$s_1$ (a basic element). For the second type of element, from the definition of Jones normal form, $x_{\ul{a},\ul{b}}$ must end in a string~$U_1\ldots U_{n-i-1+j}$ for~$0\leq j \leq k$. We proceed by induction on~$k$.\\
		\textbf{Base case:} We start with the base case $k=0$, so the only option is that~$j=0$. i.e.~$x_{\ul{a},\ul{b}}=y_{\ul{a},\ul{b}}U_1\ldots U_{n-i-1}$ for some~$y_{\ul{a},\ul{b}}$ in Jones normal form with $i(y_{\ul{a},\ul{b}})\geq 2$. We aim to show in this case $x_{\ul{a},\ul{b}}$ lies in~$F^0_i$. Compare $x_{\ul{a},\ul{b}}$ to $y_{\ul{a},\ul{b}}s_1\ldots s_{n-i-1}$, which does lie in~$F^0_i$ by Definition~\ref{definition-filtration}. From Lemma~\ref{lem-string s1 to sp lin comb} multiplying out the string $s_1\ldots s_{n-i-1}$ will result in $y_{\ul{a},\ul{b}}s_1\ldots s_{n-i-1}$ being written as a linear combination (up to scalar multiplication by units) of three types of elements, and we consider their image in $\tl_n\otimes_{\tl_{n-i-1}}\t$.
		\begin{enumerate}
		\item $y_{\ul{a},\ul{b}}U_1\ldots U_{n-i-1}$. This is equal to $x_{\ul{a},\ul{b}}$ and appears as a single summand of $y_{\ul{a},\ul{b}}s_1\ldots s_{n-i-1}$.
		\item $y_{\ul{a},\ul{b}}w$ where $i(w)\geqslant 2$ and~$t(w)\leqslant n-i-1$. These are all basic elements since $i(y_{\ul{a},\ul{b}})\geq 2$, and thus lie in $F^0$.
		\item $y_{\ul{a},\ul{b}}w$ where $i(w)=1$ and~$t(w)< n-i-1$. Due to the terminus, these are all zero in $\tl_n\otimes_{\tl_{n-i-1}}\t$.
		\end{enumerate}
		So it follows that in $\tl_n\otimes_{\tl_{n-i-1}}\t$,~$x_{\ul{a},\ul{b}}$ is, up to scalar multiplication by units, equal to a linear combination of $y_{\ul{a},\ul{b}}s_1\ldots s_{n-i-1}$ and basic elements. Since this is a linear combination of elements in~$F^0_i$, it follows that $x_{\ul{a},\ul{b}}$ lies in~$F^0_i$ as required.\\
		\textbf{Inductive step:} We now assume the Lemma is true for $k-1$ and prove for~$k$. Let~$x_{\ul{a},\ul{b}}=y_{\ul{a},\ul{b}}U_1\ldots U_{n-i-1+j}$ for some~$y_{\ul{a},\ul{b}}$ in Jones normal form with $i(y_{\ul{a},\ul{b}})\geq 2$, $t(y_{\ul{a},\ul{b}})>n-i-1+j$, and~$0\leq j\leq k$. When~$0\leq j <k$, by the inductive hypothesis this element lies in~$F^{k-1}_i\subset F^k_i$, and so we can restrict to the case where~$j=k$, i.e.~$x_{\ul{a},\ul{b}}=y_{\ul{a},\ul{b}}U_1\ldots U_{n-i-1+k}$. We aim to show that, in this case, $x_{\ul{a},\ul{b}}$ lies in~$F^k_i$. As in the base case, we compare $x_{\ul{a},\ul{b}}$ with $y_{\ul{a},\ul{b}}s_1\ldots s_{n-i-1+k}$, which lies in~$F^k_i$ by Definition~\ref{definition-filtration}. From Lemma~\ref{lem-string s1 to sp lin comb}, $y_{\ul{a},\ul{b}}s_1\ldots s_{n-i-1+k}$ is a linear combination (up to scalar multiplication by a unit) of three types of elements, which we evaluate in~$\tl_n\otimes_{\tl_{n-i-1}}\t$.
		\begin{enumerate}
			\item $y_{\ul{a},\ul{b}}U_1\ldots U_{n-i-1+k}$. This is equal to $x_{\ul{a},\ul{b}}$ and appears as a single summand of $y_{\ul{a},\ul{b}}s_1\ldots s_{n-i-1+k}$. 
			\item $y_{\ul{a},\ul{b}}w$ where $i(w)\geqslant 2$ and~$t(w)\leqslant n-i-1+k$. Since $i(y_{\ul{a},\ul{b}})\geq 2$ they are all basic elements, and thus lie in $F^0\subset F^k$.
			\item $y_{\ul{a},\ul{b}}w$ where $i(w)=1$ and~$t(w)< n-i-1+k$. Rewriting these in Jones normal form gives elements~$y_{\ul{a},\ul{b}}w=z_{\ul{a},\ul{b}}$ such that~$i(z_{\ul{a},\ul{b}})=1$ and $t(z_{\ul{a},\ul{b}})\leq t(w)<n-i-1+k$ (by Lemmas \ref{lem-index once} and \ref{lem-terminus of JNF}). These are then Jones normal form elements ending in~$U_1 \ldots U_{n-i-1+j}$ for $0\leq j\leq k-1$ so by the inductive hypothesis these lie in~$F^{k-1}_i\subset F^k_i$.
		\end{enumerate}
		Again it follows that in $\tl_n\otimes_{\tl_{n-i-1}}\t$,~$x_{\ul{a},\ul{b}}$ is, up to scalar multiplication by units, equal to a linear combination of $y_{\ul{a},\ul{b}}s_1\ldots s_{n-i-k}$, elements in~$F^k_i$ (by the inductive hypothesis), and basic elements. Since this is a linear combination of elements in~$F^k_i$, it follows that $x_{\ul{a},\ul{b}}$ lies in~$F^k_i$ as required.
	\end{proof}
\end{lemma}}

\begin{lemma}\label{lem-basis for F^k}
    Let~$k\geqslant 0$ and~$-1\leqslant i \leqslant n-1$. Then~$F^k_i$ has basis consisting of elements~$x_{\ul{a},\ul{b}} \otimes 1$, where~$x_{\ul{a},\ul{b}}$ is in Jones normal form and satisfies either:
    \begin{itemize}
        \item~$i(x_{\ul{a},\ul{b}})\geqslant 2$ and~$t(x_{\ul{a},\ul{b}})\geqslant n-i-1$, or
        \item~$i(x_{\ul{a},\ul{b}})=1$ and~$n-i-1\leqslant t(x_{\ul{a},\ul{b}})\leqslant n-i-1+k$.
    \end{itemize}
    \begin{proof}
        This is a subset of the known basis for~$\tl_n\otimes_{\tl_{n-i-1}}\t\supseteq F^k_i$, and by the previous lemma we know these elements lie in~$F^k_i$, so it is enough to show that~$F^k_i$ is spanned by these elements. \red{First of all, note that since~$F^k_i\subseteq \tl_n\otimes_{\tl_{n-i-1}} \t$ any word in $F^k_i$ written in Jones normal form will vanish if~$t(x_{\ul{a},\ul{b}})\leqslant n-i-2$, therefore we will always have~$t(x_{\ul{a},\ul{b}})\geqslant n-i-1$.} By definition~$F^k_i$ is spanned by elements of the form
        \begin{itemize}
            \item~$x\otimes 1$
            \item~$x\cdot (s_1\cdots s_{n-i-1+k'})\otimes 1$
        \end{itemize}
        where~$x$ is a word in the~$U_i$ with~$i(x)\geqslant 2$ (i.e.~containing no~$U_1$s) and $0\leqslant k' \leqslant k$ (note in the case~$i=n-1$ and~$k'=0$ the two kinds coincide). The first kind is spanned by~$x_{\ul{a}, \ul{b}}$ such that~$i(x_{\ul{a}, \ul{b}})\geqslant 2$, \red{as described in the first bullet point in the statement of the lemma.} From Lemma~\ref{lem-string s1 to sp lin comb}, expanding the product~$(s_1\cdots s_{n-i-1+k'})$ in the second kind gives a linear combination of words~$x\cdot w \otimes 1$ such that~$t(w)\leqslant n-i-1+k'$. Either~$i(w)$ will be~$\geqslant 2$ or~$i(w)=1$. In the first case, since~$i(x)\geqslant 2$ it follows that~$i(x\cdot w)\geqslant 2$ and so when written in Jones normal form this will remain the case, giving \red{an element of the first type described in the Lemma}. In the second case, when~$i(w)=1$, since~$i(x)\geqslant 2$ then either~$i(x\cdot w)\geqslant 2$ and \red{as in the previous sentence} we are done, or~$i(x\cdot w)=1$ and, by Lemma~\ref{lem-terminus of JNF}, when written in Jones normal form the terminus~$t(x\cdot w)=t(w)\leqslant n-i-1+k'\leqslant n-i-1+k$ will either remain the same or reduce. \red{This puts us in the setting of the second bullet point in the statement of the lemma, and thus we have shown that the two types of elements span~$F^k_i$.}
    \end{proof}
\end{lemma}

\begin{prop}
    The map~$\Phi^0\colon C(W(n-1))\longrightarrow F^0$ from Definition~\ref{defn-phi for k=0} is an isomorphism.
    
    \begin{proof}
        Recall that for~$-1\leqslant i \leqslant (n-1)$,
        \[
        \Phi^0_i:\left(\tl_{n-1}\ootimes{n-i-2}\t\right)
		\oplus
		\left(\tl_{n-1}\ootimes{n-i-1}\t\right) \to F^0_i
        \]
        is given by
        \[
        	\Phi^0_i(x\otimes\alpha,y\otimes\beta) 
		= 
		\xi_i(x\otimes\alpha) + \eta_i(y\otimes\beta)
	\]
	where
	\[
		\xi_i(x\otimes\alpha) = \sigma(x)\otimes\lambda^{n-1}\alpha	
	\]
	and
	\[
		\eta_i(y\otimes\beta) = \sigma(y)\cdot(s_1\cdots s_{n-i-1})
		\otimes\lambda^i\beta.
	\]
        By Lemma~\ref{lem-basis for tensor product over smaller tl}, a basis for the left hand side is given by elements of either the form~$(x_{\ul{a},\ul{b}}\otimes 1, 0)$ such that~$t(x_{\ul{a},\ul{b}})>n-i-3$ or the form 
       ~$(0, x_{\ul{a}',\ul{b}'}\otimes 1)$ such that~$t(x_{\ul{a}',\ul{b}'})>n-i-2$. Under the map~$\Phi^0_i$,~$(x_{\ul{a},\ul{b}}\otimes 1,0)$ is taken to a scalar multiple (by a unit) of~$\sigma(x_{\ul{a},\ul{b}})\otimes 1$, where~$\sigma(x_{\ul{a},\ul{b}})$ is a Jones basis element with~$i(\sigma(x_{\ul{a},\ul{b}}))\geqslant 2$ and~$t(\sigma(x_{\ul{a},\ul{b}}))>n-i-2$. 
        By Lemma~\ref{lem-string s1 to sp lin comb}, the element~$(0,x_{\ul{a}',\ul{b}'}\otimes 1)$ is taken to a linear combination of scalar multiples (by units) of terms~$\sigma(x_{\ul{a}',\ul{b}'})\cdot w\otimes 1$ such that~$t(w)\leqslant n-i-1$. Since~$F^0_i\subseteq \tl_n\otimes_{\tl_{n-i-1}}\t$ the only non-zero terms in the image will occur when~$t(w)= n-i-1$. We consider two cases:~$i(w)\geqslant 2$ or~$i(w)=1$. By Lemma~\ref{lem-terminus of JNF}, converting to Jones normal form in the first case gives an element with index~$i(\sigma(x_{\ul{a}',\ul{b}'})\cdot w)>2$ and terminus~$t(\sigma(x_{\ul{a}',\ul{b}'})\cdot w)=n-i-1$, or zero, since the terminus will either remain the same or reduce when converting. When~$i(w)=1$ and~$t(w)= n-i-1$, by Lemma~\ref{lem-string s1 to sp lin comb} \red{it follows that $w=U_1\ldots U_{n-i-1}$} and therefore the terms will be of the form~$\sigma(x_{\ul{a}',\ul{b}'})\cdot U_1\ldots U_{n-i-1}$. These elements are already in Jones normal form, with index~$1$ and terminus~$n-i-1$. Furthermore all Jones basis elements with this index and terminus arise in this way.
        By Lemma~\ref{lem-basis for F^k} a basis for~$F^0_i$ is given by elements~$y_{\ul{a},\ul{b}}\otimes 1$ where~$y_{\ul{a},\ul{b}}$ is in Jones normal form and satisfies:
    \begin{itemize}
        \item~$i(y_{\ul{a},\ul{b}})\geqslant 2$ and~$t(y_{\ul{a},\ul{b}})\geqslant n-i-1$ or
        \item~$i(y_{\ul{a},\ul{b}})=1$ and~$t(y_{\ul{a},\ul{b}})= n-i-1$.
    \end{itemize}
        By our analysis, all of these elements lie in the image of~$\Phi^0_i$, up to scalar multiplication by units, hence~$\Phi^0$ is a bijection on bases and therefore an isomorphism.
        \end{proof}
	
\end{prop}

\begin{lemma}\label{lem-basis for quotient F^k/F^k-1}
    A basis for~$(F^k/F^{k-1})_i$ is given by words~$x_{\ul{a},\ul{b}}$ in Jones normal form such that ~$i(x_{\ul{a},\ul{b}})=1$ and~$t(x_{\ul{a},\ul{b}})= n-i-1+k$.
    \begin{proof}
        This is a direct consequence of taking the quotient of the bases for~$F^k$ and~$F^{k-1}$ given in Lemma~\ref{lem-basis for F^k}.
    \end{proof}
\end{lemma}

\begin{prop}
    The map~$\Psi^k: \tau_{n-1}\Sigma^{k+1}W(n-1)
		\longrightarrow 
		F^k/F^{k-1}$ from Definition~\ref{defn-phi for k>0} is an isomorphism.
    \begin{proof}
       Recall for~$i$ in the range~$k\leqslant i\leqslant (n-1)$,
	\[
		\Psi^k_i\colon \tl_{n-1}\ootimes{n-i-1+k}\t
		\longrightarrow
		(F^k/F^{k-1})_i
	\]
	is given by	
	\[
		\Psi^k_i\colon x\otimes\alpha
		\longmapsto
		(-1)^{-i(k+1)}\sigma(x)
		\cdot (s_1\cdots s_{n-i-1+k})\otimes\lambda^i\alpha.
	\]
    	By Lemma~\ref{lem-basis for tensor product over smaller tl} a basis for the domain is given by~$x_{\ul{a},\ul{b}}$ such that $t(x_{\ul{a},\ul{b}})>(n-i-1+k)-1$. Note also that~$x_{\ul{a},\ul{b}}$ does not contain the letter~$U_{n-1}$. By Lemma~\ref{lem-string s1 to sp lin comb}, the image~$\Psi^k_i(x_{\ul{a},\ul{b}})$ is a linear combination of scalar multiples (by units) of terms~$\sigma(x_{\ul{a},\ul{b}})\cdot w$ such that~$t(w)\leqslant n-i-1+k$. These terms are zero in~$(F^k/F^{k-1})_i\subseteq \tl_n\otimes_{\tl_{n-i-1}}\t$ only when~$w$ cannot be written as a word with~$t(w)<n-i-1$. Rewriting these elements in Jones normal form will maintain or decrease the terminus, and~$i(\sigma(x_{\ul{a},\ul{b}}))\geqslant 2$, so~$i(\sigma(x_{\ul{a},\ul{b}})\cdot w)=1$ only when~$i(w)=1$. Therefore by Lemma~\ref{lem-basis for F^k} quotienting out by~$F^{k-1}$ leaves only the term for which~$i(w)=1$ and~$t(w)=n-i-1+k$. In particular by Lemma~\ref{lem-string s1 to sp lin comb} this term is a scalar multiple (by a unit) of~$\sigma(x_{\ul{a},\ul{b}})\cdot U_1\ldots U_{n-i-1+k}$. 
	
    	Since~$\sigma(x_{\ul{a},\ul{b}})$ has index~$\geqslant 2$ and terminus~$>n-i-1+k$, it follows that~$\sigma(x_{\ul{a},\ul{b}})\cdot U_1\ldots U_{n-i-1+k}$ is in Jones normal form. From Lemma~\ref{lem-basis for quotient F^k/F^k-1} this is a Jones basis element for~$F^k/F^{k-1}$ and all basis elements arise in this way. Therefore up to unit scalars, the map~$\Psi^k$ is a bijection on bases, and hence an isomorphism.
	\end{proof}
\end{prop}

\section{Jones-Wenzl projectors and vanishing}\label{section-JW}

This section relates our results with the existence of the \emph{Jones-Wenzl projectors}, to strengthen our vanishing results when~$R$ is a field. This section is written such that the reader can read the introduction, the background on Temperley-Lieb algebras, and continue straight to this section. For the time being we make the substitutions~$a\leftrightarrow\delta$ and~$v\leftrightarrow q$, as is common in the recent literature concerning Jones-Wenzl projectors.

Throughout this section, we will consider a commutative ring $R$, a unit $q\in R$, the parameter $\delta=q+q^{-1}$, and we will work in $\tl_n(\delta)$.  Recall that we show in Theorem~\ref{theorem-invertible} that, when~$\delta$ is invertible,~$\Tor_\ast^{\tl_n(\delta)}(\t,\t)$ and $\Ext^\ast_{\tl_n(\delta)}(\t,\t)$ vanish in every non-zero degree. 
In this section we investigate the case where~$\delta=0$ and $R$ is a field using established results on Jones-Wenzl projectors. We prove the following theorem:

\setcounter{abcthm}{3}
\begin{abcthm}
    Let~$n=2k+1$, and let~$R$ be a field whose characteristic does not divide $\binom{k}{t}$ for any~$0\leq t\leq k$. 
    Let~$q$ be a unit in~$R$ and assume that~$\delta=q+q^{-1}=0$. 
    Then $\Tor_\ast^{\tl_n(0)}(\t,\t)$ and $\Ext_{\tl_n(0)}^\ast(\t,\t)$ vanish in positive degrees.
\end{abcthm}

For example, when $n=3$, $R$ is a field, and $\delta=q+q^{-1}$ for $q\in R^\times$, \red{combining this theorem with Theorem \ref{theorem-invertible}} demonstrates that $\Tor_\ast^{\tl_3(\delta)}(\t,\t)$ and $\Ext_{\tl_3(\delta)}^\ast(\t,\t)$ vanish in positive degrees with no further condition on $\delta$.
And if one wishes to show that 
$\Tor_\ast^{\tl_5(\delta)}(\t,\t)$ and $\Ext_{\tl_5(\delta)}^\ast(\t,\t)$ can be nonzero in positive degrees, then the only chance of this happening is in characteristic $2$.

The theorem is in strict contrast to the~$n$ even case, where we show in Theorem~\ref{theorem-invertible} that for a general ring~$R$ and $\delta$ not invertible,~$\Tor^{\tl_n(\delta)}_{n-1}(\t,\t)=R/bR$ is non-zero for~$b$ some multiple of~$\delta$.
Therefore, in the particular case where $n$ is even, $R$ is a field, and $\delta=0$, there can be no vanishing in all positive degrees.

\subsection{Jones-Wenzl projectors}

In this subsection we introduce the Jones-Wenzl projector and relate its existence to the projectivity of the trivial module $\t$.
The original references are~\cite{JonesIndex} and~\cite{WenzlProjections}, see also \cite{KauffmanLins}, \cite[section 4]{LickorishCalculations}.

\begin{defn}\label{def-JWn}
Recall that~$I_n\subseteq \tl_n$ is the two sided ideal generated by the~$U_i$ for $i=1,\ldots,n-1$. 
Then, if it exists, the\emph{~$n$th Jones-Wenzl projector~$\JW_n$} is the element of $\tl_n$ characterised by the following two properties:
\begin{enumerate}[(i)]
    \item $\JW_n \in 1+I_n$ and
    \item $I_n\cdot \JW_n=0=\JW_n\cdot I_n$.
\end{enumerate}
\end{defn}

\begin{lem}\label{lemma-jw-unique}
If~$\JW_n$ exists it is unique.
\begin{proof}
Suppose a second element~$\JW_n'$ in~$\tl_n$ satisfies (i) and (ii) of Definition~\ref{def-JWn}. Write~$\JW_n=1+i$ and $\JW_n'=1+i'$ for~$i, i' \in I_n$. Then~$\JW_n\cdot i'=0=i\cdot \JW_n'$ by (ii). It follows that 
\[
    \JW_n'=\JW_n' +i\cdot\JW_n'
    =(1+i)\cdot\JW_n'
    =\JW_n\cdot\JW_n'
\]
and similarly that $\JW_n=\JW_n\cdot \JW_n'$. 
\end{proof}
\end{lem}

The Jones-Wenzl projector was first introduced by Jones~\cite{JonesIndex}, was further studied by Wenzl~\cite{WenzlProjections}, and has since become important in representation theory, knot theory and the study of $3$-manifolds. 
It is a key ingredient in the definition of the coloured Jones polynomial and $\mathrm{SU}(2)$ quantum invariants more generally, and is important in the study of tilting modules of (quantum) $\mathfrak{sl}_2$.

\subsection{$\JW_n$ and projectivity of $\t$}

We will now show that the Jones-Wenzl projector exists if and only if the trivial module $\t$ is projective.
Thus, existence of $\JW_n$ implies the vanishing of $\Tor_\ast^{\tl_n(\delta)}(\t,\t)$ and $\Ext^\ast_{\tl_n(\delta)}(\t,\t)$ in positive degrees.
Our own Theorem~\ref{theorem-invertible} implies that vanishing for $\delta$ invertible, while Theorem~\ref{theorem-sharpness} proves non-vanishing for $n$ even and $\delta$ not invertible. 
It turns out that there is a rich interplay between these two sources of (non-)vanishing results.

\begin{prop}\label{prop-JW and t projective}
$\JW_n$ exists if and only if~$\t$ is a projective left~$\tl_n(\delta)$-module, which is if and only if $\t$ is a projective right~$\tl_n(\delta)$-module.
\end{prop}

Before proving the proposition we need the following.

\begin{lem}\label{lem-alternate JW (ii)}
In Definition \ref{def-JWn}, it is sufficient to replace (ii) with either
\begin{itemize}
    \item[(ii)'\phantom{'}] $I_n\cdot\JW_n=0$, or
    \item[(ii)''] $\JW_n\cdot I_n=0$.
\end{itemize}
\end{lem}

\begin{proof}
Suppose~$\JW \in \tl_n$ satisfies (i) and (ii)'. We have suggestively named this element, and will show it is in fact~$\JW_n$, by showing that~$\JW$ also satisfies (ii)'' and hence (ii). Let~$\tl_n\to \tl_n, d\mapsto \overline{d}$ be the anti-automorphism which reverses the order of letters in a monomial i.e.~$\overline{U_{i_1}\ldots U_{i_n}}=U_{i_n}\ldots U_{i_1}$. In diagrammatic terms, this map flips the diagram corresponding to the monomial in the left-to-right direction. 
Since~$\JW$ satisfies (ii)', it follows that $\overline{\JW}$ satisfies (ii)''. 
Then the argument of Lemma~\ref{lemma-jw-unique} can be repeated to show that $\JW=\overline\JW$ so that~$\JW$ satisfies (ii)' and (ii)'', hence it satisfies (ii) and~$\JW=\JW_n$.
\end{proof}

\red{
\begin{proof}[Proof of Proposition~\ref{prop-JW and t projective}]
We prove the equivalence for left-modules.

If~$\JW_n$ exists, then the maps $\t\to\tl_n$, $1\mapsto\JW_n$ and $\tl_n\to\t$, $d\mapsto d\cdot 1$ 
are maps of left~$\tl_n$-modules composing to the identity. It follows that~$\t$ is a direct summand of~$\tl_n$, and thus is projective.

Conversely, if $\t$ is a projective left $\tl_n$-module, 
then the surjection $\tl_n\to\tl_n/I_n = \t$, 
regarded as a map of left $\tl_n$-modules,
has a splitting $s\colon\t\to\tl_n$, again a map of left $\tl_n$-modules.
By construction the element $s(1)$ then satisfies
condition (i) of~\ref{def-JWn} and condition (ii)' of 
\ref{lem-alternate JW (ii)}, so that $\JW_n=s(1)$ exists as required.
\end{proof}
}

\subsection{Jones-Wenzl projectors and quantum binomial coefficients}

Here we work in the Laurent polynomial ring $\Z[q,q^{-1}]$, and we set $\delta=q+q^{-1}$. \red{For this section, let~$n$ and $r$ be integers such that~$n \geqslant r \geqslant 0$.} 

\begin{defn}\label{defn-qbc and qI}
The \emph{quantum integer}~$\qi{n}$ is defined to be
$$
\qi{n}=\frac{q^n-q^{-n}}{q-q^{-1}}=q^{n-1}+q^{n-3}+\ldots + q^{-(n-3)} +q^{-(n-1)},
$$
the \emph{quantum factorial} $\qi{n}!$ is defined by 
\[
    \qi{n}!=\qi{n}\qi{n-1}\cdots\qi{1},
\]
and the \emph{quantum binomial coefficient}~$\qbc{n}{r}$ is then given by computing the normal binomial coefficient but replacing integers with quantum integers:
$$
\qbc{n}{r}=\frac{\qi{n}!}{\qi{r}!\qi{n-r}!}.
$$
\end{defn}

The quantum binomial coefficients satisfy the following recursion relations:
\begin{align*}
    \qbc{n}{r}&=q^{n-r\phantom{(-)}}\qbc{n-1}{r-1}+q^{-r}\qbc{n-1}{r}
    \\
    \qbc{n}{r}&=q^{-(n-r)}\qbc{n-1}{r-1}+q^{r\phantom{-}}\qbc{n-1}{r}.
\end{align*}
Either one of these relations gives an inductive proof that $\qbc{n}{r}$ lies in $\Z[q,q^{-1}]$. 
\red{
And taken together, these relations give an inductive proof that 
$\qbc{n}{r}$ is invariant under inverting $q$, 
and consequently that it lies in $\Z[\delta]$.
(Recall that $\delta=q+q^{-1}$.)
}
This means that we may evaluate $\qbc{n}{r}$ in any ring containing an element named $\delta$, to obtain an element of that ring, which we continue to denote by $\qbc{n}{r}$.

\red{The following result is proved in an Appendix to \cite{EliasLibendinsky} by Webster, using Schur-Weyl duality. For a purely diagrammatic approach see recent work of Spencer \cite[Section 10.3]{Spencer}.}

\begin{thm}[{\cite[Theorem A.2]{EliasLibendinsky}, \cite[Section 10.3]{Spencer}}] \label{prop:JW and qbc connection}
Let $R=\k$ be a field, let $q\in\k$ be nonzero, and set $\delta=q+q^{-1}$.
The~$n$th Jones-Wenzl projector~$\JW_n\in\tl_n(\delta)$ exists if and only if the {quantum binomial coefficients} 
$\qbc{n}{r}$ are non-zero in $\k$ for all~$0\leq r \leq n$.
\end{thm}

\begin{remark}
    Suppose that $R=\k$ is a field.
    Whenever $\k$, $q$ and $n$ satisfy the conditions of Theorem~\ref{prop:JW and qbc connection}, we obtain the vanishing of $\Tor_\ast^{\tl_n(\delta)}(\t,\t)$ and $\Ext_{\tl_n(\delta)}^\ast(\t,\t)$ in positive degrees.
    (We will refer to this as simply ``vanishing'' for the present remark.)
    \begin{itemize}
        \item
        In the case of $n$ even, Theorems~\ref{theorem-invertible} and~\ref{theorem-sharpness} show that vanishing holds if and only if $\delta\neq 0$, and this is in fact stronger than the result obtained from Theorem~\ref{prop:JW and qbc connection}.
        For example, if we take $n=4$ then the $\qbc{n}{k}$ take values $1$, $\delta(\delta^2-2)$ and $(\delta^2-1)(\delta^2-2)$, so that Theorem~\ref{prop:JW and qbc connection} requires $\delta$ to avoid the values $0,\pm 1,\pm\sqrt{2}$.
        For $n$ even, $\delta$ is always a factor of $\qbc{n}{1}$, so that Theorem~\ref{theorem-invertible} will always apply more generally than Theorem~\ref{prop:JW and qbc connection} in this case.
        \item
        In the case of $n$ odd, the situation is more interesting.
        Theorem~\ref{theorem-invertible} demonstrates vanishing when $\delta\neq 0$.
        But if we take $n=3$, for example, then the $\qbc{n}{k}$ take values $1$ and $(\delta^2-1)$, so that Theorem~\ref{prop:JW and qbc connection} demonstrates vanishing so long as $\delta\neq\pm 1$.
        Neither of these vanishing results implies the other, but taken together they demonstrate vanishing for \emph{all} values of $\delta$.
    \end{itemize}
\end{remark}

\subsection{Identifying the quantum binomial coefficients}
In this section, we identify the quantum binomial coefficients upon specialising~$\delta=q+q^{-1}=0$. The results are assembled in the following proposition.
\begin{prop}\label{prop:quantum binomial triangle}
When~$\delta=q+q^{-1}=0$ then the quantum binomial coefficients have the following form:
\begin{itemize}
    \item When~$n$ is even and~$r$ is odd,~$\qbc{n}{r}=0$.
    \item When~$n$ and~$r$ are both even, let~$n=2a$ and~$r=2t$. Then $\qbc{n}{r}=\bc{a}{t}$.
    \item When~$n$ is odd and~$r$ is even, let~$n=2a+1$ and~$r=2t$. Then $\qbc{n}{r}=(-1)^t\bc{a}{t}$.
    \item When~$n$ and~$r$ are both odd, let~$n=2a+1$ and~$r=2t+1$. Then~$\qbc{n}{r}=(-1)^{a-t}\bc{a}{t}$.
\end{itemize}
\end{prop}

\begin{rem}
Proposition~\ref{prop:quantum binomial triangle} shows that the `quantum Pascal's triangle' with $\delta=0$ looks like a Pascal's triangle in the even rows, with every coefficient separated by a zero, and a `doubled' Pascal's triangle with signs on the odd rows. This is shown in Figure~\ref{fig:quantumtriangle}.
\end{rem}
\begin{figure}
    \centering
    \begin{tikzpicture}[scale=0.4]
    \foreach \x in {0,1,2,3,4,5,6,7,8,9}
    \draw  (\x,9-\x) node {$1$};
    \foreach \x in {0,1,2,3,4,5,6,7,8,9}
    \draw (-\x,9-\x) node {$1$};
    \foreach \x in {0,2,4,6}
    \draw  (\x,7-\x) node {$0$};
    \foreach \x in {0,2,4,6}
    \draw (-\x,7-\x) node {$0$};
    \foreach \x\y in {2/1, -2/1, 0/3}
    \draw (\x,\y) node {$0$};
    \foreach \x\y in {2/3, -2/3, 1/2, -1/2}
    \draw (\x,\y) node {$3$};
    \foreach \x\y in {-3/2, 3/2}
    \draw (\x,\y) node {$-3$};
    \foreach \x\y in {-4/1, 4/1}
    \draw (\x,\y) node {$4$};
    \foreach \x\y in {-3/0, 3/0,-5/0,5/0}
    \draw (\x,\y) node {$-4$};
    \foreach \x\y in {-1/6, 1/6, -5/2,5/2}
    \draw (\x,\y) node {$-1$};
    \foreach \x\y in {-1/4, 1/4}
    \draw (\x,\y) node {$-2$};
    \foreach \x\y in {-7/0, 7/0, -3/4, 3/4}
    \draw (\x,\y) node {$1$};
    \foreach \x\y in {1/0, -1/0, 0/1}
    \draw (\x,\y) node {$6$};
    \draw (0,5) node {$2$};
    \foreach \x in {0,1,2,3,4,5,6,7,8,9}
    \draw (-12,9-\x) node {$n=\x$};
    \end{tikzpicture}
    \caption{The quantum binomial coefficients with~$\delta=q+q^{-1}=0$.}
    \label{fig:quantumtriangle}
\end{figure}
The proof of the four points in this proposition are given by applying a result of D\'esarm\'enien \cite[Proposition 2.2]{Desarmenien}, which we recall below.
This result is given not in terms of \emph{quantum} binomials, but in terms of \emph{Gaussian} binomials, so we recall these first.

Let $p$ be an indeterminate.
The \emph{Gaussian binomial coefficients} are the quantities
\[
    \gbc{n}{r}{p} = \frac{\gi{n}{p}!}{\gi{r}{p}!\gi{n-r}{p}!}
\]
defined in terms of the \emph{Gaussian integers} $\gi{n}{p}=1+p+\cdots +p^{n-1}$ and \emph{Gaussian factorials} $\gi{n}{p}!=\gi{n}{p}\gi{n-1}{p}\cdots\gi{1}{p}$.
The relation between the Gaussian and quantum binomial coefficients is
\[
\qbc{n}{r}=q^{r^2-nr}\gbc{n}{r}{q^2}.
\]

\begin{prop}[{\cite[Proposition 2.2]{Desarmenien}}]\label{prop:desarmenien}
Fix a~$k\geq 0 \in \mathbb{N}$ and let~$\Phi_k$ be the~$k$th cyclotomic polynomial. Let~$n=ka+b$ and~$r=kt+s$ with~$0\leq b,s \leq k-1$. Then the Gaussian binomial coefficient satisfies the congruence
\[
\gbc{n}{r}{p}\equiv\bc{a}{t}\gbc{b}{s}{p} \mod \Phi_k.
\]
\end{prop}

\begin{proof}[Proof of Proposition \ref{prop:quantum binomial triangle}]
Note that when~$\delta=q+q^{-1}=0$ that rearranging this equation gives that~$q^{\pm2}=-1$. Recall that the parameter~$p$ in the Gaussian binomial coefficient is~$q^2$, and so~$p^2=q^4=1$.
We invoke Proposition \ref{prop:desarmenien} with~$k=2$. Then the cyclotomic polynomial~$\Phi_2(p)=1+p=1+q^2=0$.
Let~$n=2a+b$ and~$r=2t+s$ with~$0\leq b,s \leq 1$, then the quantum binomial coefficient satisfies
\[
\qbc{n}{r}=q^{r^2-nr}\gbc{n}{r}{q^2}=q^{r^2-nr}\bc{a}{t}\gbc{b}{s}{q^2}.
\]
When~$n$ is even and~$r$ is odd, $\gbc{b}{s}{q^2}=\gbc{0}{1}{q^2}=0$ which gives the first case of the proposition. For all other cases,~$\gbc{b}{s}{q^2}=1$ and so 
\[
\qbc{n}{r}=q^{r^2-nr}\bc{a}{t}.
\]
Computing the coefficient $q^{r^2-nr}$, using~$q^{\pm2}=-1$, yields the result for the remaining three cases.
\end{proof}

\subsection{Proof of theorem}
In this section we prove Theorem~\ref{thm-JW}.

\begin{proof}
This proof puts together three previous results. By Proposition \ref{prop-JW and t projective} we know that if $\JW_n$ exists then~$\t$ is projective and it follows that~$\Tor^{\tl(0)}_i(\t,\t)$ and $\Ext_{\tl_n(0)}^\ast(\t,\t)$ vanish for all~$i>0$. So it is enough to show that~$\JW_n$ exists under the hypotheses of the theorem. Proposition~\ref{prop:JW and qbc connection} tells us that~$\JW_n$ exists precisely when  the quantum binomial coefficients~$\qbc{n}{r}$ are non-zero for all~$1\leq r\leq n$. Finally, Proposition~\ref{prop:quantum binomial triangle} explicitly describes these coefficients when~$\delta=0$. We see that for~$n$ even, there is always a quantum binomial coefficient~$\qbc{n}{r}=0$ and so we learn nothing new. However when~$n=2k+1$ is odd, the quantum binomial coefficients take values in the set
\[\red{\left\{\pm \bc{k}{t}\, \middle| \,0\leq t\leq k\right\}}\]
and---up to sign---all values in this set are realised as some~$\qbc{n}{r}$. The hypotheses of the theorem precisely say that these numbers are non-zero in~$R$ and so the result follows. 
\end{proof}

\end{document}